\documentclass{amsart}
\usepackage[parfill]{parskip}
\usepackage{graphicx} 
\usepackage{color} 
\usepackage{hyperref} 
\usepackage{graphicx}  
\usepackage{verbatim}  
\hypersetup{pdfpagemode=FullScreen,colorlinks=true}
\usepackage[autostyle]{csquotes}
\usepackage{appendix, url}
\usepackage{amsmath, amssymb, mathrsfs, amsfonts, amsthm} 
\usepackage{tikz-cd}
\DeclareMathOperator{\card}{Card}

\DeclareMathOperator{\image}{image}

\newtheorem {theorem} {Theorem} [section]

\newtheorem{lemma}[theorem] {Lemma} 

\newtheorem {question}  [theorem] {Question}
\newtheorem {definition} [theorem] {Definition} 
\newtheorem {proposition}  [theorem]{Proposition}

\newtheorem {notation}[theorem] {Notation}
\newtheorem {terminology}[theorem] {Terminology}
\newtheorem {remark} [theorem] {Remark}
\numberwithin {equation} {section}

\DeclareMathOperator{\obj}{obj}

\begin{document}
\href{http://yashamon.github.io/web2/papers/fukayaI.pdf}{Direct link to author's version}
\author{Yasha Savelyev} 
\address{University of Colima, Bernal Díaz del
Castillo 340,
Col. Villas San Sebastian,
28045, Colima, Colima,
Mexico}
\email{yasha.savelyev@gmail.com}
\title [Global Fukaya category I]{Global Fukaya category I}
\begin {abstract} 
Let $Ham (M,\omega ) $ denote the Frechet Lie group of
Hamiltonian symplectomorphisms of a monotone
symplectic manifold $(M, \omega) $.  
Let $NFuk (M, \omega)$ be the $A
_{\infty} $-nerve of the Fukaya category $Fuk
(M, \omega)$, and let $(|\mathbb{S}|, NFuk (M, \omega))$ denote the $NFuk
(M, \omega)$ component of the ``space  of
$\infty$-categories'' 
$|\mathbb{S}| $. 
Using Floer-Fukaya theory for
a monotone $(M, \omega)$ we construct a natural up
to homotopy classifying map 
\begin{equation*}
   BHam (M, \omega) \to (|\mathbb{S}|, NFuk (M, \omega)).
\end{equation*}
This verifies one sense of a conjecture of
Teleman on existence of action of $Ham (M ,
\omega)$ on the Fukaya category of $(M, \omega )
$. This construction is very closely related to
the theory of the Seidel homomorphism and the
quantum characteristic classes of the author, and 
this map is intended to be the deepest expression of their underlying geometric theory.
In part II the above map is shown  to be nontrivial by an explicit calculation. In particular, we arrive at a new non-trivial
``quantum'' invariant of any smooth manifold,
which motives the statement of a kind of ``quantum'' Novikov conjecture.
\end {abstract}
\maketitle    
\tableofcontents
\section {Introduction}
Smooth fibrations over a Lorentz 4-manifold with fiber a Calabi-Yau 6-fold are a model for the physical background in string theory. This suggests that there may be some string theory linked mathematical invariants of such a fibration. Indeed, when the structure group of $M \hookrightarrow P \to X$
can be reduced to the group of Hamiltonian
symplectomorphisms of $M$, (with its $C ^{\infty}$
topology)  in which
case $P$ is called a \emph{Hamiltonian fibration}, 
there are a couple of basic invariants of such a fibration based on
Floer-Gromov-Witten theory. One such example is the Seidel representation
\cite{citeSeidelpi1ofsymplecticautomorphismgroupsandinvertiblesinquantumhomologyrings} and
the related quantum characteristic classes of the author
\cite{citeSavelyevQuantumcharacteristicclassesandtheHofermetric}. Related
invariants are also proposed by Hutchings
\cite{citeHutchingsFloerhomologyoffamilies.I}.
Even earlier there is work on
parametric Gromov-Witten invariants of Hamiltonian fibrations by Le-Ono
\cite{citeLeOnoParametrizedGromov-Witteninvariantsandtopologyofsymplectomorphismgroups} and Olga Buse \cite{citeOlgaBuse}.
At the same time, Costello's  theorem
\cite{citeCostelloTopologicalconformalfieldtheoriesandCalabi-Yaucategories} 
on reconstruction of topological conformal
field theories from Calabi-Yau $A _{\infty}$ categories suggests that the
above invariants must have a similar reconstruction principle.

For a given Hamiltonian fibration $P$ as above, the $A _{\infty} $ Fukaya
categories of the fibers fit into a ``family'', although exactly what this ``family'' should mean is a non-trivial problem by itself, since we must somehow remember the continuity of $P$. Then our basic idea is that associated to a Hamiltonian fibration there
should be a classifying map from $X$ into an appropriate ``classifying'' space of $A _{\infty} $ categories, from
which the other invariants can be reconstructed via a version of 
Toen's derived Morita theory, 
This can also be understood to say that $Ham (M ,
\omega) $ naturally (continuously) acts on $Fuk
(M)  $, verifying in one sense a conjecture of
Teleman.  We say more
on this in Section \ref{sectionHochschild}.  

 This paper will be mostly self-contained, as we will explain many (especially algebraic) concepts used.
\subsection {A functor from the category of smooth 
simplices of $X$ to $A _{\infty}$ categories} \label{sec:IntroAfunctor}
The first basic ingredient for our construction is
as follows.
Given $P$ as above, and a choice of
geometric-analysis theoretic perturbation data
$\mathcal{D}$, to each smooth simplex
$$\Sigma: \Delta ^{n} \to X,$$ we associate an $A
_{\infty} $ category $F _{P} (\Sigma)$.  This data
$\mathcal{D} $ 
involves certain compatible choices of Hamiltonian
connections and almost complex structures, similar
to the kind in the construction of the Seidel
morphism
~\cite{citeSeidelpi1ofsymplecticautomorphismgroupsandinvertiblesinquantumhomologyrings}. 
This will be discussed in Section \ref{section:dataD}.

A key geometric ingredient is the following. Let
${\mathcal{R}}_{d}$, $d \geq 2$,  denote
the  moduli space of Riemann surfaces which are topologically
disks with $d+1$ punctures on
the boundary. Let $\overline{\mathcal{R}}_{d}$
denote the standard compactification. We construct natural, axiomatically determined maps from the universal
curves \footnote{Technically from certain spaces
obtained from the universal curves.}  over $\overline{\mathcal{R}}_{d} $, for
each $d$, into the standard topological
simplices $\Delta ^{n} $.
 This topological-combinatorial connection of the universal curves with simplices is new, and is likely of independent interest.

Let $X _{\bullet} $ denote the smooth singular set
of $X$ and let $P$ be as above. Let $$\Delta^{} (X)
= \Delta /X _{\bullet}$$ denote the category of
simplices of $X _{\bullet }$, see
Section \ref{section:simplexCategory} for
the particulars.

The above data is then extended to a functor
\begin{equation*} \label{eqF}
   F _{P}: \Delta (X) \to A _{\infty}-Cat ^{unit},
\end{equation*}
with $A _{\infty}-Cat ^{unit}$
the category of small, unital, $\mathbb{Z} _{2} $-graded $A _{\infty} $ categories over $\mathbb{Q}$, with morphisms strict embeddings, which are moreover
quasi-equivalences.   

We had mentioned above the ``space of $A
_{\infty}$ categories''. However, technically it
will be simpler to work with a related space of $\infty$-categories we denote by
$|\mathbb{S}| $, discussed
in Appendix \ref{section:prelimQuasi}.  
Slightly more
explicitly, it is the geometric realization
of a certain Kan complex $\mathbb{S} $  whose
vertices are $\infty$-categories, and whose edges
are  equivalences of $\infty$-categories,  called
categorical equivalences.

The connection of the functor $F _{P}$ with $\mathbb{S} $  comes via the
nerve functor $$N: A _{\infty}-Cat ^{unit} \to
sSet,$$
with right-hand side the category of simplicial
sets.  The functor $N$ is an analogue for
$A _{\infty} $ categories 
of the classical nerve construction, which is due
to Grothendieck.  The $A _{\infty}$ version, first
suggested in Lurie~\cite{citeLurieHigherAlgebraa}, can
be considered to be a special case of the more
general nerve construction for simplicial
categories,  and was developed by Faonte
\cite{citeFaonteSimplicialNerve}.
See also, Tanaka
\cite{citeLeeAfunctorfromLagrangiancobordismstotheFukayacategory}.
 What will be crucial for us
is that
$N$ takes an $A _{\infty}$ category to a
$\infty$-category.

One basic reason that working with $\mathbb{S} $
is useful, is that it 
will allow us to convert all the algebraic data of the
functor $F$ above, to the data of a single
combinatorial-topological 
object, which we call the global Fukaya
category $Fuk _{\infty} (P)$, as appearing in the
title of the paper. More specifically,
$Fuk _{\infty} (P)$ has the structure of a
categorical fibration (an analogue for
$\infty$-categories of Serre fibrations):     
\begin{equation}
   \label{eq:Fukinfty}
   NFuk (M,  \omega) \to  Fuk
_{\infty} (P) \to X _{\bullet},
\end{equation}
described in Section \ref{sec:GlobalFukaya}.
 This will be crucial for computations in Part II
~\cite{citeSavelyevGlobalFukayacategoryII}.

Together with suitable invariance, under
deformation of the perturbation data $\mathcal{D}
$, the categorical fibration \eqref{eq:Fukinfty}
leads to the
following theorem. Denote the
connected component of an element $\mathcal{X} \in
\mathbb{S} (0) $ (corresponding to an
$\infty$-category)  by
$(\mathbb{S}, \mathcal{X}) $, cf. Definition
\ref{def:connectedComponent}. In what follows $\mathcal{X}
$  will be $NFuk (M,\omega) $.  

\begin{theorem} \label{thmInfinityUniversal}
For $(M,\omega)$ a monotone symplectic manifold, 
and $M \hookrightarrow P \to X$  a smooth Hamiltonian
   fibration over a smooth manifold $X$,
there is a natural up homotopy map
  $$cl _{P}: X  \to |(\mathbb{S}, NFuk
   (M, \omega))|,$$ with $|\cdot|$ still denoting
   geometric realization. Moreover, this extends to the
   universal level, so that there is a natural up
   to homotopy map  $$cl: BHam
(M,\omega) \to |(\mathbb{S}, NFuk (M, \omega))|, $$   
  corresponding to the universal Hamiltonian $M$-fibration
$p: E _{M} \to BHam (M,  \omega) $. This is further
   natural, so that $$[cl _{P}]
   = [cl] \circ [\widetilde{cl} _{P}],$$ for
   $\widetilde{cl} _{P}: X \to BHam (M,  \omega)
   $ the classifying map of the Hamiltonian
   fibration $P$, and $[\cdot] $  denoting the
   homotopy class.
\end{theorem}   
Natural up to homotopy just means that the map is
natural in the homotopy category of topological
spaces, with morphisms
homotopy classes of continuous maps. 
The name $cl _{P}$  comes from the fact that in a
certain sense $cl _{P}$ is classifying. In fact it
classifies the categorical fibration $Fuk _{\infty} (P) \to X 
_{\bullet}
$.   The
proof of this theorem is in Section \ref{sectionFukayaHochschild}.

This theorem can also be interpreted to say that $Ham (M,
\omega)$ ``continuously acts'' on $NFuk (M) $.
\begin{remark} If we work in the category of simplicial sets, the action can be understood as follows. Suppose we have a simplicial action of $Ham (M, \omega) _{\bullet }
$ on $NFuk (M, \omega )$. Then we have an induced simplicial map $$BHam (M, \omega) _{\bullet} \to BAut (N Fuk (M,
\omega)), $$ with the group of simplicial automorphisms $Aut (N Fuk (M,  \omega) ) $ interpreted as a simplicial group,  and where $BG $  denotes the simplicial  nerve
of a simplicial group $G$. And it's easy to see that there is a natural map $$BAut (N Fuk (M, \omega) )  \to (\mathbb{S}, N Fuk (M, \omega)), $$ by the construction of $\mathbb{S} $. 
Now starting with  a simplicial map  
\begin{equation}
	\label{eq:BHam}
 BHam (M, \omega) _{\bullet } \to
	(\mathbb{S}, N Fuk (M,
	\omega))
\end{equation}
we do not generally get an induced homomorphism
$$Ham (M,  \omega) _{\bullet} \to Aut (NFuk (M, \omega)
), $$ naturally. But if we  take the based
loop space of both sides of \eqref{eq:BHam}
we get something   approaching this
homomorphism, since $\Omega B Ham (M,
\omega) _{\bullet} \simeq Ham (M,  \omega)
_{\bullet} $ (simplicial homotopy equivalence).  And since
the simplicial $H$-space $$\Omega (\mathbb{S}, NFuk (M,
\omega) ) $$ in a sense ``extends'' the 
group $$Aut (NFuk (M, \omega)),$$   (loops in
$\mathbb{S} $ based at $N Fuk (M,  \omega) $
are in correspondence with categorical
self-equivalences of $NFuk (M, \omega ) $, which generalize simplicial automorphisms, cf. Appendix \ref{appendix.quasi}.)  
\end{remark}
  
The continuous action of $Ham (M,  \omega) $ on $NFuk (M, \omega )$  is one interpretation of the existence of a
``continuous action'' of $Ham (M, \omega) $ on $Fuk
(M, \omega) $.  
And this verifies one sense of a conjecture of Teleman ICM 2014 on the existence of such an action.
A kind of discrete version of such an action can
be found in
Seidel~\cite[Section
10c]{citeSeidelFukayacategoriesandPicard-Lefschetztheory}.
Another interpretation of this ``continous
action'', in exact setting, appears in the work
of Oh-Tanaka ~\cite{citeOhTanakaCoherentActions}. There the
functor $F _{P}$ above (or a close relative)  is converted to a map of spaces by means of localizations of categories. This provides an
alternative algebraic topological
perspective.

In part II
\cite{citeSavelyevGlobalFukayacategoryII} of this
paper, this map is shown to be homotopically non-trivial in a specific example,
and some possibly unexpected geometric applications of this are
developed.
\subsection {Towards new invariants and quantum Novikov conjecture} By the above discussion we automatically obtain a
new  invariant of a Hamiltonian fibration $M \hookrightarrow P \to X$ as the
homotopy class of the classifying map
$cl _{P}: X \to |\mathbb{S}|$.  

It may be difficult to get intrinsic motivation for Hamiltonian fibrations for a
reader outside of symplectic geometry, as a start one may read
\cite{citeGuilleminLermanEtAlSymplecticfibrationsandmultiplicitydiagrams}.
However, as
one particular case we can fiberwise projectivize the complexified tangent bundle:
 $$P (X) = P (TX \otimes \mathbb{C}),$$ of a
 smooth manifold $X$. This $P (X)$ in particular has the structure of a
smooth Hamiltonian fibration with fiber
$\mathbb{CP} ^{r-1} $ for $r$ the real dimension
of $X$. In this way we also
 get a new invariant of a smooth $r$ manifold $X$,
 given by the homotopy class of the classifying
 map $$cl _{P (X)}: X \to (\mathbb{S}, NFuk (\mathbb{CP}
 ^{r-1} )),
 $$ 
induced by Theorem \ref{thmInfinityUniversal}.

Recall that Pontryagin classes of a smooth manifold are defined as Chern classes of its complexified tangent bundle. Novikov has shown that rational Pontryagin
classes are
topologically invariant. 
It is then very natural to ask the
 following, ``quantum'' variant of the Novikov conjecture:
 \begin{question} Suppose that $f: X \to Y$ is a homeomorphism of smooth
manifolds. Is $cl _{P(X)} $ homotopic to $cl _{P (Y) } \circ f$? 
 \end{question}  
I suspect that the answer is yes,  simply because
the whole construction involves a kind of integration
theory, not fantastically far removed from
Chern-Pontryagin theory, (if we understand
``Gromov-Witten counts'' as integration). But this would lead to
further intriguing questions. For example: how would the resulting invariants be related to more classical topological invariants of smooth manifolds?   

The answer of ``no'' is possibly even more interesting,
since it means that our construction gives new
smooth invariants of manifolds via
holomorphic curves in symplectic geometry.
\subsection {Hochschild and geometric Hochschild cohomology and homotopy groups
of $Ham (M, \omega)$} 
\label{sectionHochschild}
This section is an excursion, meant to relate our geometric theory with the algebraic derived Morita theory of Toen.
For an $A _{\infty} $ category $C$ we define $$HH _{geom} ^{2-i} (C) = \pi _{i}
(\mathbb{S}, NC), \quad i>2.
$$ The left-hand side is named geometric Hochschild cohomology, the name and notation will be justified shortly. By Theorem \ref{thmInfinityUniversal} above we then get:
 \begin{theorem} \label{thmGroupHomoIntro} For $ (M,\omega)$ monotone,  there is a natural group homomorphism 
\begin{equation} \label{eqClassifying1}
   \pi _{i-1} (Ham (M, \omega), id)  \to
   HH ^{2-i} _{geom}  (Fuk (M, \omega)), \quad i > 2.     
\end{equation}
\end{theorem}
$HH ^{*} (Fuk (M, \omega )) $ is known to be
isomorphic to $QH ^{*} (M) $ in some cases, for
example in the monotone setting, relevant to us
here, this is due to
Sheridan~\cite{citeNickSheridanOntheFukayaCategory}. 
And so the above morphism, when $i>2$, has the
same formal form as (a special case of) the author's quantum characteristic classes
\cite{citeSavelyevQuantumcharacteristicclassesandtheHofermetric}, taking the form of  homomorphisms:
$$\Psi: (\pi _{k} (\Omega Ham (M,  \omega), id)
\simeq \pi _{k+1} (Ham (M,  \omega), id))  \to QH _{2n+k}
(M,\omega), \quad $$ 
where $2n=\dim M$. 
This is provided there is a connection between $HH ^{*} (Fuk (M)) $ and $HH ^{*} _{geom}
(Fuk (M))$. Such a connection is described
further below. 
 This would be the most basic form of the
``reconstruction'' that was mentioned in the
first paragraph of the paper.

In Part II we calculate  with Hamiltonian $S ^{2} $ fibrations over $S ^{4} $ to
get:
\begin{theorem}  \label{thm:pi4} The map
   $$\left(\pi _{3} (Ham (S ^{2},  \omega), id) =
   \mathbb{Z} \right) \to \left(HH _{geom} ^{-2} (Fuk (S ^{2} )) = \pi _{4} 
   (\mathbb{S}, NFuk (S ^{2} )) \right),$$
determined by \eqref{eqClassifying1} is an injection.
\end{theorem}
This has some possibly surprising consequences,
particularly for the theory of singular
connections.
\subsubsection {Geometric Hochschild cohomology and  Toen's derived Morita theory
 }
 \label{section:Toen}
A small disclaimer. $HH _{geom} ^{*} (C)  $ is just a name for an object 
whose construction is immediate from work of Joyal
and Lurie, and quite possibly
appears elsewhere. We claim no originality for this construction.
What may however be
interesting is the connection to symplectic geometry that we discover
in these papers.

Let us then very briefly indicate the connection of $HH _{geom} ^{*} (C)  $ with Hochschild
cohomology via Toen's derived Morita theory. Let
$dg-Cat$  denote the category of differential
graded categories, a.k.a. dg categories with
morphisms quasi-equivalences. 
\begin{theorem} [Corollary 8.4 \cite{citeToenThehomotopytheoryofdg-categoriesandderivedMoritatheory}] \label{thmToen}
For a small dg category $C$, (with cohomological grading conventions) there are natural isomorphisms
\begin{align} \label {eq.toen1} \pi
_{i} (|dg-Cat|, C) \simeq HH ^{2-i} (C), \text{ for }i >2, \\
\label {eq.toen2} \pi _{2}  (|dg-Cat|, C)
\simeq HH ^{0} (C) ^{\times},
\end{align}
with $HH ^{0} (C) ^{\times} $ denoting  the multiplicative group of invertible
elements, and with $|dg-Cat|$ denoting the  geometric
realization of the nerve of $dg-Cat$, a.k.a. the
classifying space. Here $C \in |dg-Cat|$ is the element
   corresponding to $C \in dg-Cat$.
\end{theorem} 

On the other hand the nerve functor $N$ naturally induces a homomorphism, 
\begin{equation*}
   N _{*} : \pi
     _{i} (|dg-Cat|, C)| \to \pi
     _{i} (|\infty-\mathcal{C}at|, NC) \simeq \pi
   _{i} (|\mathbb{S}|, NC).
\end{equation*}
When $C$ is a $\mathbb{Z}
$-graded, (pre)-triangulated dg category over
$\mathbb{Q} $ there are folklore theorems of Lurie
(personal communication) to the
effect that this is an isomorphism.

Thus, in this case, for $i>2$ $$HH ^{2-i} (C) = \pi _{i}
(\mathbb{S}, NC) = HH _{geom} ^{2-i} (C),    $$ by our definition.  This
extends to $\mathbb{Z}$-graded rational (pre)-triangulated $A _{\infty} $
categories, along
the lines of Faonte~\cite{citeFaonteFunctors}.
\begin{remark}
   \label{remark:triangule}
As I understand, these hypotheses apply to at least
monotone symplectic manifolds if we
pre-triangulate the Fukaya categories.  It is important to note however that we do not
pre-triangulate the Fukaya categories in the
main construction of the paper, it should be
possible to do that, following the same
ideas,   but this possibly loses
information, and it may make the computation
   in Part II ~\cite{citeSavelyevGlobalFukayacategoryII} more difficult. Without pre-triangulating the
connection of $HH ^{*} _{geom}$ and $HH ^{*}$
appears to be more complicated.  It is also worth
noting that even if we did identify  $HH ^{*}
_{geom}$ and $HH ^{*}$ then there is still  a
hard geometric problem of identifying the
actual morphisms - the quantum characteristic
   classes/Seidel morphism and the morphisms from the data of
   $F$. So all in all the reconstruction 
is still an open problem.   In
   Oh-Tanaka~\cite{citeOhTanakaCoherentActions} a
different approach is taken. Starting with the
functor   $F$,  or a close  cousin, the authors
use categorical techniques of localization, which
allows to avoid introduction of the space
   $\mathbb{S} $. 
However, the above problem of identifying the
morphisms remains.
\end{remark}

\begin{remark}
In the case of $i=2$, by Theorem
   \ref{thmInfinityUniversal}, we have a homomorphism 
\begin{equation*}
   \pi _{1} (Ham (M,  \omega), id) \to \pi _{2} (\mathbb{S},
   N Fuk (M,  \omega) ).    
\end{equation*}
So again, if we could again identify $\pi _{2} (\mathbb{S}, N Fuk (M,  \omega) )$ with $(HH ^{0} (Fuk (M,  \omega) )) ^{\times} $ and the latter with $QH _{2n} (M) $, then we would, a priori only in form, recover the Seidel homomorphism 
\end{remark}
\subsection {Organization}   Section 3 is
concerned with preliminaries. The crucial
construction of the system of maps from the
universal curves to $\Delta^{n} $ is in Section 4.
Perturbation data $\mathcal{D} $  is constructed
in Section 5. The main functor $F$ is constructed
in Section 6. Finally, the global Fukaya category
in constructed in Sections 7,8. Section 7
contains the proofs of the main Theorems
\ref{thmInfinityUniversal},
\ref{thmGroupHomoIntro}.
\subsection{Acknowledgements}
I would like to thank Octav Cornea, Egor Shelukhin, and
Kaoru Ono for discussions and support,  as well as Kevin Costello and Paul Seidel for interest. Hiro Lee Tanaka for enthusiasm,  generously providing me with an early draft of his thesis and finding a number of misprints in a draft of the paper. Bertrand Toen for explaining to me an outline of the proof of some conjectures and for enthusiastic response. 
I also thank Jacob Lurie, for feedback on some
questions. Special thanks to the referees for
much help in the shaping of this paper.
The paper was primarily conceived while I was a CRM-ISM
postdoctoral fellow, and I am grateful for the wonderful
research atmosphere provided by CRM-Montreal. I am also
grateful for the hospitality of RIMS center at Kyoto university and ICMAT Madrid where parts of the paper were expanded.  
\section {Notations and conventions and large categories}
We use diagrammatic order for composition of morphisms in the Fukaya
category, and in $\infty$-categories  so $f \circ g$ means $$\cdot \xrightarrow{f}
\cdot \xrightarrow{g} \cdot,$$ as reversing order
for composition in $\infty$-categories is geometrically very confusing, since morphisms are identified with edges of
simplices. Elsewhere, we use the more common
Leibnitz functional convention. Although this is somewhat contradictory in practice things should be clear from context. By simplex and notation $\Delta ^{n} $ we will interchangeably mean the topological $n$-simplex and the standard representable $n$-simplex
as a simplicial set, for the latter we may also write $ \Delta ^{n} _{\bullet}
$.  

Given a category $C $ the over-category of  an object $c \in C$ is denoted by
$C/c$. We say that a morphism in $C$ is \emph{over $ c $} exactly if it is a morphism
in the over-category of $c$.  

Given an $A _{\infty} $ category by the nerve we always mean the $A _{\infty}
$ nerve $N$, as previously described. 

Some of our $\infty$-categories  are
``large'' with proper classes of simplices instead of sets. The standard formal treatment of this
is to work with Grothendieck universes. 
We shall not however make this explicit. 
\section{Preliminaries} \label{section:construction} 
\subsection {The simplex category of a smooth
manifold $X$}  \label{section:simplexCategory}
Let $\Delta$ denote the category of combinatorial
simplices, whose objects are totally ordered
finite sets $[n] = \{0, \ldots,
n\} $, with $hom _{\Delta } ([n], [m] ) $ the set
of non-strictly increasing maps   
\begin{equation*} \{0, \ldots, n\} \to \{0, \ldots,  m\}.
\end{equation*} 
A simplicial set $S _{\bullet}$ is a functor $ {S}
_{\bullet}: \Delta  \to Set ^{op}$.  We will
usually write $S _{\bullet } (n) $   instead of 
 $S
_{\bullet } ([n]) $, and this is called the \emph{set of
$n$-simplices} of $S _{\bullet }$. 

A map of simplicial set $f: A _{\bullet } \to S
_{\bullet }$ is a natural transformation of the
corresponding functors. 

 Let $\Delta ^{n} _{\bullet}$  denote the
simplicial set $\Delta ^{n}
 _{\bullet} = hom _{\Delta} (\cdot, [n])$. Then we
have the \emph{category of simplices}  over $S _{\bullet}$, $\Delta/S _{\bullet}$, whose set of objects is the set of natural transformations
$Nat (\Delta ^{n} _{\bullet}, S _{\bullet})$ and morphisms commutative diagrams
$$
\begin{tikzcd}
 \Delta ^{n} _ {\bullet} \ar [rd] \ar[r] & \Delta ^{m} _{\bullet}
 \ar[d] \\ & S _{\bullet},  \\
 \end{tikzcd}
$$
 s.t. the natural transformations $\Delta ^{n} _{\bullet} \to \Delta ^{m} _{\bullet}$
 are induced by maps $[n] \to [m]$. To simplify
 notation we
 rename:
\begin{equation*}
   \Delta  (S _{\bullet}):= \Delta/S _{\bullet}.
\end{equation*}

Let $\Delta ^{n}$ denote the standard
topological $n$-simplex, 
i.e.  $$\Delta^{n} := \{(x _{1}, \ldots, x _{n})
\in \mathbb{R} ^{n} \,|\, x _{1} + \ldots + x _{n} \leq 1,
\text{ and } \forall i:  x _{i} \geq 0  \}. $$ The
vertices of $\Delta^{n} $ are assumed ordered in
the standard way $0, \ldots, n$.  
Let $X$ be a smooth manifold.   We say that
$\Sigma: \Delta ^{n} \to X$ is a \textbf{\emph{smooth map}}
if it has an extension  $V \subset \mathbb{R} ^{d}
\to X$, for $V \supset \Delta ^{n} $ some open set.  
\begin{definition}\label{def:}
 We say that a smooth 
map $\Sigma: \Delta^{n}  \to X$  is
\textbf{\emph{collared}} if there is a
neighborhood $U \supset \partial \Delta^{n} $  in
$\Delta^{n} $, such that $\Sigma| _{U} =
\Sigma \circ ret$ for $ret: U \to \partial
\Delta^{n}  $  some smooth retraction.   Here
smooth means that $ret$ has an extension to a
   smooth
 map $V \subset \mathbb{R} ^{d} \to \mathbb{R}
   ^{d}  $, with $V \supset \partial \Delta^{n} $
   open in $\mathbb{R} ^{d} $.

\end{definition}
For $X$ a smooth manifold, define a simplicial set $X _{\bullet }$  by: $$X _{\bullet} (n) := C
^{\infty} _{col} (\Delta ^{n}, X),$$  with the right-hand side the set of
all smooth collared maps
$\Delta ^{n} \to X$.   It is easy to see that $X _{\bullet} $ is a Kan complex.    The same
surely holds without the collared condition but
the proof is more difficult. \footnote {A
reference is not known to me.}   For simplicity, we
will work with collared simplices throughout and
this may no
longer be mentioned.

In this case the simplex
category $$\Delta (X) :=\Delta^{}  (X _{\bullet })
$$   can be 
elaborated as follows. It is the category with
objects smooth, collared maps $\Sigma: \Delta ^{n}
\to X $. A morphism $f$  from $\Sigma _{1}$ to
$\Sigma _{2}$ is a commutative diagram
\begin{equation} \label{eq:morphismovercategory} 
\begin{tikzcd}
\Delta ^{n} \ar [rd, "\Sigma _{1}" ] \ar[r,
   "f"]& \Delta ^{m} \ar[d, "\Sigma _{2}"] \\
& X,
\end{tikzcd}
\end{equation}
 and top horizontal arrow a
\textbf{\emph{simplicial map}}, also denoted
$f$, that
 is an affine map taking vertices to vertices preserving the order.  We say that $\Sigma: \Delta ^{n} \to X$ is \emph{non-degenerate} 
if it does not fit into a commutative diagram
 \begin{equation*} 
\begin{tikzcd}
\Delta ^{n} \ar [rd, "\Sigma"]  \ar[r] & \Delta ^{m} \ar[d] \\
& X, 
\end{tikzcd}
\end{equation*}
with $m < n$. 

We will denote by $Simp (X)$ the full subcategory
of $\Delta (X)  $, consisting of its non-degenerate objects. The significance   of
$Simp (X) $   is that the perturbation data in the
construction of $F$ (as in Section
\ref{sec:IntroAfunctor}) of the introduction,
must first be constructed in the context of $Simp
(X) $, and then formally extended to all
simplices.  This is necessary to insure
functoriality of $F$ on $\Delta (X) $.

\subsection{Preliminaries on Riemann surfaces}
\label{sec:preliminariesRiemannSurfaces}
Much of this material is adopted from
the book of
Seidel~\cite{citeSeidelFukayacategoriesandPicard-Lefschetztheory}.
Although there are some notation changes, to fit better
with our goals.  Some other notions like the
linear ordering, appearing further on, might be
new, at least in present type of context. 

Let $S'$ be a nodal, connected, simply
connected, Riemann surface, with each
smooth component topologically a disk $D ^{2}$  with some
marked points on the
boundary, indexed by a finite set $I$.   
 Removing the marked points we
obtain a surface $S$ with ends alternatively
called \emph{punctures}. However, it is
sometimes simpler to represent $S$ as the original
compact surface $S'$ with marked points.
The ends/marked points are labeled by $\{e _{i} \} _{i \in I} $. 
The nodal points of $S'$ are denoted by $\{n
_{j} \} _{j \in J} $, again for some index set
$J$, and these are distinct from the set of marked
points $\{e _{i} \} _{i \in I} $. 

For each $j \in J$, we have a
pair $S _{j, \pm} $ of smooth components of $S$, that are topologically 
disks with punctures $$\{e _{i} \} _{i \in I _{j, \pm}}  \subset
 \{e _{i} \} _{i \in I},
 $$ we explain the signs $\pm$ shortly, for now
 they just distinguish the pair of components 
 $S _{j,+}$  and $S _{j,-}$.
More explicitly, $I _{j,+}$  respectively $I _{j,-}$  are just
the subsets of $I$ corresponding to the
punctures on the components $S _{j,+}$
respectively $S _{j,-}$.
If we remove the node $n _{j}$  from $S$ then  $$S _{j, \pm} ^{\circ} := S _{j, \pm} - n _{j}   $$ has an
additional puncture $n _{j, \pm} $ called the
\textbf{\emph{node end}}.

We distinguish one end of $S$ as the root,
to be denoted as $e _{0}$. Using the clockwise
orientation of the boundary of $S$, and if
$\card (I) = d$  we then have
an induced ordering, $e _{0}, \ldots, e _{d}$  of
the punctures.

It is sometimes convenient to depict such Riemann surfaces as rooted
semi-infinite trees, embedded in the plane.
We do this by assigning a vertex
to each smooth component as above, a half infinite edge to each marked point, and an edge to each nodal point, as depicted in Figure  \ref{bubblestotrees}. 
\begin{figure}[h]
 \includegraphics[width=3in]{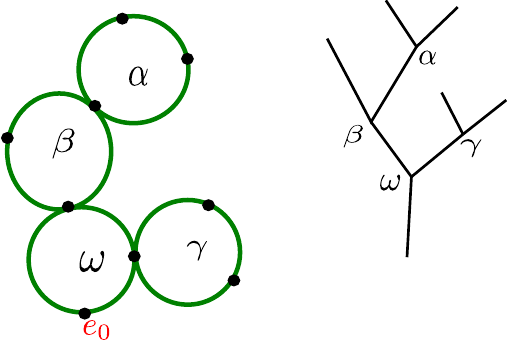}
 \caption {} \label{bubblestotrees}
\end{figure} 
We say that $S$ is \textbf{\emph{stable}} if 
for the associated tree the valency of each vertex is at least 3. 

To make some arguments and notation  cleaner, we also introduce a linear ordering on the smooth components of $S$, or
vertices, by ``order of
operation'' defined as follows.
The component with the root semi-infinite edge $e _{0} $ will be called the root vertex denoted by $\omega$. In terms of the associated tree for the surface we have
a pre-order on vertices given  by the distance to the root vertex, (by giving each edge length
 1). To get an actual order, first
isometrically embed the tree in the plane, while
preserving the clockwise ordering of each
half-infinite edge, corresponding
to the ordering of the punctures.
Then clockwise order vertices equidistant to the root, as in Figure \ref{bubblestotrees}.
We shall denote by ${\alpha}$ the furthermost
component from $\omega$, i.e. it is the greatest
element with respect to our order.  Then ${
   \beta}$ is the next furthermost component, etc. (Pretending that we can't run out of letters.) 
Note, that $\alpha$ may
not be the leftmost component, in fact
``leftmost''
may be ambiguous (dependent on the embedding) for vertices not equidistant to
$\omega$. 
\begin{remark} \label{remark:}
This correspondence of letters to the order 
may seem counterintuitive, but this is motivated
   by idea that these trees are operadic trees
   determining composition. More explicitly, later
   on this is the composition
   in certain $A _{\infty}$ Fukaya categories. Composition
   corresponding to furthermost elements from
   $\omega$   is performed first. Hence $\alpha$
   corresponds to the first operation we need to perform.
   Although the operations corresponding
   to components equidistant
   from $\omega$ can be performed in any order. 
\end{remark}
As part of the data,  we may ask for a holomorphic
diffeomorphism at each $i$'th end, having the name of the end: $$e _{i}:  [0,1] \times (0,
  \infty) \to S,$$ $i \neq 0$. And at the 0'th puncture we ask for a holomorphic diffeomorphism $$e _{0}: [0,1] \times (-\infty, 0) \to S.
  $$  
These charts will be called
  {\emph{strip end charts}}. 

When $S'$ is not nodal, these strip end charts have the property that $$S - ends:= S - \cup
_{i \in \{0, \ldots, d\}} \image (e _{i}) $$ is a compact surface with
corners. 

Let $S _{j, \pm} ^{\circ}, S _{j, \pm}   $ be as above. 
We further specify the $\pm$
 distinction so that  $S _{j,-} > S _{j, + }  $ with respect to the linear order above.
  And we may ask for a similar pair of  strip
charts \begin{align}
   &  e ^{}
  _{j,-}: [0,1] \times (-\infty, 0) \to S ^{\circ}  _{j,-}, \\ 
   &  e ^{} _{j,+}: [0,1] \times (0,
  \infty) \to S ^{\circ}  _{j,+}
 \end{align}  
 at the $n _{j, {\pm} } $ ends. The data of all such strip charts for a given $S$, will be called a \emph{a strip end
structure}.  

The moduli space of the Riemann surfaces as above,
with $\card I=d$, will be denoted by $\overline{\mathcal {R}} _{d}$. 
(Note that Seidel \cite{citeSeidelFukayacategoriesandPicard-Lefschetztheory} calls our $ \overline{\mathcal
  {R}} _{d}$ by $ \overline{\mathcal {R}} _{d+1}$.) 
$\overline{\mathcal {R}} _{d}$ is a real dimension $d-2$  manifold with
corners. We will also denote by $\mathcal{R} _{d} \subset
\overline{ \mathcal{R}} $  the
subspace corresponding to non-nodal surfaces.

For $d \geq 2$ let $\rho': {\mathcal {S}} _{d}
  \to \overline{\mathcal {R}} _{d}$ denote the universal family of the Riemann surfaces
  $S$, as above. 
Denote by $$\rho: {\mathcal {S}} ^{\circ}  _{d}
  \to \overline{\mathcal {R}} _{d},$$ this universal family where the nodal
  points of the surface fibers have been removed.
    \begin{notation}
We denote by $ \mathcal{S} _{d,r} $ and sometimes just by $\mathcal{S}_{r} $ the fiber $\rho ^{-1} (r) $, for $r \in \overline{\mathcal{R}}_{d}$. 
\end{notation}

Choose $r$-smooth (varying smoothly with respect
to  $r$)  families $ \{e _{i,r} \}$, $\{e _{j,\pm,
r} \}$ of  strip end structures for the entire
  universal family $ {\mathcal {S}}_d \to \overline{ \mathcal {R}}
  _{d}$, (note that further on $r$ is suppressed). These choices have to be  consistent with gluing in the natural sense  as explained in \cite[Section 9g]{citeSeidelFukayacategoriesandPicard-Lefschetztheory}. 
We will keep track of these systems of choices of strip end structures only implicitly.
  
\subsubsection{Metric characterization of the moduli
space} 
\label{sec:Metric characterization of the moduli space}  It will be helpful to recall the 
characterization of the moduli space $ {\mathcal
{R}} _{d}$ and its compactification  $
\overline{\mathcal {R}} _{d}$,  
 in terms of hyperbolic metrics   on the punctured
disks $S _{r}$. 
The family $ \{{\mathcal {S}} _{d,r} \}
_{r \in   {\mathcal{R} } _{d}}$ is in a
bijective correspondence with a  suitably
universal family $\{\mathcal {M}et _{r}\} = \{\mathcal {M}et _{d,r } \}$ of constant curvature $-1$
metrics on the disk with $d+1$ punctures on the boundary. Under this correspondence the complex structure on $ \mathcal {S} _{r}$ is just the conformal
structure induced by $ \mathcal {M}et _{r}$. This is of course classical, to see
all this use Schwartz reflection to ``double''
each $\mathcal {S} _{r}$ to a connected genus 0 Riemann surface $\mathcal {D}_{r}$, without boundary. This determines an embedding of $
{\mathcal {R}} _{d}$ into the moduli space $ {M}
_{0,d+1}$ of Riemann surfaces that are
topologically $S ^{2}$  with $d+1$ points removed.  As $d>2$,  by
the uniformization theorem, each $\mathcal {D} _{r}$ is a quotient of the disk by a subgroup of $PSL (2, \mathbb{R})$, which must also preserve the
hyperbolic metric. Therefore, $ \mathcal {S} _{r}$ inherits a hyperbolic
metric, that we call $Met _{r}$. 

The metric point of view gives an illuminating description of the
compactification  $ \overline{\mathcal{R}}
_{d}$, for $d \geq 2$. Starting with some $ \mathcal {S} _{r}$ and taking
$r$ to a boundary stratum, corresponds to 
some fixed collection of embedded,  disjoint
geodesics on $ \mathcal {S} _{r}$, with boundary
in the boundary of $\mathcal{S} _{r} $,   have
their length shrunk to zero. Each boundary stratum of
$\overline{\mathcal{R} } _{d}$ is completely determined by such a collection of
geodesics.
\subsubsection{Gluing} 
\label{sec:Gluing}
The gluing construction
(see for example
\cite{citeSeidelFukayacategoriesandPicard-Lefschetztheory})
takes a surface $\mathcal{S} _{r} $, $
r \in \partial \overline {\mathcal
{R}} _{d}$ and produces a surface with one less node. This gluing  is determined by gluing parameters which we parametrize by $ [0,1)$, assigned to each node. For us $0$ means don't glue, and 1 is meant to correspond to some small value of the gluing parameter used in actual gluing. We will write $d _{\alpha, \beta} $ for the parameter used in the gluing of components $\alpha, \beta $, and likewise with other components.

The gluing construction for parameters in $ [0,1)$ determines an open
neighborhood called \textbf{\emph{gluing neighborhood}}
 of the boundary of $ \overline {\mathcal
{R}} _{d}$.
 A \textbf{\emph{gluing normal neighborhood}} of the boundary of $\overline {\mathcal
{R}} _{d}$: will be an open neighborhood (usually
denoted $N$) of the boundary, deformation retracting to the boundary, contained in the gluing neighborhood.

The gluing construction also induces a kind of thick-thin
decomposition of the surface, with thin parts conformally identified with $
[0,1] \times (0, l)$ for $l>0$     determined by the corresponding gluing
parameter.  This decomposition is not intrinsic,
as it depends in particular on the choice of the
family of strip end structures.
However, instructively these gluing parameters can be
thought of as lengths of geodesic segments, for example $m _{\alpha}$, $m
_{\beta}$ in figure \ref{figuregeodsegm}, and the thin parts are closely related to thin parts of thick-thin decomposition in hyperbolic geometry.
\def\svgwidth{2in}
\begin{figure} 
\centering 
\includegraphics[width=1.5in]{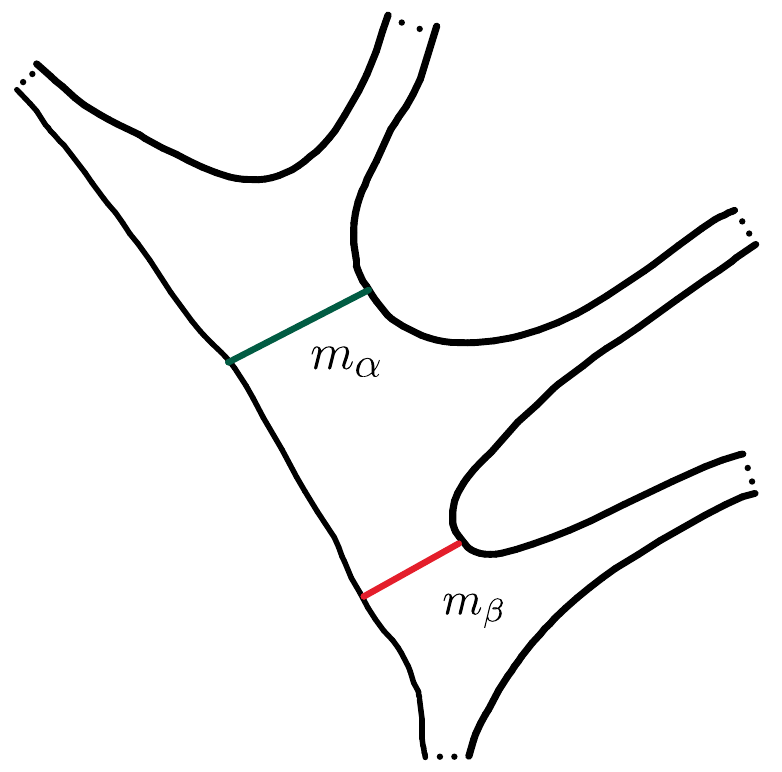}
\caption {This diagram is only schematic. The embedding into the plane is not meant to be holomorphic or isometric for the natural hyperbolic structure on
the surface.}
 \label {figuregeodsegm}
\end{figure} 
\subsection{Hamiltonian fibrations} 
 \label{sec:HamiltonianBundles}  
A \textbf{\emph{Hamiltonian fibration}}  is a smooth fiber bundle $$M \hookrightarrow P \to X,$$
with structure group $\mathcal{G} = \text {Ham}(M, \omega) $
with its $C ^{\infty} $ Frechet topology. A
\textbf{\emph{Hamiltonian connection}}, in this
paper usually denoted by caligraphic letters of
the kind $\mathcal{A} $,  is just
an Ehresmann $\mathcal{G} $-invariant connection for a Hamiltonian
fibration. When $P = M \times X$  is trivial, such
$\mathcal{A} $ 
may be specified as a one form valued in $C
^{\infty} _{norm} (M) $ - the space of  smooth
functions on $M$ with mean 0. 
\section {A system of natural maps from the
universal curve to $\Delta ^{n} $}
\label{sec:systemOfmaps}
We explain here a remarkable connection between the universal curve over
$\overline{\mathcal{R}} _{d}  $ and the standard
topological simplices $\Delta
^{n} $. This will be used in our construction but may be of independent
interest.

Let $\Pi (\Delta ^{n}) $ be the small groupoid, whose
objects set $obj$ is the set of vertices
of $\Delta ^{n}$. The morphisms set $hom$ is
the set of affine maps $m: [0,1] \to \Delta ^{n}$, 
(possibly constant) sending end points to the
vertices. The source map $$s: hom \to obj$$ is defined by
$s (m) = m (0)  $ and the target map $$t: hom \to obj$$ is
defined by $t(m) = m (1)$.  Thus, the set $hom
(v_1, v _{2}) 
  $  of
morphisms between $v_1, v_2 \in \Pi (\Delta ^{n})
$ consists of a single morphism.  It is the unique
affine map $m:[0,1] \to \Delta^{n}  $, satisfying
$m (0)  = v_1$  and $m (1) = v _{2}$. 
Consequently, the composition maps in $\Pi (\Delta
^{n}) $ are forced, by the above uniqueness.   And
the identity at $v$ is the constant $m: [0,1] \to
\Delta^{n} $, $m (0) = m (1) =v$.
\begin{notation}
   \label{} We denote the composition of
   $m _{1}, m _{2}  $ by $m _{1} \cdot m _{2}  $.
   The order is diagrammatic, so that this means
   the composition
 \begin{equation*}
      \cdot \xrightarrow{m _{1}} \cdot
    \xrightarrow{m _{2}} \cdot.
   \end{equation*}
     \end{notation}
We say that $(m_1, \ldots , m _{d} ) $ is a
\emph{composable chain} of morphisms $m_i \in
hom(\Pi (\Delta ^{n} ))$ if $t(m _{i-1}) = s (m_i)
$, for all $2 \leq i \leq d$ .  For future use let $m _{i-1,i} $ denote the unique morphism from the $i-1$ vertex to the $i$ vertex in $\Delta^{n} $.

The goal is to construct a  ``natural'' system of
maps $${u} (m _{1}, \ldots, m _{d}, n):
{\mathcal {S}} ^{\circ}  _{d}  \to \Delta ^{n},$$ $d \geq 2$, for each composable chain $(m_1, \ldots ,m _{d} )$. 
To give a preview, the reason why we work with
$\mathcal{S} ^{\circ }_{d}$ instead of
$\mathcal{S} _{d}$, 
is that certain naturality conditions that we impose,    force
that the node ends (images of charts $e _{j, \pm }$ )  of the surface map  to an edge of
$\Delta ^{n}$. Given continuity, this is of course impossible if the edge is non-constant, unless the node itself is removed.

Let us order the boundary components $0, \ldots, d$ of a surface
$\mathcal{S} _{r}$, $r \in \mathcal{R} _{d}$, as
follows.  The ordering is
clockwise with 0 the component on the
left of the $e _{0}$ end, given that we have chosen the
embedding with $e _{0}$ pointing downward.   See
Figure \ref{fig:sidenumberedsurface}. 
\begin{figure}[h]
  \includegraphics[width=2in]{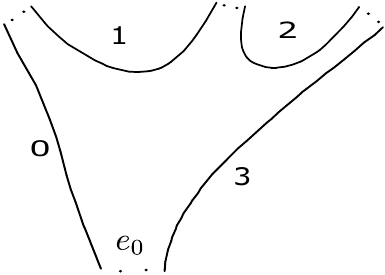}
 \caption {} \label{fig:sidenumberedsurface}
\end{figure}

\begin{definition}\label{def:partialnaturality}
We say that ${u} (m _{1}, \ldots, m _{d}, n):
{\mathcal {S}} ^{\circ}  _{d}  \to \Delta ^{n}$ 
   satisfies \textbf{\emph{partial naturality}}   if
   the following holds.
\begin{enumerate} 
   \item \label{axiom:partial1} The map ${u} (m _{1}, \ldots, m _{d},
      n)$ is continuous and its restriction to
      $$\rho ^{-1} (\mathcal{R} _{d}  \subset
      \overline {\mathcal{R}}_{d})  $$ is smooth.
\item \label{axiom:partial2} Let $ 
{u} (m _{1}, \ldots, m _{d}, n,r)$ denote the restriction of  ${u} (m _{1}, \ldots, m _{d},
n)$ to $ \mathcal {S} _{r}$, and let $$m_0:=m_1 \cdot
\ldots \cdot m_d$$ be the composition in $\Pi (\Delta ^{n}) $.
Then given the strip end charts $e _{k}:
[0,1] \times (0, \infty) \to \mathcal {S}
_{r}$ at each $e _{k}$ end,  $1 \leq k \leq d$, 
$$e _{k} \circ {u} (m _{1}, \ldots, m _{d}, n,r)=
      m _{k} \circ pr$$ for $pr: [0,1] \times (0, \infty) \to
      [0,1]  $  the projection.
\item \label{axiom:partial3} Likewise, in the strip
   end coordinates $e _{k}:
   [0,1] \times (- \infty,   0) \to \mathcal {S}
      _{r}$ at the end $e _{0}$ of $\mathcal{S}
      _{r}$, 
${u} (m _{1}, \ldots, m _{d}, n,r)$ has the form of the projection to $
[0,1]$ composed with $m _{0}$.  
\item \label{propertie4} For $0 \leq k \leq d-1$, $r
   \in \mathcal{R} _{d}$, the $k$'th component  of
      $\partial \mathcal{S} _{r} $ is mapped to $s
      (m _{k})$, and the $d$'th component is mapped  by ${u} (m _{1}, \ldots, m _{d}, n,r)$ to $t (m_d)$. 
\end{enumerate} 
\end {definition}    
See Figure \ref{fig:naturality} for an example
of a map satisfying partial naturality.
\begin{figure}[h]
  \includegraphics[width=3in]{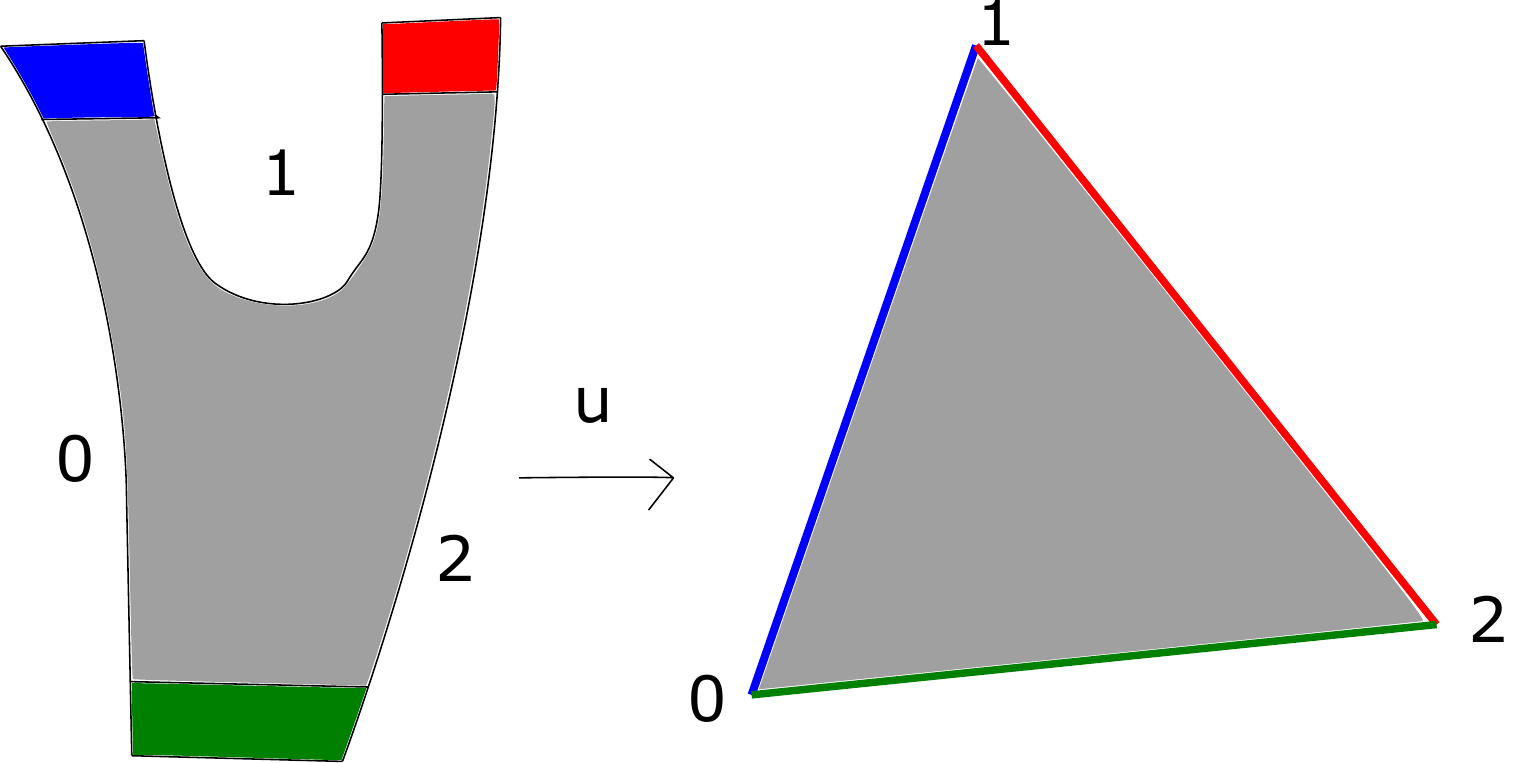}
 \caption {Figure for a map $u (m _{0,1}, m _{1,2},2)
$ satisfying partial naturality. The edges of the
   surface labeled $0,1,2$ are mapped to the
   vertices $0,1,2$ respectively. Colored ends are
   collapsed onto correspondingly colored edges.} 
   \label{fig:naturality}
\end{figure} 
So far we haven't put any special conditions for
$r \in \partial \overline{\mathcal{R} } _{d}$,
having to do with nodes.
These conditions will be forced by certain
additional naturality axioms,  which arise from
various gluing conditions.    After imposing these
additional axioms, Theorems
\ref{lemmanaturalmaps}, \ref{thm:naturaltargetdependent} which will follow, state that such a natural system of maps exists and is
unique up to suitable concordance.

We start by explaining the natural gluing
operations that appear.   
\begin{notation}
  \label{} In what follows we may
   interchangeably use the notation $\mathcal{S}
^{\circ } (d), \mathcal{S} (d)  $ for $\mathcal{S}
^{\circ}_{d}, \mathcal{S} _{d}$. This is to
avoid notation clash with various notations for
fibers. 
\end{notation}

First denote by $ \mathcal {T} (m _{1},
\ldots, m _{d}, n)$  the space of maps satisfying
the partial naturality properties above. We have the natural gluing map 
\begin{equation} \label {eq.composition}  St _{i}: \overline{\mathcal {R}}
_{s_1} \times \overline {\mathcal {R}} _{s_2} \times [0,1) \to
\overline{\mathcal {R}} _{s_1+s_2-1},
\end{equation}
whose value on $(r, r', \tau)$ is given by gluing
the surfaces $\mathcal{S} _{r}, \mathcal{S} _{r'} $ at the
root of $\mathcal{S} _{r}$  and at the $i$'th
marked point of $\mathcal{S} _{r'}$, with gluing parameter
$\tau \in [0,1)$, and then associating to this
glued surface its isomorphism class in $ \overline{\mathcal {R}} _{s_1+s_2-1}$.
 (When the value of the gluing parameter is 0, this is the composition map in the
 Stasheff topological $A _{\infty} $ operad).

Given an element $u \in \mathcal {T} (m _{1},
\ldots, m _{s_1}, n)$ and an element $$u' \in \mathcal {T} (m'_1, \ldots, m'
_{i-1}, m_1 \cdot \ldots \cdot m _{s_1}=m' _{i} ,
m' _{i+1}, \ldots, m' _{s_2}, n),$$ we have a
naturally induced map $$(u \star _{i}  u')_0: { \mathcal
{S}} ^{\circ} 
({s_1, s _{2}, 0}) \to  \Delta ^{n},$$ where $${ \mathcal
{S}} ^{\circ} 
({s_1, s _{2}, 0}) \to
\overline{\mathcal {R}} _{s_1} \times \overline{\mathcal {R}} _{s_2} $$ is the 
pullback of the fibration $${\mathcal{S}} ^{\circ}
({d = s _{1} + s _{2}-1  }) \to
\overline{\mathcal {R}} _{s_1+s_2-1}$$ by $St _{i}| _{\overline{\mathcal {R}}
_{s_1} \times \overline {\mathcal {R}} _{s_2} \times \{0\}} 
$. More specifically, the fiber $\mathcal{S} _{r
_{1}, r _{2}}$  of ${
   \mathcal
{S}} ^{\circ} 
({s_1, s _{2}, 0})$ over $(r
_{1}, r _{2})$ is the disjoint union of a pair of
distinguished (possibly disconnected) subspaces
${S} _{1}, {S} _{2}$ 
naturally identified with $\mathcal{S} _{ r _{1} } (s _{1})  $,
respectively
$\mathcal{S} _{r _{2} } (s _{2}) $.  
So under this identification, we apply $u$ to
${S} _{1}$  and
$u'$ to ${S} _{2}$, and this is the map $(u \star
_{i}  u')_0 | _{\mathcal{S} _{r _{1}, r _{2}}}$, 
see Figure  \ref{fig:splitmap}.
\begin{figure}[h]
  \includegraphics[width=4in]{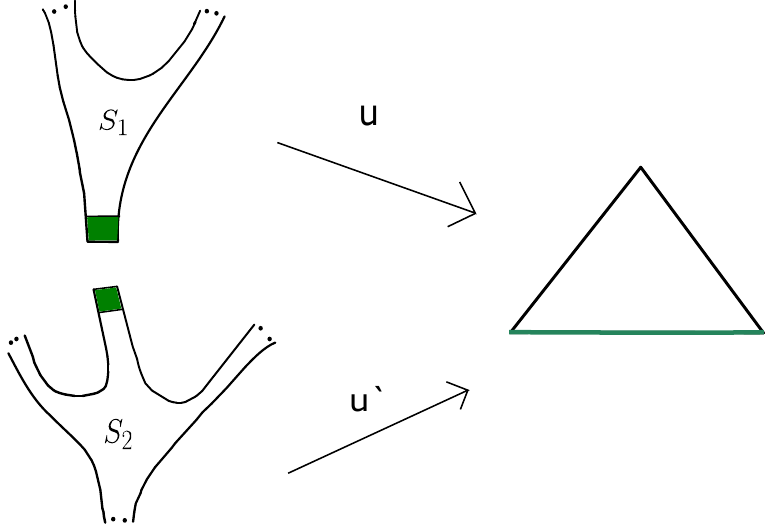}
 \caption {The green shaded ends are collapsed
   onto the same edge of $\Delta^{n} $.} \label{fig:splitmap}
\end{figure}

We can extend $(u \star _{i}  u')_0$ to a map
$$(u  \star _{i} 
u') _{1}:  { \mathcal {S}}^{\circ} ({s_1,s _{2}, 1})   \to \Delta ^{n},$$ where 
$${\mathcal {S}}  ^{\circ} ({s_1,s _{2}, 1}) \to \overline{\mathcal {R}} _{s_1} \times \overline{\mathcal {R}}
_{s_2} \times [0, 1) $$ is the pullback of the
family  
$${\mathcal{S}} ^{\circ}
({d = s _{1} + s _{2}-1  }) \to
\overline{\mathcal {R}} _{s_1+s_2-1}$$ by $St _{i}| _{\overline{\mathcal {R}} _{s_1} \times \overline{\mathcal {R}}
_{s_2} \times [0, 1)} 
$. 

To get this extension we specify  
\begin{equation}
   \label{eq:ustaru'taur}
   (u \star _{i}  u') _{r,r', \tau}: = (u \star _{i}  u') _{1}| _{\mathcal{S}
_{r,r',\tau} }, 
\end{equation}
for $\mathcal{S} _{r,r',\tau} $  the fiber of ${
   \mathcal {S}} ^{\circ} ({s_1,s _{2}, 1})$ over
$$(r, r', \tau) \in \overline{\mathcal {R}} _{s_1} \times \overline{\mathcal {R}}
_{s_2} \times [0, 1), \quad \tau   \in (0,1).
$$ Recall that the surface $\mathcal{S}_{r,r',\tau} $ glued from $\mathcal{S}_{r}, \mathcal{S} _{r'}  $ has a subdomain which we denote by 
$thin=thin _{\tau,i} $ that has a determined conformal identification with a strip of the form $[0,1] \times
(-\phi(\tau), \phi(\tau))$, for some continuous function (not
explicitly needed)  
\begin{equation}
   \label{eq:phiearly}
  \phi: (0,1] \to (0,\infty),    
\end{equation} 
s.t. $\lim _{\tau \mapsto 0} \phi (\tau) = \infty$.
This function is determined by
the particular parametrization of the gluing operation. 

Now $\mathcal{S}
_{r,r'} - thin $  is the disjoint union of 
subspaces holomorphically identified with
the regions $$Reg _{r} \subset
\mathcal{S} _{r}, Reg _{r'} \subset \mathcal{S}
_{r'},  $$ so that $Reg _{r} $ is identified with
$\mathcal{S} _{r} - \image {e _{0}} $, where $e
_{0}$ is the strip end chart for the root end of $\mathcal{S} _{r}$. And likewise, so that $Reg _{r'} $ is
identified $\mathcal{S} _{r'} - \image e _{i}$,
for $e _{i}$ the strip end chart for the $i$'th end of
$\mathcal{S} _{r'}$. 

We then define $(u \star _{i}  u') _{r,r', \tau}$ to coincide
with $u$, $u'$ on $Reg _{r} $ respectively $Reg _{r'} $, while on $thin$
in the distinguished coordinates $$[0,1] \times (-\phi (\tau),
\phi(\tau)),$$ $(u \star _{i}  u') _{r,r', \tau}$ is the
map given by the projection $$[0,1] \times (-\phi (\tau), \phi(\tau)) \to [0,1]$$ followed by 
the map $m' _{i} $, similarly to the Axiom
\ref{axiom:partial2} of partial naturality.  This operation is well-defined
for all $\tau \in (0,1) $  and so determines the
needed extension.  

In order to state the naturality axioms we need more geometry.
Let $N$ be a gluing normal neighborhood of the
boundary $\partial \overline{\mathcal{R}} _{d} $,
$d \geq 3$.    Recall that $\mathcal{R} _{d}$ has
dimension $d-2$. Let $S ^{d-3} _{0} \subset N  $ be an embedded sphere of dimension $d-3$, not intersecting  $\partial
\overline{\mathcal{R}} _{d} $, homotopic to the
inclusion $(\partial
\overline{\mathcal{R}} _{d} \simeq S ^{d-3}) \to N $. (A $0$-dimensional
sphere is understood throughout as a pair of points.) 
  Let $R ^{d-2} 
_{0} $ be the dimension $d-2$ ball subdomain of
$\overline{\mathcal{R}} _{d}$ bounded by $S ^{d-3}
_{0}  $.  Finally, let $\mathcal{S} _{r} - ends $ denote the compact Riemann surface with
boundary obtained from $\mathcal{S} _{r} $ by
removing the ends. Specifically, for $1 \leq i
\leq d$   we remove the  
images of the charts $e _{i}: {[0,1] \times (0,
\infty)} \to \mathcal{S} _{r}   $, and we remove the image of the chart $e
_{0}: [0,1] \times (-\infty,
0) \to \mathcal{S} _{r}$.   

For $r \notin
\partial \overline{\mathcal{R}}_{d}  $ set $D
^{2}_r: = \mathcal{S} _{r} - ends \simeq D ^{2}$.
And let $$G
^{2}  _{r} := (D ^{2} _{r},
\partial D ^{2}_{r}) \simeq (D ^{2}, \partial D
^{2}) $$ denote the
pair, where $\simeq$ is homeomorphism.

So we get a fiber
bundle of pairs $$(D ^{2}, \partial D ^{2})
\hookrightarrow (\widetilde{F}, A)  \to R ^{d-2} _{0},
$$ with fiber over $r$: $G ^{2} _{r}  $. And where
$A$ is just the corresponding sub-bundle with
fiber over $r$: $\partial D ^{2}_{r} $.  

We are almost ready to state our axioms.
Let $C _{s} (\Delta ^{n})$ denote the set of composable chains $(m_1, \ldots , m_s)$ of length
$s \geq 2$ in $\Pi (\Delta ^{n} )$. 
\begin{definition}\label{def:minimaldimension}
   For $(m _{1}, \cdots, m _{s})  \in C _{s} (\Delta ^{n}) $ let $D (m_1, \ldots, m
_{s})$ denote the minimal dimension of a subsimplex of $\Delta ^{n}  $
      which contains the edges corresponding to the morphisms $m _{i} $.
\end{definition}
 A system of maps $\mathcal{U}$ is an element of:
\begin{equation} \label{eq:productU}
\prod _{n \in \mathbb{N}} \prod _{s \in \mathbb{N}
   _{\geq 2} } \prod
_{(m_1, \ldots , m_s) \in C _{s} (\Delta ^{n} )}
\mathcal{T} (m_1, \ldots m_s, n).
\end{equation}
Given a system $\mathcal{U}$ its projection onto $(n, s, (m_1, \ldots , m_s))$
component will be denoted by $u (m_1, \ldots ,
m_s, n)$.   To put this another way, let $$B= \{(m_1, \ldots
m _{s}, n) | s \in
\mathbb{N} _{\geq 2}, n \in \mathbb{N},   (m
_{1}, \ldots, m _{s}) \in   C _{s} (\Delta^{n} )    \},$$
then set theoretically the product
\eqref{eq:productU} above is the set of certain
set maps $\mathcal{U}$ with domain $B$.  Then in
this language $u (m_1, \ldots , 
m_s, n) = \mathcal{U} (m_1, \ldots ,
m_s, n). $ 
\begin{definition} \label{defNaturality}
We say that $\mathcal{U}$ 
is \textbf{\emph{natural}} if it satisfies the following axioms:
\begin{enumerate}
  \item  \label{axiom:1} For all $s_1, s _{2}$ and for all $i$ if $m _{i}' = m _{1} \cdot \ldots \cdot m _{s _{1} }   $ then the map 
  \begin{equation} \label {eq.delta1} 
     ({u} (m _{1},
\ldots, m _{s _{1} }, n) \star _{i} u (m'_1,
     \ldots, m'_{s_2}, n)) _{1}
\end{equation}
coincides with the composition   
\begin{equation} \label {eq.delta2}
   {\mathcal {S}} ^{\circ}_{s_1, s _{2},
   1}
\xrightarrow{St _{i,*}} {\mathcal {S}}
   _{s_1+s_2-1} ^{\circ} \xrightarrow{u (m'_1,
\ldots, m' _{i-1}, m _{1},
\ldots, m _{s _{1} }, m _{i+1}, \ldots,   m'_{s_2}, n)} \Delta
^{n},
\end{equation}
for $St _{i,*}$ the bundle map induced by $St _{i} $. ($St
_{i,*}$ is the natural map in the corresponding pull-back square.) 
\item \label{axiom:naturality2}   Let $f: \Delta ^{n} \to \Delta ^{m}$  be a  simplicial map, which recall is an affine map preserving the orders of the vertices. Then there is an induced functor 
      $$f: \Pi
      (\Delta ^{n} ) \to \Pi (\Delta ^{m})
      $$ 
and  we ask that
\begin{equation*}
f \circ u(m_1, \ldots, m_s, n) = u (f(m_1),
   \ldots, f (m_s), m). 
\end{equation*}
\item \label{axiom:naturality3} 
  Let $(m _{1}, \ldots, m _{d}) \in C _{d}
  (\Delta^{d} ) $ for $d \geq 3$.    Suppose that $D (m_1, \ldots, m _{d} ) = d$ then
by the Lemma \ref{lemma:inducedmapofpairs} below, $u (m _{1}, \ldots, m _{d}, d)$ induces a map of pairs:
\begin{equation}  \label{eq:widetildeu}
\widetilde{u}:  
   (\widetilde{F}, A  \sqcup \widetilde{F} | _{S _{0}
   ^{d-3}})   \to (\Delta ^{d}, \partial \Delta
   ^{d}),
\end{equation}  
and we ask that $\widetilde{u} $ be a homological degree 1 map.
\end{enumerate}
\end{definition}   

\begin{remark}
   \label{remark:}
Note that when $d=2$ we may still state an analogue of Axiom \ref{axiom:naturality3}, but it would be immediate from the partial naturality properties of Definition
   \ref{def:partialnaturality}.
\end{remark}
\begin{lemma}
\label{lemma:inducedmapofpairs} 
   Let
$\mathcal{U} $ satisfy the first pair of axioms
in the Definition \ref{defNaturality}.
For $(m _{1}, \cdots, m _{d})  \in C _{s} (\Delta
   ^{d} ) $, $d \geq 3$,  
suppose that $D (m_1, \ldots, m _{d}) = d$. Then 
$\widetilde{u}:=u (m _{1}, \ldots, m _{d}, d)$ is
a map of pairs:
\begin{equation*} 
\widetilde{u}:  
   (\widetilde{F}, A  \sqcup \widetilde{F} | _{S _{0}
   ^{d-3}})   \to (\Delta ^{d}, \partial \Delta
   ^{d}).
\end{equation*}  
\end{lemma}
\begin{proof} 
   The map $\widetilde{u} $ maps $A$ 
   to $\partial \Delta^{d} $ by the partial
   naturality properties 2,3 and 4 in Definition
   \ref{def:partialnaturality}. In fact $A$ is
   mapped to the edges of $\Delta^{d} $.  
   
   Since
   $S _{0} ^{d-3}$ is contained $N$, by Axiom
   \ref{axiom:1}  for each $r \in S _{0}
   ^{d-3}$, ${u} _{r} = u (m _{1}, \ldots, m
   _{d},n,r)   $  has the form
   $(u _{1} \star _{i} u
   _{2}) _{\tau, r _{1}, r _{2}} $ for some $i$,
   for some $$u _{1} \in
   \mathcal {T} (m _{1}, \ldots, m _{s _{1}}, d)$$
   for some $$u _{2} \in \mathcal {T} (m'_1,
   \ldots, m' _{i-1}, m_1 \cdot \ldots \cdot m
   _{s_1}=m' _{i} , m' _{i+1}, \ldots, m' _{s_2},
   d),$$ and some $\tau, r _{1}, r _{2}$.
   
   Now 
   $s _{1} <d$ and $s _{2} <d$    since
   $s_1+s_2-1=d$, and since $s _{1}, s _{2} \geq
   2$. It then follows by this and by Axiom
   \ref{axiom:naturality2} that the images of $u _{1}$  and
	 $ u_{2}$  are contained in  proper
   faces of $\Delta^{d} $. Consequently, the image of
   $\widetilde{u}| _{S _{0} ^{d-3}}  $ is
   contained in $\partial \Delta^{d} $    and so
   we are done. 
\end{proof}
\begin{theorem} \label{lemmanaturalmaps}
   A natural system $\mathcal{U}$ exists. 
\end{theorem}
We first give a not explicitly constructive proof in all generality, and afterwards describe a partial explicit construction.
\begin{proof}  [Proof of \ref{lemmanaturalmaps}]
 To construct our maps $u (m_1, \ldots, m_s,n)$ we will proceed by induction. 
When $n=0$ there is nothing to do, as we have unique
maps for all $s \geq 2$, and they trivially
satisfy Axioms 1,2,3.  However, we will have to be
careful with the basis of the induction.

Given a composable sequence $(m _{1}, \ldots, m _{s}  )$ of morphisms in $\Pi (\Delta
^{n}) $, for some $n$, let $D (m_1, \ldots, m_s)$,
called the  \emph{$D$-number}, 
denote the least dimension of a non-degenerate
   subsimplex of $\Delta ^{n} $ that contains
the edges corresponding to $\{m _{i} \}$. Let $S
   (N) $ be the following statement: 
there are maps
\begin{equation} \label{eqInductionMaps}
u (m_1, \ldots, m_s, n),
\end{equation} for all $s \geq 2$ and all $n \leq
N$ and every composable chain $ (m_1, \ldots, m_s)$ with $D (m_1, \ldots, m_s) = s \leq n$ such that  axioms 1,3 are satisfied and such that 
\begin{equation}
	\label{eq:axiom2sigma}
	\sigma \circ u (m_1, \ldots, m _{s}, n) = u (\sigma (m_1), \ldots, \sigma(m _{s}), n)
\end{equation}
for $$\sigma: \Delta ^{n} \to \Delta ^{m}  $$
an injective simplicial map. That is we satisfy
Axiom \ref{axiom:naturality2} only partially and only for restricted $(m_1, \ldots,
m _{s} )$ at the moment.  

We will also denote by
$S (N) $ a corresponding collection of maps
\eqref{eqInductionMaps} with the requisite property.
 It can be seen directly that $S (N) $ holds for
   $N=0,1,2,3$. $N=0$ is the trivial case already
   considered above.   Indeed,   when $N=4$
   such a construction is given in the following
   Section \ref{sec:Outline}, and this also implies the
   cases of $N=1,2,3$. 
We intend to prove:
\begin{equation*}
   (N \geq 3) \land S (N) \implies S (N+1),
\end{equation*}
moreover the collection of maps $S (N+1) $  can be
chosen to extend  the collection of maps $S (N) $.

Note that the condition \ref{eq:axiom2sigma} and the maps \eqref{eqInductionMaps}
uniquely determine 
\begin{equation} \label{eqMapsInduced}
 u (m_1, \ldots, m_s, N+1)
\end{equation} 
 for $ \forall  s \geq 2, \forall  (m_1,
 \ldots, m_s)$ with  $D (m_1, \ldots, m_s) =s \leq N$.
We need an extension in the case $D (m_1, \ldots, m_s) = s = N+1$. 

By assumption $N+1>3$.  Let then $(m ^{0} _{1}, \ldots, m _{N +1} ^{0}    )$, be a chosen
   composable sequence with $D (m ^{0}_1, \ldots, m ^{0} _{N+1}  ) = N+1$.
Then gluing as in the Axiom \ref{axiom:1} of naturality and the maps \eqref{eqMapsInduced} naturally determine a map
\begin{equation} \label{eq:subs}
   {u}= {u} (m ^{0} _1, \ldots,
m ^{0}  _{N+1} , N+1): Sub _{N+1}   \to \Delta ^{N+1},
\end{equation}
 where $Sub _{N+1} = \rho ^{-1} (\partial \overline{\mathcal{R}} _{N+1} )  $, and where as before
 $$\rho: { \mathcal {S}} _{N+1} ^{\circ}  \to
 \overline{\mathcal{R}} _{N+1}.  $$ 

Let $U$  be a gluing normal neighborhood of
$\mathcal{R} _{N+1}$. (We previously used the letter
$N$, but just for this proof we use   the letter
$U$.)  
Extend ${u}$ in any way to $\rho ^{-1} (U) $, so
that Axiom \ref{axiom:1} of naturality
is satisfied.    We then need to further extend ${u}$ to ${ \mathcal {S}} _{N+1} ^{\circ} $ so that
 Axiom \ref{axiom:naturality3} of naturality is satisfied.  

Let $S ^{N-2} _{0}  \subset U $
be an embedded sphere in $U$ not intersecting
$\partial \overline{\mathcal{R}} _{N+1} $,
homotopic to the inclusion $\partial
\overline{\mathcal{R}} _{N+1} \to U$.
Then $u$ induces a map
of a pair
\begin{equation} 
   g: (\widetilde{F} | _{S _{0}
   ^{N-2}}, 
A  | _{S _{0}
   ^{N-2}})   \to (\partial \Delta ^{N+1} \simeq S ^{N}, loop ),
\end{equation}
where $loop$ is a topologically embedded $S ^{1} $ in $\partial \Delta ^{N+1} $
that is the image of the loop
 $$\gamma =m ^{0}  _{1} \cdot \ldots \cdot m ^{0} _{N+1} \cdot m  _{s
(m ^{0} _1), t (m ^{0} _{N+1})} ^{-1},$$ where $\cdot$ is concatenation of paths and the
order of composition is diagrammatic.
 The map $g$ is just the restriction of the map
 $\widetilde{u} $,  
\eqref{eq:widetildeu}.
\begin{lemma} The map $g$ is homological degree 1.  \end{lemma}
\begin{proof}
As $N > 2$ $loop$ has codimension greater than 1, so that the meaning of
homological degree is unambiguous, as the pair $(S ^{N}, loop) $ has a well-defined fundamental class by the homology long exact sequence for a pair.
Moreover, approximating $g$ by a smooth map we may compute the homological
degree via the smooth
degree, (denote the approximation still by $g$).
   That is let $f$ be a $N$-face of $\Delta ^{N+1} $ and $p \in interior (f)$ a
regular image point of $g$. The homological degree of $g$ is
then the count of elements of $g
^{-1} (p) $ with signs given by whether $dg _{k} $, $k \in g ^{-1} (p) $, is
orientation preserving or reversing.

   Suppose without loss of generality that the vertices of $f$ are $0, \ldots
,N$.  As the degree of $g$ is clearly independent of the choice of $S _{0} ^{N-2}  $ we may assume that $S _{0} ^{N-2}$  is chosen so that for some $\epsilon > 0$:
\begin{equation} \label{eq:ST1}   St _{1}: \left(R _{0} ^{N-2} \subset  \overline{\mathcal {R}}
_{N} \right) \times \overline {\mathcal {R}} _{2} \times \{\epsilon\} \to
\overline{\mathcal {R}} _{N+1}
\end{equation}
is an embedding into $S _{0} ^{N-2}  $. Note that   the map
   $$St _{1}: \left(R _{0} ^{N-2} \subset  \overline{\mathcal {R}}
_{N} \right) \times \overline {\mathcal {R}} _{2}
   \times \{0\} \to
\overline{\mathcal {R}} _{N+1}$$ may be understood
as a ``face map'' for the polyhedral model of
   $\mathcal{R} _{N+1}$.  
We may in addition suppose that $S _{0} ^{N-2} -T  $ is covered by such embeddings corresponding to the various other ``faces'' of $\overline{\mathcal{R}} _{N+1}  $, where the region $T \subset S _{0} ^{N-2} $, is such that $g$ maps $\rho ^{-1} (T)$ into the union of $(N-1)$-faces of $\Delta ^{N+1} $. The image of the map \eqref{eq:ST1} will be denoted by $V$.

   Then by the naturality Axiom \ref{axiom:naturality3} the face $f$ is covered by the image of $$\kappa=({u} (m ^{0}  _{1},
\ldots, m ^{0}  _{N}, N+1) \star _{1}  u (m ^{0} _1 \cdot
 \ldots \cdot m ^{0} _{N}, m ^{0} _{N+1}, N+1))  _{1}   \vert _{\widetilde{V}},$$  where $$\widetilde{V}  = \rho ^{-1} (V). $$ 
 
By construction, the smooth
degree of $g \vert _{\widetilde{V}}$ is the smooth
   degree of $\kappa $. But then, by the naturality
   Axiom \ref{axiom:naturality3}, $\kappa$ has smooth degree
one. And again, by naturality and the assumption on the form of  $S _{0} ^{N-2}
$, as described above, no other point of $\widetilde{F} | _{S _{0}
   ^{N-2}}   $ is in $g ^{-1} (p) $. It follows that $g$ is smooth degree one and so is homological degree one.

\end{proof}
\begin{lemma}
   \label{lemma:degree1}
  There is a degree one extension: 
\begin{equation*} 
\widetilde{g}:  (\widetilde{F}, A  \sqcup \widetilde{F} | _{S _{0}
   ^{N-2}})   \to (\Delta ^{N+1}, \partial \Delta ^{N+1}),
\end{equation*}
with $\widetilde{g} (A) = loop$.
\end{lemma}
\begin{proof} This is elementary topology so that we will not give exhaustive detail.
 Note that $\widetilde{F} \simeq B ^{N-1} \times D
^{2}$, where $B ^{N-1}$ denotes the unit ball in
   $\mathbb{R} ^{N-1} $. So to find the necessary degree one extension  it is
enough to show that given a degree one map 
   $h: (S ^{d-2} \times D
   ^{2}, S ^{d-2} \times \partial D ^{2}) \to  (S
   ^{d}, loop)$, there is a degree one extension
$$\widetilde{h}: (B ^{d-1} \times D
   ^{2}, \partial B ^{d-1} \times  D ^{2}) \to  (B
   ^{d+1}, \partial B ^{d+1}).$$ 
First note that a degree one map 
   $$\widetilde{t}:  (B ^{d-1} \times D
   ^{2}, \partial B ^{d-1} \times  D ^{2}) \to  (B
   ^{d+1}, \partial B ^{d+1}),$$ 
 exists by an elementary topology construction.
   (Start with the slicing diffeomorphism $[0,1]
   ^{d-1} \times [0,1] ^{2} \to [0,1] ^{d+1} $.) 
   We may in addition ensure that the restriction
   of $\widetilde{t} $ to the boundary $S ^{d-2} \times D
   ^{2}$ 
   is a map of a pair
   $$t: (S ^{d-2} \times D
   ^{2}, S ^{d-2} \times \partial D ^{2}) \to  (S
   ^{d}, loop). $$ 
 We may now apply  
   an analogue of the classical theorem of Hopf, which says that
   homotopy classes of maps of closed $d$-manifolds to $S
   ^{d}$ are classified by degree. (We say
   analogue because we are working with maps of pairs.) In our case
   Hopf's theorem readily implies that $t$ is
   homotopic to $h$ through maps of pairs. 
   Then use classical homotopy extension theorem,
   to get a homotopy from $\widetilde{t} $ to the
   needed map $\widetilde{h} $, extending $h$.
   \end{proof}

Continuing with the proof of the theorem, given  $\widetilde{g} $ as in the lemma above,  we may readily construct our extension $u:
 {\mathcal {S}} _{N+1} ^{\circ} \to \Delta ^{N+1} $ so
that $$u \in \mathcal {T} (m ^{0}  _{1},
\ldots, m ^{0} _{N+1}, N+1)$$ and so that the last naturality
axiom is satisfied. 

Now given any other composable sequence $(m_1,
\ldots, m _{N+1} )$ with $$D (m_1, \ldots, m
_{N+1})=N+1$$ let $$\sigma: \Delta ^{N+1} \to
\Delta ^{N+1}  $$ be the unique simplicial
bijective map taking $(m_1, \ldots, m _{N+1} )$ to $(m ^{0} _1,
\ldots, m ^{0}  _{N+1} )$, (it is unique because the action is determined by the
action on the vertices. 
) And we define
\begin{equation*}
u (m_1, \ldots, m _{N+1}, N+1): { \mathcal {S}} _{N+1} ^{\circ} \to
   \Delta ^{N+1}
\end{equation*}
by
$$u (m_1, \ldots, m _{N+1}, N+1 ) = \sigma ^{-1}  \circ u.
$$ 

We then obtain a collection of maps,
$$\{u
(m_1, \ldots, m _{s'}, n)\} _{(m _{1}, \ldots, m
_{s'} ), n \leq N +1, D (m _{1}, \ldots, m _{s'})
=  s' \leq n}, $$
and these maps satisfy the condition of
the statement $S (N+1)
$ by   the
construction and the inductive hypothesis.  We
thus complete the inductive step. Moreover, by
construction the
collection of maps $S (N+1) $  extends the
collection $S (N) $. 

By recursion, we may then define a sequence of
systems $\{S _{N}\} _{N \geq 0}$,  so that $S (N+1)  $
extends $S (N) $, for each $N$. 
We then obtain the total
collection:
\begin{align*}
    \mathcal{U}' =\bigcup _{N} S (N) =   \{u
(m_1, \ldots, m _{s'}, n) | n \in \mathbb{N},  (m _{1}, \ldots, m
   _{s'} ) \in C _{s'} (\Delta^{n} ), \\  \text{ such
   that } D (m _{1}, \ldots, m _{s'}) =  s'\}.
\end{align*}

It remains to extend the partial system
$\mathcal{U}' $ above to a full natural
system $\mathcal{U} $, that is we need to remove restrictions on the $D$-number. Given $(m _{1}, \ldots, m _{s}  )$, a composable sequence in $\Pi (\Delta ^{n} )$ with $s > n$, we may write $m _{i} = pr \circ
\widetilde{m}_{i} $ for $(\widetilde{m}_{1}, \ldots, \widetilde{m} _{s}    )$
a composable sequence in $\Delta ^{s} $ s.t. $D(\widetilde{m}_{1}, \ldots,
\widetilde{m} _{s}    )= s$, for $pr: \Delta ^{s} \to \Delta ^{n}  $ surjective
simplicial map. 
We then define $$u (m_1, \ldots, m _{s},n ) := pr \circ u (\widetilde{m}_{1}, \ldots, \widetilde{m} _{s},s  ).
$$  We should check that the above is well-defined.
Let $(\widetilde{m}'_{1}, \ldots, \widetilde{m}'
_{s})$ be another choice of a composable sequence
with $pr (\widetilde{m}' _{i}  ) = m _{i} $.
There is a unique simplicial bijective map $\sigma: \Delta ^{s} \to
\Delta ^{s}  $ fixing the image $i (\Delta ^{n}
)$, for $i: \Delta ^{n} \to \Delta ^{s}  $
inclusion of face, s.t. $pr \circ i = id$, and
s.t.  $\sigma (\widetilde{m} _{i}  ) =
\widetilde{m}' _{i}  $.

Then we have 
\begin{equation*} pr \circ u (\widetilde{m}'_{1}, \ldots, \widetilde{m}' _{s},s  ) =  pr \circ \sigma \circ u (\widetilde{m}_{1}, \ldots, \widetilde{m} _{s},s  ).
\end{equation*}
But $pr \circ \sigma = pr $ since $\sigma$ fixes $i (\Delta ^{n} )$. 
So we obtain that
\begin{equation*}
pr \circ u (\widetilde{m}'_{1}, \ldots, \widetilde{m}' _{s},s  ) = pr \circ u (\widetilde{m}_{1}, \ldots, \widetilde{m} _{s},s  ), 
\end{equation*}
so that $u (m_1, \ldots, m _{s},n )$ is well-defined.
So we have constructed our  system of maps satisfying all the axioms of naturality.
\end{proof} 
\subsection{Target dependent natural systems} 
Our natural systems $\mathcal{U}$ can be made
dependent on particular simplices
$\Sigma: \Delta ^{d} \to X$.  This is useful
for proving invariance later on. More
specifically, for $X$ a smooth manifold, a \textbf{\emph{target
dependent system}}   $\mathcal{U} (X) $
is an element of
\begin{equation} \label{eq:productUX}
\prod _{\Sigma ^{n} \in \Delta (X)} \prod _{s \in \mathbb{N}
   _{\geq 2} } \prod
_{(m_1, \ldots , m_s) \in C _{s} (\Delta ^{n} )}
\mathcal{T} (m_1, \ldots m_s, n),
\end{equation} 
where $\Sigma ^{n}$ is a singular $n$-simplex for  $n \in
\mathbb{N} $. 
As before given a system $\mathcal{U} (X) $ its
projection onto $(\Sigma ^{n}, s, (m_1, \ldots , m_s))$
component will be denoted by $u (m_1, \ldots,
m_s, \Sigma ^{n})$.   The superscript $n$  may be
omitted, when the degree $n$  is not explicitly needed. 
\begin{definition}   Similarly to the
   Definition \ref{defNaturality},  
we say that $\mathcal{U} (X) $ 
is \textbf{\emph{natural}} if it satisfies the following axioms:
\begin{enumerate}
  \item  \label{axiom:NATX1} For all $\Sigma
     ^{n}$, $s_1, s _{2}$ and for all $i$ if $m _{i}' = m _{1} \cdot \ldots \cdot m _{s _{1} }   $ then the map 
  \begin{equation} 
     ({u} (m _{1},
\ldots, m _{s _{1} }, \Sigma ^{n}) \star _{i} u (m'_1,
     \ldots, m'_{s_2}, \Sigma ^{n})) _{1}
\end{equation}
coincides with the composition   
\begin{equation} 
   {\mathcal {S}} ^{\circ}_{s_1, s _{2},
   1}
\xrightarrow{St _{i,*}} {\mathcal {S}}
   _{s_1+s_2-1} ^{\circ} \xrightarrow{u (m'_1,
\ldots, m' _{i-1}, m _{1},
\ldots, m _{s _{1} }, m _{i+1}, \ldots,
   m'_{s_2}, \Sigma ^{n})} \Delta
^{n},
\end{equation}
for $St _{i,*}$ the bundle map induced by $St _{i} $.  

\item \label{axiom:naturalityX2}   Let $f: \Sigma
^{n} \to   \Sigma  ^{m}$  be a
morphism in $\Delta (X)$, then there is an induced functor 
$$f: \Pi
(\Delta ^{n} ) \to \Pi (\Delta ^{m}),
$$ and we ask that
\begin{equation*}
f \circ u(m_1, \ldots, m_s, \Sigma ^{n}) = u (f(m_1),
   \ldots, f (m_s), \Sigma ^{m}),
\end{equation*}
where $f$  on the left is the corresponding
map $f: \Delta^{n} \to \Delta^{m} $, cf. \eqref{eq:morphismovercategory}. 

\item \label{axiom:naturalityX3} 
  Let $(m _{1}, \ldots, m _{d}) \in C _{d}
  (\Delta^{d} ) $ for $d \geq 3$.    Suppose that
$D (m_1, \ldots, m _{d} ) = d$. Then, as in
Definition \ref{defNaturality},
$u (m _{1}, \ldots, m _{d}, \Sigma ^{d})$ induces a map of pairs:
\begin{equation}  
\widetilde{u}:  
   (\widetilde{F}, A  \sqcup \widetilde{F} | _{S _{0}
   ^{d-3}})   \to (\Delta ^{d}, \partial \Delta
   ^{d}),
\end{equation}  
and we ask that $\widetilde{u} $ be a homological degree 1 map.
\end{enumerate}
\end{definition}  
\begin{theorem} \label{thm:naturaltargetdependent}
Given a smooth manifold $X$, a natural $\mathcal{U}
   (X) $ exists and is unique up to concordance.
   (We shall say what this means in the proof.) 
\end{theorem} 
\begin{proof}
Existence is simple. Pick a natural system $\mathcal{U} $ guaranteed by Theorem
   \ref{lemmanaturalmaps}. This induces a target
   dependent natural system $\mathcal{U} (X) $ defined by:
$$\forall \Sigma ^{n}:  u (m _{1}, \ldots, m _{s},
   \Sigma ^{n}) = u (m _{1}, \ldots, m _{s}, n),  $$
where the maps $u (m _{1}, \ldots, m _{s}, n)$ correspond to
   $\mathcal{U} $.  
   
  Now uniqueness up to concordance means that
   given a pair $\mathcal{U} _{1} (X), \, \mathcal{U}
   _{2} (X) $ of natural systems, there is a
   natural system $\widetilde{\mathcal{U} } (X
   \times I) $ whose restriction to $X \times
   \{i\}$ is $\mathcal{U} _{i} (X)  $, $i=0, 1$.        
The proof of this is totally analogous to
   the inductive construction in the
   proof of Theorem
   \ref{lemmanaturalmaps}.
\end{proof}
 
\subsection {Outline of an explicit construction} \label{sec:Outline}
This section is not logically necessary, but in
order to give the reader more intuition we now
give a partial explicit construction of a natural
system $\mathcal{U} $ (not its target dependent
analogue).    To be more formal, we give an  explicit
construction of the system satisfying condition $S (4) $, as in the proof
of Theorem \ref{lemmanaturalmaps}. However, we
will not check all the properties. (Although this would be straightforward.) 
This construction could be in principle extended to all generality but at the cost of much complexity. 

Fix a geometric model for $ \overline{\mathcal {R}} _{d}$, for
example as the Stasheff associahedra.
When $d=4$ this is a pentagon.
Recall that to each corner of $ \overline{ \mathcal {R}}_4 $ we have a uniquely
associated nodal Riemann surface with 3 components and 5 marked points, one of which is called the root. 

Recall that we label the root component by $\omega$, the
  next component by $ \beta$ and the component furthest from root by $\alpha$.
  (With respect to the linear ordering described earlier.)
Denote by $M _{\alpha}$ the collection of marked points, different from the root
$e _{0} $, on $\alpha$, likewise
with $\beta, \omega$. This determines
a sub-composable sequence $mor (S _{\alpha})$ of a composable sequence $ (m
_{1} , \ldots, m _{4})$,
and likewise with $\beta, \omega$, (note that $M _{\omega}, M _{\beta} $ could be empty).
 
Let $r$ be in the gluing normal neighborhood of some corner,
corresponding to
non-zero gluing parameters $d _{\alpha, \beta} $, $d _{\beta, \omega} $. 
We now construct a map
$$f _{r} = f _{r} (m _{1}, \ldots, m _{4}): [0,4] \times
[0,1] \to \Delta ^{4}.$$
In what follows by \emph{concatenation} of a collection of paths we mean their
product in the Moore path category of $\Delta ^{4} $, the notation for
composition will be assumed to be diagrammatic.
This is the category with
objects: points of $\Delta ^{4} $, and morphisms from $x _{0} $ to $x _{1} $:
continuous paths $[0,T] \to \Delta ^{4} $, $T>0$, between $x _{0}, x _{1}  $,
with composition the natural concatenation of
paths. Note that this is quite different from our previously defined groupoid $\Pi (\Delta ^{4} )$.

For a morphism $m$ in $\Pi (\Delta ^{4})$
let $s(m)$ and $t({m})$ denote the source respectively target of $m$. Let $H
^{m}: \Delta ^{4} \times [0,1] \to \Delta ^{4}$ denote  the natural deformation
retraction of $\Delta ^{4}$ onto the edge determined by $s (m), t (m)$, with
time $1$ map the orthogonal affine projection onto this edge (for the standard
metric on $\Delta ^{n} $).
 Set $H
^{m} _{\tau} = H ^{m}| _{\Delta ^{n} \times \{\tau\}}$. Next, for a general
piece-wise affine path $p: [0,T] \to \Delta ^{4}$, with end points $s (m), t
(m)$, we have a homotopy $H ^{m} _{\tau}  \circ
p$, $ \tau \in [0,1]$, from $p$ to a path $$ \widetilde{p}: [0,T] \to \Delta
^{4}$$ with image in the edge determined by $m$.
Let $D (p, \tau)$, $ \tau \in [0,1]$, be the concatenation of $H ^{m} _{\tau}  \circ
p$ with the homotopy $G
_{\tau}$, $\tau \in [0,1]$, of paths with fixed end points, from $ \widetilde{p}$ to the map $$
\widetilde{m}: [0,T] \to \Delta ^{4}$$ linearly parametrizing the edge determined by $m$. This second homotopy
$G _{\tau}$ can be chosen in a way that depends only on $
\widetilde{p}$. This can be done explicitly, using piece-wise linearity of $p$.

 The map $f ^{t} _{r}$ from
the $y=t$ slice $ [0,4] \times \{t\}$ is constructed 
as follows.  Set $$I _{\alpha}= (1 -d
_{\alpha, \beta})/2,$$ and set $f _{\alpha, r}$ to be the concatenation  of the morphisms
in $mor(M _{\alpha})$. That is if $$M _{\alpha} = (m ^{\alpha} _{1}, \ldots, m
^{\alpha} _{k}) $$ then $$f _{\alpha, {r} } =  m ^{\alpha} _{1} \cdot \ldots \cdot
m ^{\alpha} _{k}.
$$ Then for  $t \in [0, I
_{\alpha}]$
set $ f ^{t} _{\alpha,r} = D ( f
_{\alpha, r},  2t) $. 
Then set $f ^{t} _{r}$ to be the concatenation of morphisms in $mor(M
_{\beta})$, $mor(M _{\gamma})$ and of $ f ^{t} _{\alpha, r}$, in that order,
although 
note that the
order of the morphisms in the concatenation is uniquely determined by the end
point conditions, this holds further on as well.

Next set  $$I _{\beta} = I _{\alpha}+(1- I _{\alpha})(1
- d _{\beta})/2.$$ If $\alpha$ and $\beta$ components have a nodal point in common we set $$f _{\beta,r}: [0,4] \times \{t\} \to \Delta ^{4}$$ to be the
concatenation of $f ^{I _{\alpha}} _{\alpha, r} $ with morphisms in $mor(M
_{\beta})$, and for $t \in [I _{\alpha}, I _{\beta}]$
we set $$ f ^{t} _{\beta,r}  = D( f _{\beta,r},  {\frac{2(t - I
_{\alpha})}{1 - I _{\alpha}}}  ).$$
And then for $t \in [I _{\alpha}, I _{\beta}]$ set $f ^{t} _{r}$ to be the concatenation of morphisms in  $mor(M _{\omega})$ and of $ f ^{t} _{\beta,r}$. 

Finally, set $f  _{\omega,r} $ to be the concatenation of $f  ^{I _{\beta}} _{\beta, r}$ with morphisms in $mor(M _{\omega})$, and for $ t \in [I _{\beta}, 1]$ set $${f} ^{t} _{r} = D( {f _{\omega,r}}, \frac{2(t - I _{\beta})}{1-I _{\beta}}).$$ 

When $\alpha$  has a nodal point in common with
the $\omega$ component  set $f _{\beta,r}$ 
to be the concatenation of  morphisms in $mor(M _{\beta})$, and for $t \in [I _{\alpha}, I _{\beta}]$
set $$ f ^{t} _{\beta,r}  = D({f _{\beta,r}},
{\frac{2(t - I _{\alpha})}{1 - I _{\alpha}}}).$$
Then for $t \in [I _{\alpha}, I _{\beta}]$
set $f ^{t} _{r}$ to be the
concatenation of morphisms $f ^{I _{\alpha}} _{r}$ and $ f
^{t} _{\beta,r}$, and $mor (M _{\omega})$ (although $ mor (M _{\omega})$ in this
particular case is empty, we add this so that the degenerate case $M
_{\alpha} = \emptyset$, $M _{\beta}=\emptyset$ makes sense, see the discussion
below).
Finally, for $ t \in [I _{\beta}, 1]$ set $$f  ^{t} _{r} = D({f _{r} ^{I _{\beta}}}, {\frac{2(t - I _{\beta})}{1 - I _{\beta}}}).$$

When $r \in \overline{\mathcal {R}} _{4}$ is in the gluing neighborhood of a face but not
of a corner the construction of $f_r: [0,4] \times [0,1] \to \Delta ^{4}$ is similar, in fact we
can think of it as a special case of the above by setting $d _{\beta}=1$, $M
_{\beta} = \emptyset$. 
When $r \in \overline{\mathcal {R}} _{4}$ is not in the gluing neighborhood of the
boundary we can also think of this as a special case of the above, with
 $M _{\alpha} = \emptyset, M _{\beta} = \emptyset$, $d _{\alpha}=1$, $d _{\beta}=1$ in the
construction above. 

\subsubsection {Retracting $\mathcal{S} _{r} $ onto $[0,4] \times [0,1]$}
We now construct a smooth $r$-family of maps $$ret _{r}:  { \mathcal {S}} _{r} \to
[0,4] \times [0,1],$$ $r \in {\mathcal {R}}  _{4}$,
 suitably compatible with the maps $$ {f} _{r}: [0,4]
\times [0,1] \to \Delta ^{4}.
$$ 
In figure \ref{coloredtrees} $ (a)$, $ (b)$, $ (c)$ represent cases where $ (c)$: $r$ is
not within gluing normal neighborhood of boundary, $ (b)$: $r$ is in a gluing neighborhood 
of a side but not a corner and $ (a)$: $r$ is within gluing neighborhood of a
corner, (we picked a particular corner and side for these diagrams). The color shading
will be explained in a moment. In each case $ (a), (b), (c)$ we first color
shade $ [0,4] \times [0,1]$ as in figure \ref{squareabg}, the green region is
the domain of $ f ^{t} _{\alpha, r}$ contained in $ [0,4]
\times [0, I _{\alpha}]$, in the blue regions the map $f _{r}$ is vertically
constant, the red region is the domain of $ f ^{t} _{\beta,r}$
contained in $ [0,4] \times [I _{\alpha}, I _{\beta}]$ and yellow region is the rest of the domain of
$f _{r}$. The maps $ret _{r}$ are defined
for each $r$ by taking color shaded areas to color shaded areas, so that 
the following holds. 
\def\svgwidth{4in} 
\begin{figure}
\centering 
\includegraphics[width=4in]{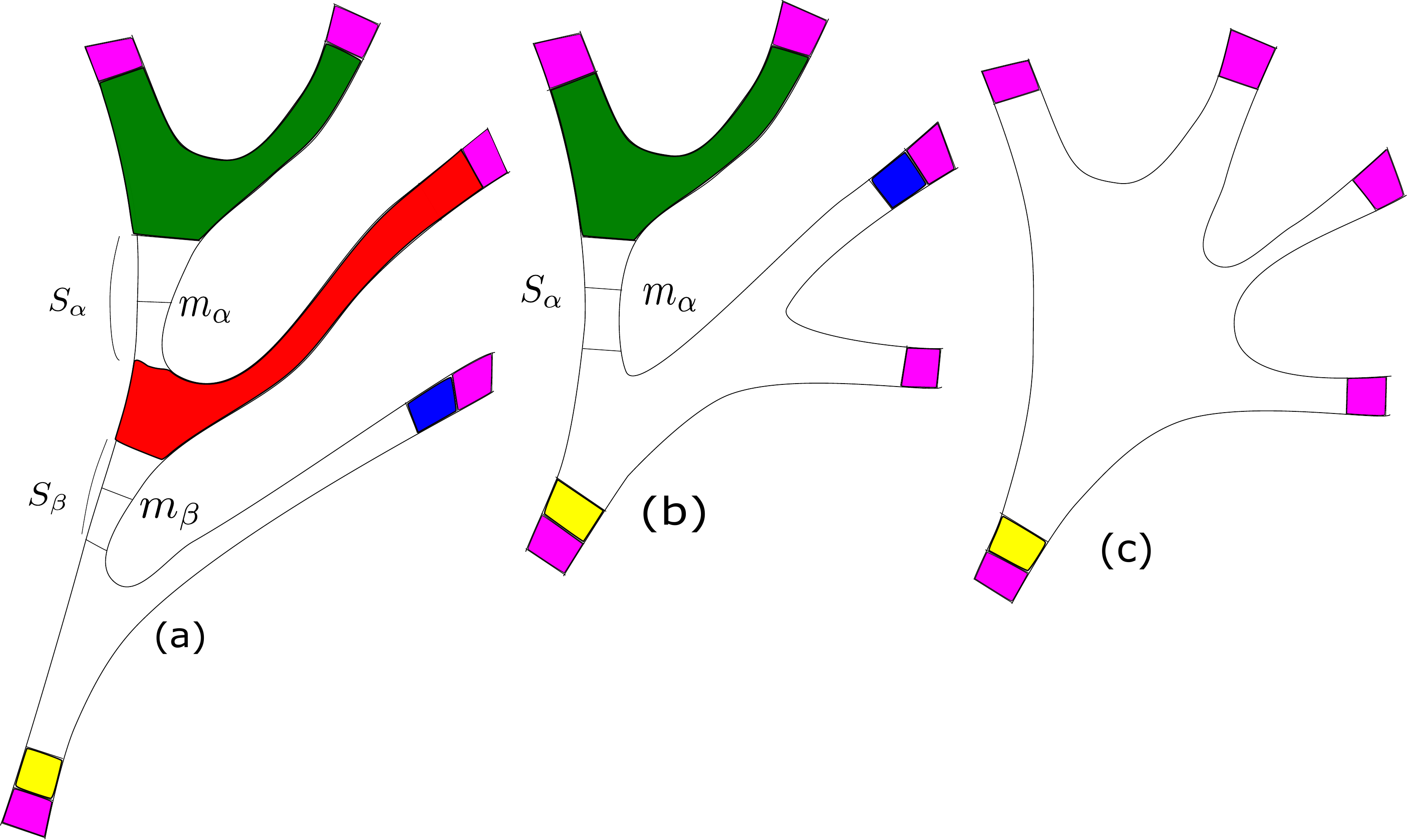}
 \caption{The uncolored enclosed regions labeled 
 $S _{\alpha}$, $S _{\beta}$
 surrounding segments $m _{\alpha}$, $m _{\beta}$ are ``thin''.}   
 \label {coloredtrees}
\end{figure}
 
 \def\svgwidth{2in} 
\begin{figure}   
\centering 
\includegraphics[width=2in]{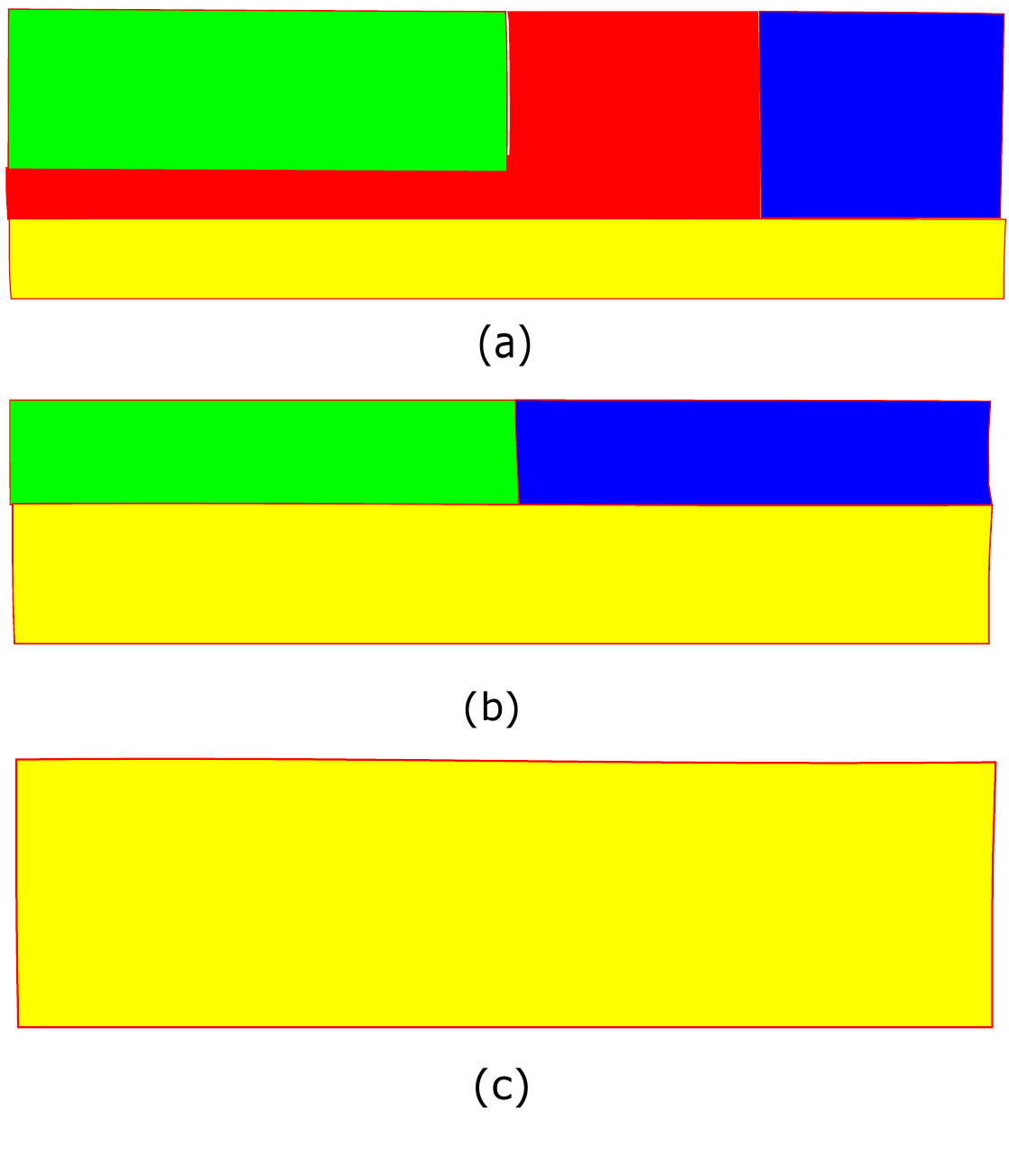}
   \caption{Diagram for $S_d$. Solid black border is boundary, while
   dashed red lines are open ends.
   The connection $\mathcal {A} ({ r,\{L_i\}})$ preserves
   Lagrangians $L_i$ over boundary components labeled $L_i$.}
 \label {squareabg}
\end{figure}
\begin{enumerate} 
  \item The ends $e _{i}$, $i=1, \ldots, 4$  of
  $ \mathcal {S} _{r}$, colored in purple, are
      identified in strip end coordinates as $
      (0, \infty)
  \times [0,1]$ and in these coordinates $ret _{r}$ is the composition of the projection $  (0,
  \infty) \times [0,1] \to [0,1]$, with the map
      $\underline {m} _{i} $  to the boundary of $
      [0,4] \times [0,1]$,  characterized as
      follows:  $f
      _{r} \circ \underline {m} _{i} $ parametrizes the  morphism $m _i \in \Pi (\Delta
      ^{4})$.   Similarly for the $e
      _{0}$ end.

\item The boundary of $\mathcal{S} _{r}$ goes either to the boundary of $ [0,4]
\times [0,1]$ or to the vertical boundary lines between colored regions.
\item The unshaded ``thin'' regions labeled $S _{\alpha}$, $S _{\beta}$ come
from the gluing construction and are identified with $ [0,1] \times (-\phi (\tau _{\alpha} ), \phi (\tau _{\alpha} ))$, respectively $ [0,1] \times (-\phi (\tau _{\beta} ), \phi (\tau _{\beta} ))$. In these coordinates
$ret _{r}$ on $S _{\alpha}$, $S _{\beta}$ is the projection to $ [0,1]$ composed
with a diffeomorphism onto the lower edge of
      the green, respectively the red
region, (affine in respective natural coordinates). 
\item The unshaded part of $\mathcal{S}_{r} $ is collapsed onto the horizontal line
bounding yellow region of $ [0,4] \times [0,1]$.
\item Blue shaded regions are identified in strip
   end coordinates $ (0,
\infty) \times [0,1] \to \Sigma _{r}$, as $ [0,1] \times [0,1]$, and are mapped
to the corresponding blue regions in $ [0,4] \times [0,1]$.
\end{enumerate}
(The above prescription naturally extends to the boundary $ \overline{\mathcal
{R}} _{4}$.)

We then set $$u (m _{1}, \ldots, m _{4}, \Delta ^{4}) | _{\mathcal{S} _{r} } = f _{r} \circ ret _{r}.  $$ 
These are almost the maps we want, but we need to
``collar them'' near the boundary of
$\overline{\mathcal{R}} _{4}  $, so that Axiom 1 of
naturality is satisfied.   We omit the details.

\section {Auxiliary data $\mathcal{D}$} 
\label{section:dataD} 
Let $$M \hookrightarrow P \xrightarrow{\pi} X$$ be a Hamiltonian fiber bundle, with model fiber $(M, \omega)$
that we shall assume here to be 
a closed, monotone: $$ [\omega] = const \cdot 2 c _{1} (TM),$$
$const > 0$, symplectic manifold. The constant
$const$  is called the \emph{monotonicity}  constant.  

We now discuss geometric-analysis theoretic data
needed for the construction of the functor $F
_{P}: \Delta^{} (X) \to A _{\infty}-Cat ^{unit}$, as
outlined in Section \ref{sec:IntroAfunctor} of the introduction. Essentially, this
data $\mathcal{D} = \mathcal{D} (P)  $ specifies a
choice of a (target dependent) natural system
$\mathcal{U}$   and various choices of
Hamiltonian connections, as well as certain
choices of almost complex structures.  These
choices are to be made for each $\Sigma \in Simp
(X)  $, while being suitably compatible, so that
we obtain our functor $F _{P}$.

For $(M, \omega) $   as above, we say that a Lagrangian submanifold $L \subset M$ is
\emph{monotone} if
the homomorphisms given by symplectic area and Maslov class
$$[\omega]: H_2(M, L) \to \mathbb{R}, \quad \mu: H _{2} (M, L) \to \mathbb{Z} $$
are proportional: $$[\omega] = const \cdot \mu.$$ 

For an $x: pt \to X$, define 
\begin{equation}
   \label{eq:Fx}
   F (x) =F _{P} (x),
\end{equation}
to be the set of
oriented, spin, monotone Lagrangian submanifold
$L$  in $$(P _{x}=x ^{*}P, \omega _{x}) \simeq (M,
\omega)$$ with minimal
(positive) Maslov
number at least 2, and such that the inclusion map $\pi _{1} (L) \to \pi _{1} (M)
$ vanishes. We call elements of $F _{P} (x)$
\textbf{\emph{objects}}. These will in fact be
objects of a  certain $A _{\infty}$ category to be
constructed.

Let $L \in F _{P} (x)$, and $j$ be an almost complex
structure on $P _{x}  $ tamed by $\omega
_{x}$, meaning that $$\forall v \in TP _{x}, v \neq
0:   \omega _{x} (v, jv)   > 0. $$   Let $\mathcal{M}
     (L,j)$ denote the moduli space of Maslov
number 2 $j$-holomorphic discs in $P _{x}  $, with
one marked point on the boundary, with boundary of
the disk in $L$.
It is well known, see Sheridan~\cite[Section 2.3]
{citeNickSheridanOntheFukayaCategory} (which also
contains a number of additional references)  that for a generic such $j$, $\mathcal{M}
(L,j)$ is regular, is a transversely cut out
$n$-dimensional manifold and is compact. The
compactness is due to the following fact.  If $\mathcal{M}
(L,j)$  were not compact, then by Gromov
compactness there would be a sequence of curves in $$\mathcal{M}
(L,j)$$
degenerating to a nodal curve with at least a pair
of components. One of these components has Maslov
number at least $2$, by our assumption on the minimal Maslov
number. And the other component 
contributes positively to the total Maslov number
of the nodal curve.   (The monotonicity, and
energy positivity preclude negative Maslov/Chern number components.) 
This would clearly be a contradiction, by
the additivity of the Maslov number.

Then we have a map corresponding to the evaluation at the marked point:
\begin{equation*}
ev: \mathcal{M} (L, j) \to L, 
\end{equation*}
and we define $\omega (L) \in \mathbb{Z}$ as the degree of $ev$.

Given a smooth $$\Sigma: \Delta ^{n} \to X, $$  
 set $x _{i}: = \Sigma (i) \in X $, $i \in
\{0, \ldots, n\}$ a vertex of $\Delta^{n} $.  Also
denote by $x
_{i}$   the corresponding inclusion map $x _{i}:
pt \to X$.  Set 
\begin{equation}
   \label{eq:FPSigma}
   F _{P} (\Sigma):= \bigsqcup _{i} F _{P} (x
   _{i}),  
\end{equation}
with elements likewise called objects, at the moment this is just a set of Lagrangians,
but later on this will be the set of objects of a
certain $A _{\infty}$ category, (with the same name). 

Given a pair of objects  $$L _{0}, L _{1}
\in F _{P} (\Sigma ),$$  satisfying 
$$\omega (L _{0} ) = \omega (L _{1} ),$$
and such that $L _{0} \in F
_{P} (x _{i}), L _{1} \in F _{P} (x _{j}),  $ 
let $$m = m _{L _{0}, L _{1}}: [0,1] \to \Delta
^{n}$$ denote the edge
between $i,j$ corners of $\Delta ^{n} $, $i,j \in
\{0, \ldots, n\}$.  We then set $$
\overline{m} := \Sigma \circ m.
$$  
For each such $\Sigma$, and for each $L _{0}, L _{1}$
as above, the data $\mathcal{D}$
prescribes a  Hamiltonian connection $$ \mathcal {A} (L _{0}, L _{1})= \mathcal {A} (L _{0}, L _{1}, \overline{m} )$$  on
$ \overline{m} ^{*} P$. (See Section
\ref{sec:HamiltonianBundles} for definition of
Hamiltonian connections.) 

Denote by $\mathcal{A} (L _{0}, L _{1}  ) (L _{0}
)$ the Lagrangian $\phi (L _{0})  \subset P _{x _{j} }
$, for $\phi$ the $\mathcal{A} (L _{0}, L _{1}  )
$-parallel transport map over $[0,1]$.
Then we require that
$\mathcal{A} (L _{0}, L _{1}  ) (L _{0} )$ be transverse to
$L _{1} $.
\begin{definition}
   \label{notation:SL0L1}  
   Let $$S (L _{0}, L _{1}  ) = S (L _{0}, L _{1},
   \mathcal{A} (L _{0}, L _{1}) )  $$ denote the space of $ \mathcal {A} ({L_0,
 L_1})$-flat sections with boundary on $L _{0}, L
   _{1}  $, over $0$ respectively over $1$.  In other words elements of $S (L _{0},
   L _{1}  ) $ are sections $$\gamma: [0,1] \to
   \overline{m} _{L _{0}, L _{1}} ^{*}P, $$
   tangent to the 
   $\mathcal{A} (L _{0}, L _{1}) $-horizontal
   distribution and satisfying $\gamma (0)  \in  L
   _{0}$ and $\gamma (1) \in L _{1}$.   By a \textbf{\emph{starting
   position}}  
   of an element $\gamma \in S (L _{0}, L _{1}  ) $ we
   mean $\gamma (0) \in L _{0}$. Likewise by
   an \textbf{\emph{ending position}}  of an element $\gamma \in S (L _{0}, L _{1}  ) $ we
   mean $\gamma (1) \in L _{1}$.
\end{definition}

\begin{definition} \label{def:jtadmissible}
 Let $m = m _{L_0, L _{1}} $ be as above,  and let 
$$\{j _{t} \} _{t \in [0,1]} = j  (L _{0}, L _{1},
   \overline{m}   ) $$ be a family of fiber-wise
   almost complex structures on $ \overline{m}
   ^{*} P$, s.t. for each $t$ $j _{t}$ is tamed by the
   symplectic form $\omega
   _{\overline{m} (t) }$ on $P _{\overline{m} (t)
   }$.  Then $\{j _{t} \}$  is said to be \textbf{\emph{admissible with respect to $\mathcal{A}(L _{0}, L _{1})$}} if the following holds. 
\begin{itemize}
  \item  
For each $t$,  Chern number $1$ $j _{t}  $-holomorphic spheres in $P _{
\overline{m} (t) } \subset  \overline{m} ^{*} P  $
do not intersect any of the images of any elements of $S (L _{0}, L _{1}  ) $. 
\item The moduli spaces $\mathcal{M} (L _{0}, j _{0}  )$, $\mathcal{M} (L _{1}, j
_{1} )$ are regular, and the evaluation map
\begin{equation*}
ev _{0} : \mathcal{M} (L _{0}, j _{0}  ) \to L _{0}
\end{equation*}
does not
intersect the set of starting positions of elements of $S (L _{0}, L _{1}  )$.
Likewise, the evaluation map
\begin{equation*}
ev _{1} : \mathcal{M} (L _{1}, j _{1}  ) \to L _{1}
\end{equation*}
does not
intersect the set of ending positions of elements of $S (L _{0}, L _{1}  )$.
\end{itemize}
\end{definition}
Such a family $j  (L _{0}, L _{1},
   \overline{m}   )$ is easily seen to exist, see
Sheridan~\cite{citeNickSheridanOntheFukayaCategory}. Our
data $\mathcal{D} $  then fixes a choice of such
$j  (L _{0}, L _{1},
   \overline{m}   )$ for each $\Sigma, m, L _{0},
L _{1}$ as above.

Next $\mathcal{D}$ makes a choice of a target
dependent natural system $\mathcal{U} (X) $. 
Finally, $\mathcal{D}$
will specify a certain natural system of
Hamiltonian connections and a system  of complex structures that
we now describe. 
This is to be done for all choices of
certain Lagrangian labels. This involves some
necessarily complicated notation, but there is nothing deep
going on, once we have the geometric input of the
system $\mathcal{U} (X) $.
Loosely speaking, $\mathcal{D} $ is just a system of compatible perturbations in the sense of Sheridan and Seidel but relative to
$\mathcal{U} (X) $. 
\subsection {From a Hamiltonian fibration over $X$
to Hamiltonian fibrations over surfaces} 
Let $M \hookrightarrow P \to X$ be as before, and
let a natural system $\mathcal{U} (X) $   be
chosen. 
Given a
composable chain $ (m_1, \ldots, m_d)$ and a map
$$u (m_1, \ldots, m_d, \Sigma ^{n}):
{\mathcal {S}} _{d} ^{\circ}  \to \Delta ^{n} \text{ that is part of a natural system
$\mathcal{U} (X) $},$$ we
have an induced fibration $$M \hookrightarrow \widetilde{S} (m_1, \ldots, m_d,
\Sigma ^{n}) \to {\mathcal {S}} _{d} ^{\circ}  $$ by pulling back $M
\hookrightarrow P \to X$ first by $\Sigma ^{n}: \Delta ^{n} \to X$ and then by
$u (m_1, \ldots, m_d, \Sigma ^{n})$. We have a natural projection 
$$  \widetilde{S} (m_1, \ldots, m_d, \Sigma ^{n}) 
\to \overline{\mathcal{R}} _{d},  $$
and we denote the fiber over $r \in \overline{{ \mathcal {R}}} _{d}$ by $ \widetilde{S} (m_1, \ldots,
m_d, \Sigma ^{n},r)$, or simply by $\widetilde{\mathcal{S}}_{r}  $ where there
can be no confusion. So $\widetilde{\mathcal{S}}_{r}  $
is naturally a Hamiltonian $M$-fibration over the surface $ \mathcal {S} _{r}$, smooth over smooth components.  To
state this another way, $$\widetilde{\mathcal{S} } _{r} = (u (m
_{1}, \ldots, m _{d}, \Sigma ^{n}, r) \circ \Sigma) ^{*} P,$$
where 
\begin{equation} \label{eq:unr}
  u (m
_{1}, \ldots, m _{d}, \Sigma ^{n},r) := u (m
_{1}, \ldots, m _{d}, \Sigma ^{n})| _{\mathcal{S} _{r}}.
\end{equation} 
\subsubsection {Distinguished trivializations} 
By the partial naturality properties  of the maps $u (m _{1}, \ldots, m _{d}, \Sigma ^{n},r) $, at each $e _{i}$ end, $1 \leq i \leq d$,   of $ \mathcal {S} _{r}$, we have natural
trivializations $$(0, \infty) \times \overline{m} _{i}  ^{*} P \to \widetilde{\mathcal{S} } _{r}.$$ Similarly, at
the $e _{0}$ and.  For $r \in \mathcal{R} _{d}$,  
we also have natural trivializations of
$\widetilde{\mathcal{S} } _{r} $  over the
$i$'th boundary component of $\mathcal{S} _{r}$,
$0 \leq i \leq d-1$, as $\mathbb{R} ^{} \times P
_{s (m _{i+1})}$. 
  Or as $ \mathbb{R} ^{}  \times P _{t (m_d)} $
over $d$'th boundary component.  The ordering is
as described in the preamble of Section \ref{sec:systemOfmaps},
see also Figure \ref{fig:sidenumberedsurface}. 
There are analogous natural trivializations also for general $r \in
\overline{\mathcal{R} } _{d}$.   
We shall call these \textbf{\emph{distinguished
trivializations/coordinates}}.   And the structure
of these trivializations will be called the
\textbf{\emph{distinguished trivialization
structure}}. 
\subsection {Lagrangian labels and
admissible connections} 
Given $r \in \mathcal {R} _{d}$, 
and given choices $$L _{i} \in F (\overline{m} _{i+1}
(0)), \text{ for $0 \leq i \leq d-1$, } \quad L
_{d} \in F (\overline{m} _{d} (1)),$$ 
such that $\omega (L _{i}) = \omega (L _{d} )$ for
all $i =0, \ldots, d$,
a \emph{labeling} is just an assignment of $L
_{i}$ to the $i$'th component of $\partial \mathcal{S}
_{r}$.   
Extend the labeling construction above naturally to
$\mathcal{S} _{r}$, with   $r \in \partial
\overline{ \mathcal {R}} _{d}$.  In other words,
for such an $\mathcal{S} _{r}$,  we
label the boundary components in such a way that if
we glue at some node of $\mathcal{S} _{r}$ then
each boundary component of the glued surface inherits a consistent label.  
See Figure \ref{fig:label} below. 
\begin{figure}[h]
  \includegraphics[width=2in]{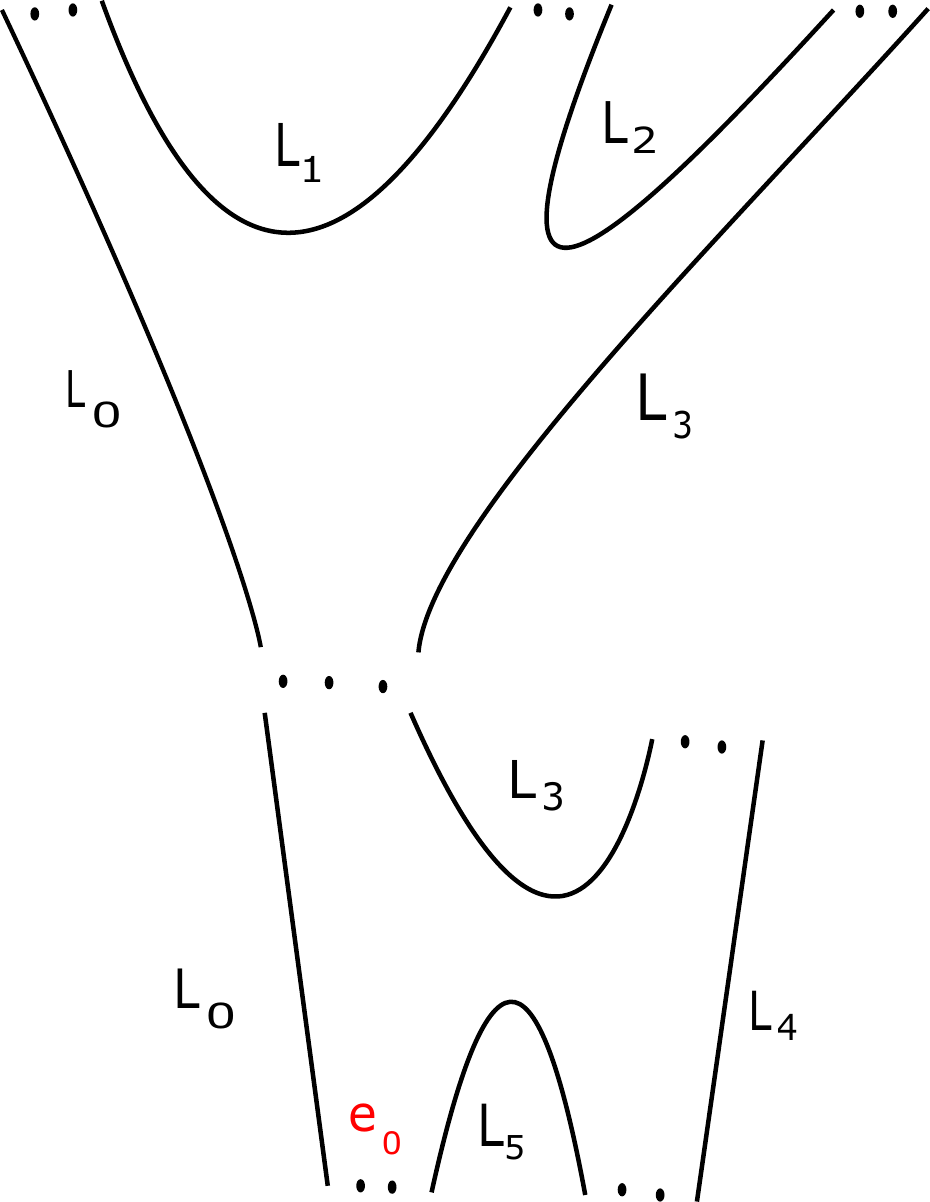}
 \caption {} \label{fig:label}
\end{figure} 
For the moment we do 
not specify any dependence of the labels on $r$.

Let us from now on omit the superscript $n$ in
$\Sigma ^{n}$  where there is no need to
disambiguate.
\begin{definition}
   \label{def:admissibleConnection} We say that a
   Hamiltonian connection (cf. Section \ref{sec:HamiltonianBundles}) $ \mathcal {A}$ on the
   Hamiltonian fibration $$
M \hookrightarrow \widetilde{ \mathcal {S}} (m_1, \ldots, m_d, \Sigma,r) \to \mathcal {S}
   _{r},$$ $r \in \overline{ \mathcal{R}} _{d}$,  is \emph { \textbf{admissible}} with respect
   to a labeling $L _{0}, \ldots, L
_{d}$ if:
\begin{itemize}
  \item  Parallel transport by $ \mathcal {A}$ over the boundary component(s) of $\partial \mathcal{S} _{r}$ labeled $ L _{i}$ preserves the Lagrangian $L _{i}$, $0 \leq i \leq d$, in the distinguished coordinates.  

\item  For $1 \leq i \leq d $,  in the distinguished coordinates
$(0, \infty) \times \overline{m} _{i}  ^{*} P$,
at each $e _{i}$ end, $$\mathcal{A} = \widetilde{\mathcal{A}} (L _{i-1}, L _{i}, \overline{m} _{i}  ) := pr ^{*} \mathcal {A} (L _{i-1}, L _{i}, \overline{m} _{i}  ) $$ for $$pr: (0,
\infty) \times \overline{m} _{i}  ^{*} P \to
\overline{m} _{i}  ^{*} P $$ the natural
bundle map projection.   
Here   $\mathcal {A} (L _{i-1}, L
_{i}, \overline{m} _{i}  )$ are part of our
data $\mathcal{D} $ as previously discussed. 
\item  Likewise, at the $e _{0}$  end in the
distinguished coordinates
$(-\infty,   0) \times \overline{m} _{0}
^{*} P$,  $\mathcal{A} = pr ^{*} \mathcal
{A} (L _{0}, L
_{d}, \overline{m} _{0}  ) $ for $pr$
similarly defined.   
\end{itemize}
\end{definition} 
We do not yet impose any conditions at the nodes,
but certain conditions will be forced by
the additional 
properties of the Definition \ref{def:natural}
below. For a preview, we remark that $\mathcal{D} $ will
make a choice of such a connection for all
possible $L _{0}, \ldots, L _{d}$   as above.
  
We also have a Lagrangian sub-fibration 
of $$
\widetilde{ \mathcal {S}} _{r}  \to \mathcal {S}
_{r}$$
over the boundary of
$\mathcal{S}_{r} $,
whose fiber over an element of
the boundary component labeled $L _{i} $ is $L
_{i} $, in the distinguished
coordinates. (This naturally
extends to the case $\mathcal{S} _{r}$ is nodal.)

We name this sub-fibration by
\begin{equation} \label{eq:subfib}
{\mathcal{L}}  (\mathcal{U},  L _{0}, \ldots , L _{d}, r  ).
\end{equation} 
In particular, by construction, if $\mathcal{A}$ is
admissible with respect to $L _{0}, \ldots, L
_{d}$ as above then it preserves this sub-fibration.
\begin{notation}
   \label{} 
Denote by $$ \mathcal {T} (m _{1}, \ldots, m _{d};L _{0}, \ldots, L _{d},
\Sigma,r) = \mathcal {T} (L _{0}, \ldots, L _{d},
\Sigma,r)$$ the space of
Hamiltonian connections on $$ M \hookrightarrow  \widetilde{ \mathcal {S}} (m_1, \ldots, m_d,
\Sigma,r) \to \mathcal{S} _{r}$$
admissible with respect to $L _{0}, \ldots, L _{d}$. Note that this implicitly requires a chosen system $\mathcal{U} (X) $, which will not be indicated.
\end{notation}
   
\subsection {Gluing admissible connections.}
\label{sec:gluingconnections} 
Given an element $$ \mathcal {A} \in \mathcal {T}
(m _{1}, \ldots, m _{s _{1}}; L_0,
\ldots, L _{s_1} , \Sigma,r)$$ and an element $$ \mathcal {A}' \in \mathcal {T}
(m _{1}, \ldots, m _{s _{2}};L'_0, \ldots, L' _{i-1}, L' _{i} , L' _{i+1}, \ldots, L'
_{s_2}, \Sigma, r'),$$  s.t. $m_1 \cdot \ldots \cdot m _{s_1}=m' _{i}$ 
and s.t. 
$L' _{i-1}= L _{0}, L
' _{i}= L _{s _{1} }     $, we have a naturally
induced glued connection $(\mathcal{A} \star _{i}
\mathcal{A}') _{\tau}$ 
on $$\widetilde{\mathcal{S} } _{r,r', \tau}:= (u \star _{i}
u') _{r,r', \tau} ^{*} \circ \Sigma ^{*} P, $$
where $(u \star _{i}  u') _{r,r', \tau}$  is as in
\eqref{eq:ustaru'taur}.  The construction of
the connection $(\mathcal{A} \star _{i}
\mathcal{A}') _{\tau}$  is analogous to the construction of the maps
$(u \star _{i}  u') _{r,r', \tau}$, see also
Figure \ref{fig:gluedconnection}.
\begin{figure}[h]
\includegraphics[width=2in]{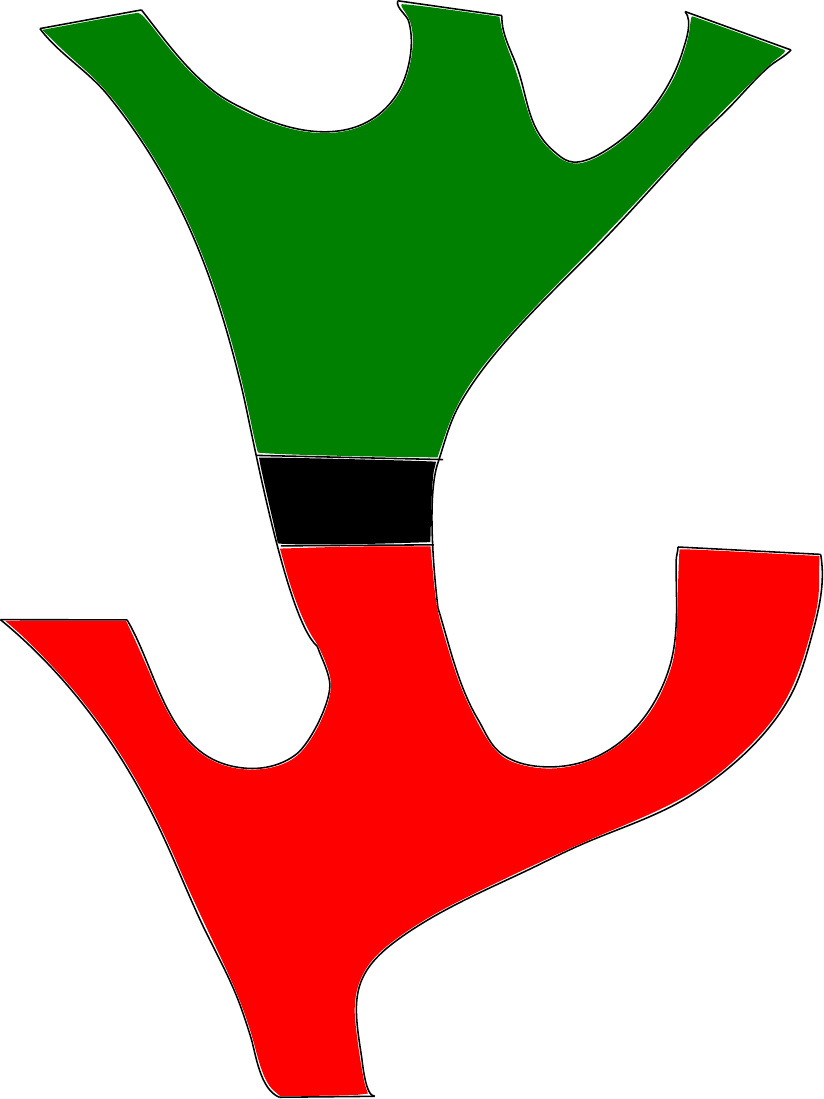}
 \caption{The green region is identified with a
   subregion of the surface $\mathcal{S} _{r}$,
   the red region is identified with a
   subregion of the surface $\mathcal{S} _{r'}$. 
   We have similar identifications  of the
   fibration  $\widetilde{\mathcal{S} } _{r,r',
   \tau}$ over these regions with (sub-fibrations
   of) of $\widetilde{\mathcal{S} } _{r},
   \widetilde{\mathcal{S} } _{r'} $.  
Then with respect to this identification $(\mathcal{A} \star _{i}
\mathcal{A}') _{\tau}$ 
   over the green region is the connection
   $\mathcal{A} $, and over the red region it is the connection
   $\mathcal{A}' $. Over the black region $(\mathcal{A} \star _{i}
\mathcal{A}') _{\tau}$ 
  is the connection $pr ^{*} \mathcal {A} (L
_{i-1}, L _{i}, \overline{m}' _{i}  )$, discussed
   ahead.} \label{fig:gluedconnection}
\end{figure} 
The pair $\mathcal{A}, \mathcal{A}'$ as above will be called \emph{composable}. 
Thus, applying Axiom 1 of naturality
of $\mathcal{U} (X) $, for a composable pair $\mathcal{A}, \mathcal{A}'$
as above we get induced connections $\{St _{i}
(\mathcal {A}, \mathcal {A}', \tau)\} _{0 \leq \tau <1} $, so that $$St _{i} (\mathcal {A}, \mathcal {A}', \tau) \in \mathcal {T} (L'_0, \ldots, L'
_{i-2}, L_0, \ldots, L _{s_1}, L' _{i+2}, \ldots, L'
_{s_2}, \Sigma, St _{i} (r, r', \tau)),$$
for $0 \leq \tau < 1$, and so that in addition we have the following. 

Over the thin region $thin _{\tau,i} \subset \mathcal{S} _{St _{i} (r,r', \tau) } $, for $\tau>0$,
$\widetilde{\mathcal{S}}_ {St _{i} (r,r', \tau) }
$ is naturally isomorphic to the fibration $$ (-\phi (\tau), \phi
(\tau)) \times (\overline{m}' _{i})
^{*} P \to (-\phi (\tau), \phi
(\tau)) \times [0,1] $$ by
Axiom \ref{axiom:partial2}, \ref{axiom:partial3} of
partial naturality, and Axiom \ref{axiom:1} 
naturality. Here $\phi$  is as in
\eqref{eq:phiearly}. We likewise call this
isomorphism the \emph{distinguished coordinates/representation}
extending the previous use of this term. 
Then over $thin _{\tau,i} $, in the above
  distinguished representation, $St _{i} (\mathcal {A}, \mathcal
{A}', \tau)$ is the connection $pr ^{*} \mathcal {A} (L
_{i-1}, L _{i}, \overline{m}' _{i}  )$, where
$$pr:  (-\phi (\tau), \phi
(\tau)) \times (\overline{m}' _{i})
^{*} P \to (\overline{m}' _{i})
^{*} P$$ is the natural projection.

\subsection {Admissible fiber almost complex structures} 
\begin{definition} We say that a family $\{j _{z} \}$ of fiber-wise, $\{\omega
_{z} \} $-compatible, almost complex
structures on the Hamiltonian $M$-fibration $
\widetilde{ \mathcal {S}} (m_1, \ldots, m_d, \Sigma,r) \to \mathcal {S}
_{r}$ is \emph { \textbf{admissible}} with respect to $L _{0}, \ldots, L
_{d}$ if:
\begin{itemize}
\item At the $i$'th end of $ \mathcal {S} _{r}$, $1 \leq
i \leq d$ in the distinguished trivialization $$(0, \infty)
\times \overline{m} _{i}  ^{*} P  \to \widetilde{ \mathcal
{S}}$$ we have $$\{j _{z} \} = \widetilde{j} (L _{i-1}, L _{i})
:=  pr ^{*} j (L _{i-1}, L _{i}, \overline{m} _{i})$$ for   
$$pr: (0, \infty) \times \overline{m} _{i}  ^{*} P \to \overline{m} _{i}  ^{*} P$$ the projection. 
Here $ {j}  (L _{i-1}, L _{i},
\overline{m} _{i}  )$, is as
in Definition \ref{def:jtadmissible}, and is part of our data $\mathcal{D} $  as previously discussed.
\item At the $e _{0}$ end, define admissibility analogously. 
\end{itemize}
\end{definition}
\subsection {Gluing admissible fiber almost
complex structures} \label{sec:gluingalmostcomplex}
Denote by $$ \mathcal {J} (L _{0}, \ldots, L _{s},
\Sigma, r)$$ the space of
families of fiberwise almost complex structures $\{j _{z} \}$ on $$ \widetilde{ \mathcal {S}} (m_1, \ldots, m_s,
\Sigma,r)$$
admissible with respect to $L _{0}, \ldots, L _{s}$.

Given an element $ \{j _{z} \}$ in $ \mathcal {J} (L_0,
\ldots, L _{s_1} , \Sigma,r)$ and an element $$ \{ {j}' _{z} \}  \in \mathcal {J}
(L'_0, \ldots, L' _{i-2}, L_0, L _{s_1}, L' _{i+1}, \ldots, L'
_{s_2}, \Sigma, r'),$$ 
  the pair $ \{j _{z} \}, \{j' _{z} \}$ will be called \emph{composable}.
For such a composable pair, analogously to the
definition of $St _{i} (\mathcal {A}, \mathcal
{A}', \tau)$,  we have an induced 
element:  $$St _{i} (
\{j _{z} \}, \{j' _{z} \}, \tau) \in \mathcal {J} (L'_0, \ldots, L'
_{i-2}, L_0, \ldots, L _{s_1}, L' _{i+2}, \ldots, L'
_{s_2}, \Sigma, St _{i} (r, r', \tau)),$$ for each $ 0 \leq \tau <1$.

\subsection{Combining admissible connections and
fiber almost complex structures}
\label{sec:combineConnectionsAlmostComplex}
\begin{definition} \label{def:natural} Let $M
   \hookrightarrow P \to X$  be as above.
 A \textbf{system} $\mathcal{F} = \mathcal{F} (P)
   $ of connections, and almost complex structures
   \textbf{\emph{relative}} to   
 a system $\mathcal{U} (X) $ is an element of  
 \begin{align*}
\prod _{\Sigma \in Simp (X) } \prod _{s \geq 2}  
\prod _{\{(L_0, \ldots , L_s) \vert L _{i} \in
    F (\Sigma) \}} \prod _{r \in
   \overline{ \mathcal{R}}_{s}}    \mathcal{T} (L _{0}, \ldots , L _{s}, \Sigma, r ) \times  \mathcal{J} (L _{0}, \ldots , L _{s}, \Sigma, r ),
\end{align*} 
(the system $\mathcal{U} (X) $ is implicit in the
   above.) 
 The projection of $\mathcal{F}$ onto the $(\Sigma
   ,s,  (L_0, \ldots , L_s), r
)$
component  will be
denoted by $\mathcal{F}   (L_0, \ldots , L_s,
   \Sigma,r  )$.   To phrase this   in functional
   language, 
let 
\begin{align*}
         O= \{(L_0, \ldots
  L _{s}, \Sigma, r) | s \in
\mathbb{N} _{\geq 2},  \Sigma
   \in Simp (X), 
    \\ L _{i}
   \in   F (\Sigma),  
   r \in \overline{\mathcal{R} } _{s}
       \},
   \end{align*}   
then set theoretically the product
above is the set of certain 
maps $\mathcal{F}$ with domain $O$.   Then in
this language $\mathcal{F}  (L_0, \ldots, 
L_s, n,r)$ is just the value $\mathcal{F} (L_0, \ldots, 
L_s, n,r), $  of the map $\mathcal{F} $. 
\end{definition}
Let $pr _{i}$, $i=1,2$, denote the projections 
\begin{align*}
   & pr_1: \mathcal{T} (L _{0}, \ldots , L _{s},
   \Sigma, r ) \times  \mathcal{J} (L _{0}, \ldots
   , L _{s}, \Sigma, r ) \to  \mathcal{T} (L _{0},
   \ldots , L _{s}, 
   \Sigma, r ) \\
    & pr_2: \mathcal{T} (L _{0}, \ldots , L _{s},
   \Sigma, r ) \times  \mathcal{J} (L _{0}, \ldots
   , L _{s}, \Sigma, r ) \to  \mathcal{J} (L _{0}, \ldots
   , L _{s}, \Sigma, r ).  
\end{align*}      
 For shorthand, in what follows, we say that a Hamiltonian connection $\mathcal{A}\in
\mathcal{F}$  if it is of the form $pr _{1}   \mathcal{F} (L_0, \ldots , L_s, \Sigma, r
)$, for some $(L _{0}, \ldots, L _{s}, \Sigma, r)
   $.

\begin{definition} \label{def:naturalF}
We say that $\mathcal{F}$, relative to a natural
   $\mathcal{U} (X) $, 
is \textbf{\emph{natural}} if:
\begin{enumerate}
   \item \label{axiom:connection1} The families of connections/almost complex structures are smooth in the
parameter $r$,  over smooth components   of
the surfaces $\mathcal{S} _{r}$. (At the nodes the behavior will be characterized the Axioms \ref{axiom:connection2}, \ref{axiom:connection3}, below.) 
\item  \label{axiom:connection2} For a composable pair $\mathcal{A}, \mathcal{A}' \in \mathcal{F}$ as above  the connection $St _{i} ( \mathcal
{A}, \mathcal {A}', 0)$ coincides with $$ pr _{1}  \mathcal {F}  (L'_0, \ldots, L'
_{i-2}, L_0, \ldots, L _{s_1}, L' _{i+1}, \ldots, L'
_{s_2}, \Sigma, St _{i} (r, r',0)).$$ 
  \item  \label{axiom:connection3}
   The pair of connections,  $$St _{i} ( \mathcal
{A}, \mathcal {A}', \tau),$$  $$ pr _{1}  \mathcal {F}  (L'_0, \ldots, L'
_{i-2}, L_0, \ldots, L _{s_1}, L' _{i+1}, \ldots, L'
_{s_2}, \Sigma, St _{i} (r, r', \tau)),$$
   also agree for all
$0 < \tau < 1$ on the ``thin region'' $thin
      _{\tau,i}$ of $ \mathcal {S}
_{ St _{i} (r, r',\tau)}$. 
\item  \label{property:natfacemap}
Given a morphism in $Simp (X) $,  $f: \Sigma
 _{1} ^{n} \to \Sigma_{2} ^{m}$,  by Axiom
   2 of naturality of $\mathcal{U} (X) $, 
the Hamiltonian bundle $\widetilde{
\mathcal {S}} (m_1, \ldots, m _{d}, \Sigma ^{n}_{1},
r) $ is expressed as a certain pull-back  of 
$\widetilde{ \mathcal {S}} (f(m_1), \ldots, f(m
   _{d}), \Sigma _{2} ^{m},r)$, where $f$ on the right denotes the corresponding
   simplicial map $f: \Delta^{n} \to \Delta^{m}
   $. 
So that there is a natural bundle map of  Hamiltonian $M$-fibrations $$p: \widetilde{
\mathcal {S}} (m_1, \ldots, m _{d}, \Sigma _{1}
^{n},  r) \to
\widetilde{ \mathcal {S}} (f(m_1), \ldots, f(m
   _{d}), \Sigma _{2} ^{m},r), $$  preserving the
   distinguished trivalization structure.
Then we ask that $$p ^{*}pr _{1} \mathcal {F} (L_0, \ldots, 
L_d, \Sigma _{2} ^{m},r) = pr _{1} \mathcal {F}
   (L_0, \ldots, L _{d}, \Sigma ^{n}_{1}, r).$$
\item There are analogous conditions on the families of almost complex
structures $pr _{2}   \mathcal{F} (L_0, \ldots , L_s, \Sigma,
r
)$ that we will not state.
\end{enumerate}
\begin{notation}
   We will sometimes write by abuse of notation $\mathcal{F} ( \ldots )$, for either the connection
 $pr _{1}  \mathcal{F} ( \ldots ) $, or the family of almost complex
 structures  $pr _{2}  \mathcal{F} ( \ldots ) $, since there usually can be no
 confusion.

\end{notation}
\end{definition}
 \begin{theorem} \label{lemma:naturalF}  A natural 
   system $\mathcal{F} $ relative to   
   any given natural system $\mathcal{U} (X)  $ exists.
\end{theorem}
\begin{proof}  

Restricting to a single $\Sigma: \Delta ^{0} \to X$  this is the classical Fukaya category case, and the
proof of existence of a natural system  is given in Seidel~\cite[Section
9i]{citeSeidelFukayacategoriesandPicard-Lefschetztheory} in the
language of what Seidel calls compatible system of
perturbations, which is completely analogous to
   the language  of connections used here. Although in Seidel's book only
the case of exact Lagrangians in exact
symplectic manifolds is considered, this readily
extends to our context, since we are not yet
concerned with compactness or regularity
 properties.

   In what follows, as usual, we write $\Sigma
   ^{n}$ for a degree $n$ 
element of $Simp (X) $, i.e. of the form $\Sigma
^{n}: \Delta^{n} \to X$.  We proceed by induction. 
Let $S (N) $ be the statement: 
there is an element 
\begin{align*}
 \mathcal{F} ^{N} \in \prod _{r \in
   \overline{ \mathcal{R}}_{s} }     \prod _{\{(L_0, \ldots , L_s) \vert L _{i} \in
    F (\Sigma ^{n} )\}}  \prod _{s \geq 2} \prod
   _{\{\Sigma ^{n} \vert n \leq
   N\}}  
    \mathcal{T} (L _{0}, \ldots , L _{s}, \Sigma
^{n}, r ) \times  \mathcal{J} (L _{0}, \ldots , L _{s}, \Sigma
^{n}, r )
\end{align*}
satisfying naturality condition of Definition
   \ref{def:naturalF}, where the fourth  axiom is
   only required to hold on $Simp ^{N} (X) $, the
   latter
   denoting the subcategory of simplices of degree
   up to $N$. $S (N) $ will also denote the
   corresponding system $\mathcal{F} ^{N}$.   
 
$S (0) $ is already explained above. We prove $$S (N) \implies S (N+1), $$ in addition
the corresponding system
$S (N+1) $ can be assumed to extend $S (N) $. 

Let $\Sigma ^{N+1}: \Delta^{N+1} \to X$ be given. 
Let $L _{0}, \ldots,  L _{s} \in F (\Sigma ^{N+1})
$, so that each $L _{i} \in F (x _{i}) $  for $x
_{i} = \Sigma ^{N+1} (v _{i}) $ for some vertex $v _{i}
\in \Delta^{N+1} $. 
In particular the set $\{L _{0}, \ldots,  L _{s}
\}$ determines the set of 
vertices $\{v _{0}, \ldots, v _{s} \}$ of
$\Delta^{N+1} $.    
Denote by $D (L_0, \ldots, L_s )$ the least dimension of a subsimplex of $\Delta
^{N+1} $ with vertices $\{v _{0}, \ldots, v _{s}\}$. 
Clearly, there is a unique extension of $\mathcal{F}$ to an element 
\begin{equation} \label{eqInducedF}
\begin{split}
 \mathcal{F} \in \prod _{r \in
   \overline{\mathcal{R}}_{s} } \prod _{\{(L_0, \ldots , L_s) \mid \, N
\geq D (L _{0} , \ldots, L_s)}  \prod _{s \geq 2} \prod _{\{\Sigma ^{n} \mid  n
\leq N +1\}}  \\
    \mathcal{T} (L _{0}, \ldots , L _{s}, \Sigma
^{n}, r ) \times  \mathcal{J} (L _{0}, \ldots , L _{s}, \Sigma
^{n}, r )
\end{split}
\end{equation}
satisfying the naturality condition.

We need to extend to the case $N+1
= D (L _{0} , \ldots, L_s)$ and 
so that naturality is satisfied.
For all $ (L_0 , \ldots, L_s)$ with $$D (L_0,\ldots, L_s) = N+1,$$ and given $\Sigma ^{N+1} $,  the naturality condition and $\mathcal{F}$ from \eqref{eqInducedF} determine $$ \mathcal {F} (L_0, \ldots, L _{s}, \Sigma^{N+1},r)$$ for
 $r$ in the boundary of $ \overline{ \mathcal {R}} _{s}$, see the discussion following \eqref{eq:subs}.

Set
$$\mathcal{P} := \bigcup _{r \in \overline{
   \mathcal{R}}
_{s}} \mathcal{T} (L _{0}, \ldots , L _{s}, \Sigma
^{n}, r ) \times  \mathcal{J} (L _{0}, \ldots , L _{s}, \Sigma
^{n}, r ). $$   So we have a natural fibration   $\mathcal{P} \to
\overline{\mathcal{R}} _{s}  $, with the fiber
over $r \in \overline{\mathcal{R}} _{s} $  denoted
by $\mathcal{P} _{r}$. 
The topology on $\mathcal{P} $ is the natural
metric ``Gromov topology'',
constructed using gluing operations of Sections
\ref{sec:gluingconnections},
\ref{sec:gluingalmostcomplex}.    We will only
describe this briefly. First,
constructing  
 a metric $d$  on
$\mathcal{P}| _{\mathcal{R} _{s}}$  can be reduced
to constructing a metric on the spaces of
connections/almost complex structures on a fixed
Hamiltonian fibration $M \hookrightarrow
\widetilde{S} \to S$, as $\mathcal{R} _{s}$
is contractible.  In other words it is enough
to construct $d$ on the fiber $\mathcal{P}
_{r}$, $r \in \mathcal{R} _{s}$.  
Since   $\mathcal{P} _{r}$  is naturally a Frechet
manifold we just suppose that $d$ on $\mathcal{P}
_{r}$ is the metric
inducing the corresponding topology. 
Given $\mathcal{S} _{r}$, $r \in \partial
\overline{\mathcal{R}} _{s} $  there is,
corresponding to each gluing parameter $0 <\tau
\leq 1$,  a ``glued'' non-nodal surface $$gl _{\tau}
(\mathcal{S} _{r}) \simeq \mathcal{S} _{gl _{\tau}
(r) \in \mathcal{R} _{s}}, \quad \text{$\simeq$ being
holomorphic isomorphism.}  $$   In other words
we glue at each node of $\mathcal{S} _{r}$ with
gluing 
parameter $\tau$.
Similarly, using the gluing
operations of Sections
\ref{sec:gluingconnections},
\ref{sec:gluingalmostcomplex}, given $e \in
\mathcal{P} _{r}$, $r \in \mathcal{R} _{s}$  there is
an element $gl
_{\tau} (e) \in \mathcal{P} _{gl _{\tau} (r)} $. 
Now, for $r _{1} \in \partial
\overline{\mathcal{R}} _{s} $,    $r _{2} \in
{\mathcal{R}} _{s} $  and for elements $e _{1} \in
\mathcal{P} _{r _{1}}, e _{2} \in \mathcal{P}
_{r _{2}}$  we define:
\begin{equation*}
   d (e _{1}, e _{2}) := \lim _{\tau \mapsto
   0} d (gl _{\tau} (e _{1}), e _{2}).  
\end{equation*}
Define this similarly in the case $r _{1}, r _{2}
\in \partial \overline{\mathcal{R} } _{s} $.

The fibers
of $\mathcal{P} $ are non-empty, the corresponding statement for
just connections follows by \cite[Lemma
3.2]{citeAkveldSalamonLoopsofLagrangiansubmanifoldsandpseudoholomorphicdiscs}.
The fibers are contractible, for the connection
component this is just because the relevant space
is naturally affine. For the almost complex
structure component, this is basically classical
by work of Gromov ~\cite{citeGromovPseudo}.
Moreover,   $\mathcal{P} $  is a Serre fibration,
this is only non-obvious at boundary points of
$\overline{\mathcal{R}} _{s}$, but there the
corresponding lifting property for cubes can be
easily verified
directly, again using the gluing operations of
Sections  \ref{sec:gluingconnections},
\ref{sec:gluingalmostcomplex}.

To summarize we have a Serre fibration
$\mathcal{P} \to \overline{\mathcal{R}} _{s}$ with non-empty
contractible fibers.   We have a section of $\mathcal{P} $ over
$\partial \overline{\mathcal{R}} _{s}$ 
corresponding to the partially constructed family $$ \{ \mathcal {F} (L_0, \ldots, L _{s},
 \Sigma^{N+1},r) \} _{r \in
 \partial {\overline{\mathcal{R}}} _{s}}$$ above.
By the classical obstruction theory, there
is an extension to a section $\zeta$   over $\overline{
   \mathcal {R}} _{s}$. We may need to
homotopically adjust the
section $\zeta$  to satisfy the Axiom \ref{axiom:connection3} of naturality, but
this is straightforward. So that this completes
the proof of the inductive step.  

By recursion, we may then define a sequence of
systems $\{S _{N}\} _{N \geq 0}$,  so that $S (N+1)  $
extends $S (N) $, for each $N$, 
we then
set $\mathcal{F} := \bigcup_{N} S (N). $ 
And this completes the proof.

\end{proof}
\subsection {The summary of the perturbation data
$\mathcal{D} (P)$}   \label{sec:summaryperturbation}
Let $M \hookrightarrow P \to X$ be as above. To summarize, the
perturbation data $\mathcal{D} = \mathcal{D} (P)  $ consists of a 
choice of a natural system
 $\mathcal{U} (X)$, and a choice of a natural system
$\mathcal{F} = \mathcal{F} (P)   $ of connections/almost complex
structures relative to $\mathcal{U} (X)  $.   
\begin{theorem}
   \label{thm:concordanceMain}
Any pair $\mathcal{D} _{0} (P) , \mathcal{D} _{1}
   (P) $  of
perturbation data are concordant. Concordant
   means that
   there is data $\widetilde{\mathcal{D}} (I
   \times P)  $, for
   $P \times I$ the pull-back of $P$ by the
   projection  $X \times I \to X$, so that
   $\widetilde{\mathcal{D} } (I \times P)  $
   restricted over $X \times \{0\}$ is
   $\mathcal{D} _{0} $ and restricted over $X
   \times \{1\}$    is $\mathcal{D} _{1}$.
   (Interpreted naturally.) 
\end{theorem}
\begin{proof}
   Theorem \ref{thm:naturaltargetdependent} tells
   us that $\mathcal{U} _{0} (X), \, \mathcal{U} _{1}
   (X)   $ are concordant,  where the latter
   are the systems corresponding to
   $\mathcal{D} _{0}, \mathcal{D} _{1}$. Let
   $\widetilde{\mathcal{U}} (X \times I)  $    denote the
   corresponding system. 
The proof of Theorem \ref{lemma:naturalF} then
   readily gives a natural system
   $\widetilde{\mathcal{F}} (P \times I) $, relative to
   $\widetilde{\mathcal{U} } (X \times I)  $,
   restricting to  $\mathcal{F}
   _{1} (P), $ on $P
   \times \{0\}$, respectively to $\mathcal{F}
   _{2} (P)$ on $P \times \{1\}$.
   Here $\mathcal{F}
   _{1} (P), \, \mathcal{F} _{2} (P) $  correspond
   to $\mathcal{D} _{1} (P) , \mathcal{D} _{2} (P) $. 
\end{proof}

\section {The functor $F$} \label{section:functor} 
 Let $A_{\infty}-Cat$ denote the
 category of small $\mathbb{Z} _{2} $ graded $A _{\infty}$ categories over
 $\mathbb{Q}$,  with morphisms fully-faithful
 embeddings, as defined below, that are in addition quasi-equivalences. 

\begin{definition} \label{defFullyfaithful} We say that an $A _{\infty} $
functor $G$ is a
\emph {\textbf{fully-faithful embedding}}, if $G$ has vanishing higher order components, is
injective on objects and if the first component map on hom spaces is an isomorphism
of chain complexes. In other words $G$ above is just an identification map of
a full $A _{\infty} $ sub-category.
\end{definition}

 We now describe the construction of the functor $$F _{P, \mathcal {D}}: Simp(X) \to A_{\infty}-Cat$$
 associated to a Hamiltonian fibration $P$  and the chosen data $
 \mathcal {D}$, described in the previous section.
 In what follows we usually drop $ \mathcal {D}$
 and $P$ from notation. For a point $x: pt \to X $
 the associated category will be constructed
 following Sheridan~\cite{citeNickSheridanOntheFukayaCategory}.
 In fact the analysis does not change for the case
 of higher dimensional simplices $\Delta^{n} \to
 X$, the geometry however needs to be
 substantially generalized.

\subsection {$F$ on a point} \label{section:Fonpoint} 
For $x: pt \to X$, $F (x)$ is defined to be a
certain 
Fukaya type $A _{\infty}$ category, 
whose set of objects is the set $F (x) $ discussed
in Section \ref{section:dataD}, cf. \eqref{eq:Fx}.

For a pair $L _{0}, L _{1} \in F (x) $, with $\omega (L _{0} ) \neq \omega (L
_{1} )$ we set $hom (L _{0}, L _{1}   ) =0$, (to avoid dealing with curved $A
_{\infty} $ categories), otherwise we set $$hom
(L_0, L_1) = CF ({L _{0}, L
_{1}}, \mathcal{D}),$$ where the latter is a $\mathbb{Z} _{2} $ graded Floer chain complex
over $ \mathbb{Q}$ that is defined as follows.

Let $ \mathcal {A} ({L_0, L_1})$ be the Hamiltonian connection on
$P _{x} \times [0,1]$ determined by the chosen data $\mathcal{D}$, and likewise let $j (L
_{0}, L _{1}  )$ to be the family of almost complex structures determined by
$\mathcal{D}$.

Then $CF ({L _{0}, L
_{1}}, \mathcal{D})$
is the vector space over $\mathbb{Q} $,
freely generated
by elements of $S (L _{0}, L _{1}) $, where the
latter is as in Definition \ref{notation:SL0L1}.   
To quickly recall, $S (L _{0}, L _{1}) $ is the
space of  $\mathcal {A} ({L_0, L_1})$-flat
sections $\gamma$ of $P _{x} \times [0,1]$, with boundary on the pair of Langrangians $$L_0 \subset P _{x} \times
\{0\},  L_1 \subset P _{x} \times \{1\}.$$  
  
These $\gamma$ are called \textbf{\emph{geometric
generators}}. To relate this with more classical
Lagrangian Floer homology generators,  we point out that
there is a natural set isomorphism:
\begin{equation*}
  \phi:   S
(L _{0}, L _{1}) \to  (\mathcal{A} (L _{0}, L _{1}) L _{0}) \cap L
_{1}
\end{equation*}
where $\mathcal{A} (L _{0}, L _{1}) L
_{0}$ is as in the paragraph prior to the
Definition \ref{notation:SL0L1}.     
The map $\phi$ is given by $$\phi (\gamma) =
\gamma (1).  $$ 

Then the $\mathbb{Z} _{2} $ grading of a generator
$\gamma \in S (L _{0}, L _{1}) $ is given 
by the sign of the intersection point $\phi
(\gamma ) $. 
 \subsubsection {Differential on $CF ({L _{0}, L
_{1}}, \mathcal{D})$}
For $\gamma _{0}, \gamma _{1}$ geometric generators of
$$ CF (L _{0}, L _{1}, \mathcal{D} ),  $$ 
let $ {\mathcal{M}} (\gamma _{0}; \gamma _{1}  )$
denote the space of holomorphic (to be further
explained)  sections of $$ ([0,1] \times
\mathbb{R}) \times P
_{x} \to [0,1]
\times \mathbb{R},$$  with boundary on the
Lagrangian sub-bundles $$\{0\} \times \mathbb{R} \times L_0
  \to \mathbb{R},  \{1\} \times
\mathbb{R} \times L
_{1} \to \mathbb{R} ^{},
$$ and asymptotic to $\gamma _{0} $,
respectively to $\gamma _{1}$,  at the $\infty$, respectively $- \infty$ ends. 
Here, {\emph{asymptotic}}  means that $$\lim _{s
   \mapsto \infty}  {\sigma| _{[0,1] \times
   \{s\} }} = \gamma _{0}, $$ 
and $$\lim _{s
   \mapsto -\infty}  {\sigma| _{[0,1] \times
   \{s\} }} = \gamma _{1}, $$ 
where the limit
   is $C ^{\infty}$ limit.
And let $\overline {\mathcal{M}} (\gamma _{0}; \gamma _{1}  )$
denote the natural Gromov-Floer compactification of the quotient ${\mathcal{M}} (\gamma _{0}; \gamma _{1}  )/\mathbb{R}$, where $\mathbb{R} $ acts by translation on the domain.

\begin{terminology}
   \label{terminology:holomorphic} Here and elsewhere  the term \emph {
\textbf{holomorphic section}} of various Hamiltonian fibrations over Riemann
surfaces $S$ will mean the following. Our Hamiltonian
   fibrations $\widetilde{S} \to S$ always come with choices
of 
 a Hamiltonian connection $\mathcal{A}$, and a family of
 fiber-wise almost complex structures $\{j _{z} \}
   _{z \in S}
 $,  determined by the perturbation data
 $\mathcal{D}$. This gives an induced almost complex structure
  $J ( \mathcal {A}, \{j _{z} \} )$ on $\widetilde{S} $  restricting to $\{j _{z} \}$ on the fibers,  having a holomorphic projection map to the base, and preserving the horizontal distribution of $ \mathcal {A}$.
Holomorphic then means that the section has $J (
   \mathcal {A}, \{j _{z} \} )$-complex linear
   differential. 
\end{terminology} 
In the above case,  let $\mathcal {A} (L _{0}, L _{1}, \overline{m}  _{0})$,
${j}  (L _{0}, L _{1}, \overline{m}  _{0})$ be part of our
data $\mathcal{D} $, where 
$\overline{m}  _{0} = x \circ m _{0}$,  
$m _{0}: [0,1]  \to \Delta^{0}
$, and so $\overline{m}  _{0}: [0,1] \to X$ is the constant map
to $x$.  
Then
``holomorphic'' is with respect
to the almost complex structure induced by: 
$$\widetilde{\mathcal{A}} (L _{0}, L _{1}) :=  pr
^{*} \mathcal {A} (L _{0}, L _{1}, \overline{m}
_{0}), \quad
\widetilde{j}  (L _{0}, L _{1}):=pr ^{*}j (L
_{0}, L _{1}, \overline{m}  _{0})$$ for   
$$pr: ([0,1] \times
\mathbb{R}) \times P
_{x} \to [0,1]
\times P _{x} $$ the projection.

For a generic pair  $\mathcal{A} (L_0, L _{1} ), j (L _{0}, L _{1}  )$,  all the  moduli spaces $\overline {\mathcal{M}} (\gamma _{0}; \gamma _{1}  )$ are transversely cut out for all $\gamma _{i} $,
~\cite{citeNickSheridanOntheFukayaCategory} but
these kinds of transversality results go much
further back,
see
for
example Oh~\cite{citeOhFloerCohLagrangianIntersections}.

 The differential   $$\mu ^{1}: CF ( L _{0}, L _{1},
 \mathcal{D} )  \to CF ( L _{0}, L _{1},
 \mathcal{D} ) $$
is  defined as usual by $$\mu ^{1} (\gamma _{i} ) = \sum _{i} \# 
\overline{\mathcal{M}} (\gamma _{i}  ; \gamma _{j}  ) \gamma _{j}   ,$$ 
for $\{\gamma _{i} \}$ a basis of geometric generators for $CF ( L _{0}, L _{1},
 \mathcal{D} )$. Here $\# 
\overline{\mathcal{M}} (\gamma _{i}  ; \gamma _{j}
)$ is defined to be zero, unless the virtual
dimension of $\overline{\mathcal{M}} (\gamma _{i}
; \gamma _{j}  )$ is zero and in that case it is the signed count of points. The sum is finite by the monotonicity condition.
\subsubsection {Section classes}
\label{sec:SectionClasses}
Let $M \hookrightarrow \widetilde{S} \to S$ be a Hamiltonian
fibration over a Riemann surface with boundary
and end structure $\{e _{i}\} _{i=0} ^{i=d}$.
Suppose we have distinguished 
trivializations $\widetilde{e}  _{i}: [0,1] \times (0, \infty)
\times M \to \widetilde{S} $,   over $e _{i}: [0,1] \times (0,\infty) \to S
$, $0< i \leq d$. And $\widetilde{e}  _{0}: [0,1]
\times (-\infty, 0)
\times M \to \widetilde{S} $,  over $e _{0}$.
And
suppose we have a Lagrangian sub-fibration
$\mathcal{L}$  over the
components of the boundary $\partial S$, analogous
to the sub-fibrations \eqref{eq:subfib}. Let
$\sigma $ be a section of $\widetilde{S} $, so
that $\sigma (\partial S) \in  \mathcal{L}  $.
Suppose in addition that $\sigma$ is continuous and is $C ^{0}$
asymptotic at each end, to a section $\widetilde{\sigma } $  
which is translation invariant in the $(0, \infty)
$  factor (in the distinguished trivialization).
Here asymptotic just means $C ^{0}$
convergence in the distinguished trivialization:
$$\lim _{s \to \infty} \sigma |
_{[0,1] \times \{s\} } = \gamma,$$ for $\gamma =
\widetilde{\sigma }| _{[0,1] \times s}
$. As $\widetilde{\sigma } $ is translation
invariant,  the right-hand side is well-defined.   
In this case, as may be apparent, we may define the homology
class of $\sigma $, relative to the boundary and
relative to the asymptotic constraints.  

The above
extends to the case ${S}$  is disconnected, of the form
$\mathcal{S} _{r} $ for $r \in \partial
\mathcal{R} _{s}$. In this case we ask that our
sections $\sigma $ also have matching asymptotic
constraints, at the corresponding nodal ends $n
_{j, \pm} $, see Section
\ref{sec:preliminariesRiemannSurfaces}.         
Given this, we may again define a relative homology
class of $\sigma $.
We will
not give exhaustive detail of this, as this is
very standard. We just have a language change,
instead of maps of surfaces to a manifold $M$, we have
sections of $M$-fibrations over surfaces.   Let us denote such
relative homology
classes by letters $A$. 

\subsubsection {Higher multiplication maps}
\label{sec:multiplicationMapsPoint}
The  multiplication maps
\begin{equation}  \label {eq.mult}
\begin{split} \mu ^{d}: hom  (L_0, L_1) \otimes hom  (L_1, L_2)
\otimes \ldots \\ \otimes hom (L _{d-1}, L _{d}) \to hom (L_0, L_d),
\end{split}
\end{equation}
$d>1$ are defined as follows.
\begin{notation}
   \label{notation:generators} In the rest of the
   paper we
   use the notation $\{\gamma ^{j} _{i}\} _{i \in
   I ^{j}} \in CF (L
   _{j-1}, L _{j}) $ for the basis of geometric
   generators. (The set $I ^{j}$ will usually be omitted
   from notation.) 
   So the superscript in this notation refers to the
   vector space.  Similarly, $\gamma ^{0} _{i} \in CF (L
   _{0}, L _{d})$ will likewise denote the
   generators.  If the subscript is not specified
   then we just mean general geometric generator.
\end{notation}

For generators $\gamma ^{j} \in CF (L
_{j-1}, L _{j}, \mathcal{D} ) $, $1 \leq j \leq d$, $\gamma ^{0} \in CF (L _{0}, L _{d}, \mathcal{D}) $,  we define the  moduli space 
\begin{equation} \label{eq:modulispacepoint}
\mathcal {M} (\{\gamma ^{j} \}; \gamma ^{0},   x,
   \mathcal{D},  A) 
\end{equation} 
as follows.
The  elements  are pairs $(\sigma, r) $, for
$\sigma $ a relative class
$A$ (to be further explained),   $\mathcal
{F}  (\{L _{i} \},  x, r)$-holomorphic (cf.
Terminology \ref{terminology:holomorphic}) section of
the trivial fibration
$$ {\mathcal{S}}_{r} \times P _{x} \to \mathcal{S}
_{r},
$$ 
where $r \in {\mathcal{R}} _{d}  $, $\mathcal{F} $
is the system determined by $\mathcal{D} $. 
And s.t. each pair $(\sigma,r) $ satisfies:
\begin{itemize}
\item 
   $\sigma (\partial \mathcal{S} _{r})  
   \subset {\mathcal{L}}  (\mathcal{U},  L _{0}, \ldots , L _{d}, r  ), $
 see \eqref{eq:subfib}.
 \item By assumptions, 
at each $e _{i}$ end of $ \mathcal {S} _{r}$, $i \neq
0$, in the distinguished coordinates $$ [0,1]
      \times  (0, \infty) \times (M \simeq P
_{x})   \to 
\widetilde{\mathcal {S}} _{r},$$ the data $\mathcal
{F}  (\{L _{j} \},  \Sigma, r)$ is $
\mathbb{R}$-translation invariant in the $(0,
\infty) $  factor. Then we ask that
$\sigma$ be asymptotic  to
   $\gamma ^{j} $. Here, \textbf{\emph{asymptotic}}  means that $$\lim _{s
   \mapsto \infty}  {\sigma| _{[0,1] \times
   \{s\} }} = \gamma ^{i}, $$ where the limit
   is $C ^{\infty}$ limit.       Likewise, in the distinguished coordinates 
 \begin{equation*}
      [0,1]
      \times  (-\infty, 0) \times (M \simeq P
_{x})  \to 
\widetilde{\mathcal {S}} _{r},
   \end{equation*} at the $e _{0}$  end,
we ask that $\sigma $ be
asymptotic to $\gamma ^{0} $.
\item The pair of the conditions  above mean that
$\sigma $ determines a relative homology class,
as in Section \ref{sec:SectionClasses}, and we ask
that all the $\sigma $  are in the same class $A$.
\end {itemize}

Given geometric generators $ \{\gamma
^{j} \in \hom _{F
(x)}(L _{j-1}, L _{j}  ) \} $, $1 \leq j \leq d$, $d
\geq 2$,  and a geometric generator $ \gamma ^{0}
\in \hom _{F
(x)}(L _{0}, L _{d}  ) $,
assuming that $\mathcal
{F}  (\{L _{j} \},  x,r)$ is regular 
we define $\mu ^{d} $ by duality as:
\begin{equation} \label{eq:mud1}
\langle \mu ^{d} (\gamma ^{1}, \ldots, \gamma ^{d}    ), \gamma ^{0}   \rangle = \sum _{A} \# \mathcal {M} (\gamma ^{1}, \ldots, \gamma ^{d}; \gamma ^{0},  x, \mathcal{D}, A),
\end{equation}
when the above moduli spaces  have dimension 0, for
$ \langle ,  \rangle $ the natural inner
product pairing induced by our choice of basis
(consisting of geometric generators).
The sum is finite by monotonicity.

 \subsubsection {Compactification  regularity, and associativity} \label{section:regularitycompactness} 
Given a certain dictionary,
the moduli spaces $$ \mathcal {M} (\{\gamma ^{i} \}; \gamma ^{0},   x, \mathcal{D},  A) 
 $$  are  identical to the moduli spaces in Sheridan
\cite{citeNickSheridanOntheFukayaCategory}, with
respect to a system, determined by
$\mathcal{F}$, of Hamiltonian perturbations.

To be more explicit, a Hamiltonian connection on a  trivial $M$ bundle over a surface
$S$ is the same as the data of a 1-form on $S$ with values in $C ^{\infty} _{0}
(M) $: smooth functions with mean 0. This is the same as the data of a Hamiltonian perturbation. So in our case we just have a language change, the reason for which will be obvious when we shall construct the value of $F$ on higher dimensional simplices of $X$.
Consequently the compactification and regularity story is word for word
identical to  Sheridan
\cite{citeNickSheridanOntheFukayaCategory}.
(Again, given the right dictionary.) 
We say a bit more about compactification.
The compactification
$$\overline{ \mathcal {M}} (\{\gamma ^{i} \}; \gamma ^{0},   x, \mathcal{D},  A) 
,$$ 
is obtained by allowing $r \in \overline{\mathcal{R}} _{d}$
and allowing broken holomorphic sections over
the disconnected surfaces $\mathcal{S} _{r}$,
$r \in \partial \overline{ \mathcal{R}} _{d}$.
(This is in addition to the usual stable map
compactification.)  
A broken holomorphic section of $\widetilde{\mathcal{S} } _{r} \to \mathcal{S} _{r} $ is a holomorphic
section over each smooth component of $\mathcal{S}
_{r}$, so that at the node ends $n _{j, \pm}$, the
corresponding sections are asymptotic to the same
geometric generator $\gamma $. (In the natural
bundle trivializations at the ends.) 

\subsubsection {$A _{\infty} $ associativity} 
The maps  $\mu ^{d} $ satisfy the $A
_{\infty}$-associativity equations (stated over $\mathbb{F} _{2} $ for
simplicity)
\begin{equation} \label {eq.Ainfty} \sum _{n,m} \mu ^{d-m+1}  (\gamma
^{1}, \ldots, \gamma ^{n}, \mu ^{m}  (\gamma ^{n+1}, \ldots, \gamma ^{m+n}), \gamma
^{n+m+1}, \ldots ,\gamma ^{d}) = 0, 
\end{equation}
This is shown as usual by considering the boundary of
the one dimensional moduli spaces, of the form:
$ \overline{ \mathcal {M}} (\{\gamma ^{i} \};
\gamma ^{0},   x, \mathcal{F}, A) $.
\subsection {$F$ on higher dimensional simplices}
Let $\Sigma:
\Delta ^{n} \to X$ be smooth. The
category $F (\Sigma)$  will have objects
$\bigsqcup _{i} \obj F ({x}_i)$, where $x _{i}: pt
\to X $ is as before, the composition of the
$i$th vertex inclusion $\Delta ^{0} \to \Delta ^{n}  $, with the map $\Sigma$.
We also write $x _{i} $ for $x _{i} (pt) \in X $.

Let $$m: [0,1] \to \Delta ^{n}$$ be the edge between $i,j$ corners of $\Delta ^{n} $ and set $$
\overline{m} = \Sigma \circ m.
$$ Given a pair of objects  $L _{0} \in F ({x
_{i} }) \subset F (\Sigma ), L _{1}   \in  F (x
_{j})  \subset F (\Sigma )  $, (including $i=j$) and given the Hamiltonian connection $$ \mathcal{A} (L _{0}, L
_{1} ) =
\mathcal {A} (L _{0}, L _{1}, \overline{m} )$$ on $
\overline{m} ^{*} P$, determined by $\mathcal{D}$, we define as before
$hom _{F (\Sigma)}(L _{0}, L _{1})$  to be the $\mathbb{Z} _{2} $ graded chain complex over $\mathbb{Q}$ generated by
the elements of $S (L _{0}, L _{1}, \mathcal {A}(L
_{0}, L _{1})) $, cf. Definition
\ref{notation:SL0L1}. The grading is defined as
before.

The differential $\mu ^{1}$ is defined identically to the
differential on morphism spaces of categories $Fuk(P _{x})$.
The only difference is that $\overline {m} ^{*} P $ may no longer be naturally
trivialized. 

This completely describes all objects and morphisms of $F (\Sigma)$.
We now need to describe
the $A _{\infty}$ structure. 
Given $\{L _{\rho (k)} \in F (x _{\rho (k)})\} _{k=0} ^{k=d}  $, $$\rho: \{0,
\ldots, d\} \to \{0, \ldots,n \}, $$ with $\omega (L _{\rho ({k}) } )=
\theta \in \mathbb{Z}$,
we  need to define the
higher composition maps
\begin{equation} \label {eq.comp} \mu ^{d} _{\Sigma}: hom (L _{\rho (0)} , L
   _{\rho (1)} ) \otimes
\ldots \otimes hom (L _{\rho (d-1)} , L_{\rho (d)}) \to hom (L _{\rho (0)} , L
_{\rho (d)} ).
\end{equation}

Note that by construction, to each morphism of $F (\Sigma)$ naturally corresponds
either an edge or a vertex of $\Delta ^{n}$, in either case we may naturally associate to
these a morphism in the groupoid $\Pi (\Delta ^{n})$.
The collection $ \{x_{\rho(k)}\}$ then clearly determines a composable chain $(m_1, \ldots, m_d)$
of morphisms in $\Pi (\Delta ^{n})$.

So let $\gamma ^{j} \in hom _{F (\Sigma ) } (L
_{\rho(j-1)}, L _{\rho(j)}) $, $1 \leq j \leq d$,
$\gamma ^{0} \in hom _{F (\Sigma )} (L _{\rho(0)}, L
_{\rho(d)}) $, be geometric generators. 
Given the map $$u (m_1, \ldots, m_d, \Sigma):
{\mathcal {S}} _{d} ^{\circ}  \to \Delta ^{n} ,$$   
\text{ that is part of a natural system
$\mathcal{U} (X) $ determined by $\mathcal{D} $ } and given the system $\mathcal{F}$ determined by $\mathcal{D}$,
we define the moduli space $$ { \mathcal {M}}
(\{\gamma ^{j} \}; \gamma ^{0},   \Sigma, 
\mathcal{D}, A) $$ analogously to \eqref{eq:modulispacepoint}. 
The elements of this moduli space are pairs 
 $(\sigma,r)$, $r \in {\mathcal{R}} _{d}  $,   and
 $\sigma$ a class $A$ (to be further explained),   
 $\mathcal
{F}  (\{L _{\rho (j) } \},  \Sigma,r)$-holomorphic section
of
$$ \widetilde{S} _{r} =  \widetilde{S} (m_1, \ldots,
m_d, \Sigma ,r) \to \mathcal{S} _{r}. $$  
In addition each pair $(\sigma,r) $  satisfies:
\begin {itemize}
 
\item $\sigma (\partial \mathcal{S} _{r})  
   \subset {\mathcal{L}}  (\mathcal{U},  L
   _{\rho(0)}, \ldots , L _{\rho(d)}, r  ) $,
 see \eqref{eq:subfib}.
Recall that the right-hand side is a sub-fibration of $\widetilde{S} _{r}   $ over the boundary of $\mathcal{S}_{r} $.
\item By assumptions, 
at the $i$'th end of $ \mathcal {S} _{r}$, $i \neq
0$, in the distinguished coordinates $$  (0,
\infty) \times \overline{m}
_{i} ^{*}P    \to 
\widetilde{\mathcal {S}} _{r},$$ the data $\mathcal
{F}  (\{L _{\rho (j) } \},  \Sigma, r)$ is $
\mathbb{R}$-translation invariant in the $(0,
\infty) $ factor.
Then we ask that
$\sigma$ be asymptotic  to
   $\gamma ^{j} $ a geometric generator of $$\hom
   _{F (\Sigma )}  (L _{\rho(j -1) }, L _{\rho(j)}
   ),$$ where asymptotic is as in   Section
   \ref{sec:multiplicationMapsPoint}.   Likewise, in the distinguished coordinates 
 \begin{equation*}
     (-\infty, 0) \times  \overline{m}
    _{0} ^{*}P   \to 
\widetilde{\mathcal {S}} _{r},
   \end{equation*}
we ask that $\sigma $ be
asymptotic to $\gamma ^{0} $ a geometric
generator of $\hom _{F (\Sigma )}  (L _{\rho(0)},
L _{\rho(d)})$.
\item The pair of the conditions  above mean that
$\sigma $ determines a relative homology class,
as in Section \ref{sec:SectionClasses}, and we ask
that all the $\sigma $  are in the same class $A$.
\end {itemize}   
\subsubsection {Compactness and regularity} We do not need to reinvent the wheel
proving compactness and regularity results for the above moduli spaces.
(Although it obviously works the same way.)
Instead pick a Hamiltonian trivialization of $$M \times \Delta ^{n} \xrightarrow{tr} \Sigma ^{*} P, $$
then using this our system $\mathcal
{F}$ can be made to correspond to a system
compatible perturbations, in the sense of 
Seidel~\cite[Section
9i]{citeSeidelFukayacategoriesandPicard-Lefschetztheory},
and
Sheridan~\cite{citeNickSheridanOntheFukayaCategory}. Since
as previously mentioned in Section
\ref{section:regularitycompactness}, for a trivial Hamiltonian
$M$-fibration over a surface the data of a
Hamiltonian connection (of the type that appears
in our context) is equivalent to the data
of a Hamiltonian perturbation.   
Consequently, compactness and
regularity works the same way as described in Section
\ref{section:regularitycompactness},
which is
based on
the work of Sheridan~\cite{citeNickSheridanOntheFukayaCategory}.
We do not give extensive detail as this is likely
fairly evident.

\subsubsection {Composition maps in the $A
_{\infty}$ category $F (\Sigma ) $ }
\label{section.degenerate}  For $\{L _{\rho (j)
}\}$, as above,  given geometric generators $
\gamma ^{j}  \in \hom _{F
(x)}(L _{\rho(j-1)}, L _{\rho(j)}  ) $, $1 \leq j \leq d$, $d
\geq 2$, and a geometric generator  
$ \gamma ^{0}  \in \hom _{F
(x)}(L _{\rho(0)}, L _{\rho(d)}  ) $,
assuming that $\mathcal
{F}  (\{L _{\rho(j)} \},  \Sigma,r)$ is regular 
we define $\mu ^{d} _{F (\Sigma ) } (\gamma
^{1}, \ldots, \gamma ^{d})  $ by the pairing:
\begin{equation} \label{eq:mud2}
\langle \mu ^{d} _{F (\Sigma ) } (\gamma
^{1}, \ldots, \gamma ^{d}),
\gamma ^{0}   \rangle = \sum _{A} \# \mathcal {M} (\gamma ^{1}, \ldots , \gamma ^{d}; \gamma ^{0},  \Sigma, \mathcal{D}, A),
\end{equation}
when the above moduli spaces are of dimension 0, for
$ \langle ,  \rangle $ as before the inner product
pairing induced by our basis choice. Again the sum
is finite by the monotonicity.

\subsubsection {Associativity} This works as before.
 \begin{lemma} \label{lemma:functorialF} The assignment $\Sigma
\mapsto F (\Sigma)$  extends to a natural functor $$F: Simp (X) \to
A_{\infty}-Cat.
$$ 
\end{lemma}
\begin{proof} Given a face map $f: \Delta ^{n-1} \to \Delta ^{n}$ and 
$\Sigma^{n}: \Delta ^{n} \to X$,  by the
   naturality Axiom \ref{property:natfacemap} of our
connections there is a canonical functor $F (\Sigma ^{n} \circ f) \to F
(\Sigma ^{n})$ that is by construction  a fully-faithful embedding. 
It follows via iteration that a morphism $\sigma: \Sigma ^{k} \to \Sigma ^{l}$, with
$\Sigma ^{k}, \Sigma ^{l} \in Simp (X)$, $k<l$ induces a fully-faithful
embedding: $$F (\sigma): F (\Sigma ^{k}) \to F (\Sigma ^{l}),$$ and this assignment is clearly
functorial. Note that $F (\sigma)$ is essentially surjective on the cohomological level,
which follows by a classical continuation argument, cf. \cite[Section
10a]{citeSeidelFukayacategoriesandPicard-Lefschetztheory}, and so each $F (\sigma)$ is a
quasi-equivalence. 
\end{proof}
Let us call the functor $F _{P, \mathcal{D}}: Simp (X) \to A _{\infty}-Cat  $, as
constructed geometrically in
this section, a \emph{geometric functor} to emphasize the origin.
\subsection {Unital replacement of $F$} \label{sectionReplacement}
Let $A _{\infty} -Cat ^{unit}$ denote the subcategory of $A _{\infty}-Cat $
consisting of strictly unital $A _{\infty} $ categories and unital functors. 
By \emph{unital replacement} for $$F: Simp (X) \to A _{\infty}-Cat $$ we mean a
functor  $$ {F} ^{unit} : Simp (X) \to A _{\infty} -Cat ^{unit}$$ together with a
natural transformation $$N: F \to {F} ^{unit}, $$ which is object-wise
quasi-equivalence. 
\begin{lemma} \label{lemma:unitalreplacement} Any functor $F: Simp (X) \to A
_{\infty}- Cat$ has a unital replacement.
\end{lemma}
\begin{proof} 
To obtain this we proceed inductively: for each 0-simplex $x \in Simp  
(X)$, since each $F (x)$ is c-unital we may fix a formal diffeomorphism $\Phi
_{x}: F (x) \to F (x)$, with first component maps
   $\Phi ^{1} _{x}  $ the identity maps, such that
   the induced $A_\infty$-structure
   $$ {F} ^{unit}  (x) = (\Phi_x)_* (F (x))$$ is strictly unital, \cite[Lemma
2.1]{citeSeidelFukayacategoriesandPicard-Lefschetztheory}. Let $$N _{{x}}:
  {F} (x) \to  {F} ^{unit}  (x)$$
  denote the induced $A _{\infty}$ functor.  Let $F _{k} $ denote the 
   restriction of $F$ to $Simp ^{\leq k}  (X)$ with $Simp ^{\leq k} (X)$ denoting the sub-category of $ Simp (X)$, consisting
of simplices whose degree is at most $k$.
And 
suppose that the maps $N_{{x}}$ can be extended to a natural transformation $N
_{k}: F _{k} \to F ^{unit}_{k} $ of functors $$F _{k}: Simp ^{ \leq k} (X) \to A _{\infty} -Cat,$$ $${F}
   ^{unit} _{k} : Simp ^{\leq k} (X) \to A _{\infty} -Cat ^{unit},$$ 
$k>0$ with the following property.   $$\forall
   \Sigma:  N _{k} (\Sigma):  {F}  (\Sigma)
\to  {F}  ^{unit}   (\Sigma)$$ is induced by a formal diffeomorphism $\Phi
_{\Sigma}: F (\Sigma) \to F (\Sigma) $, whose first component maps are the
identity maps.

We construct  an extension $N _{k+1}$.
For each given 
$\Sigma ^{k+1}: \Delta ^{k+1} \to X$ and $i: \Sigma ^{k} \to \Sigma ^{k+1}  $,
a morphism in $Simp (X)$,
by assumption $F (i) $ is a fully-faithful embedding. Identifying 
$F  (\Sigma ^{k} ) $ with a full subcategory of $F (\Sigma ^{k+1} )$,
we may clearly construct, as in the proof of \cite[Lemma
2.1]{citeSeidelFukayacategoriesandPicard-Lefschetztheory}, a formal diffeomorphism $$\Phi _{\Sigma ^{k+1}}: F
(\Sigma ^{k+1} ) \to F (\Sigma ^{k+1} )$$ 
with $\Phi ^{*} _{\Sigma ^{k+1} }  (\Sigma ^{k+1}   ) $ unital, 
and so that its restriction to $F (\Sigma ^{k} )$ coincides with the formal diffeomorphisms
$\{\Phi _{\Sigma ^{k+1} \circ i} \}$, for each $i: \Sigma ^{k} \to \Sigma ^{k+1}
$.
The result then follows.  
 \end {proof} 
Let us write $F ^{unit}  $ for the particular unital
replacement of $F $
as constructed in the proof of the lemma above.

\subsection {Concordance classes of functors $F: Simp (X) \to A _{\infty}-Cat$}

We say that a pair of functors $F_0, F_1: Simp (X) \to
A_{\infty}-Cat $ are
\emph{concordant} if there is a functor $$T: Simp (X \times I) \to
A_{\infty}-Cat  ,$$ restricting to $F _{0},F _{1}  $ over $Simp (X \times \{0\})$,
respectively over $Simp (X \times \{1\})$.
Note that
 by the proof of Lemma \ref{lemma:unitalreplacement} if $F _{1}, F _{2}  $ are
 concordant then so are $F _{1} ^{unit}, F _{2} ^{unit}    $.
 
 \begin{theorem} \label{thm.concordance} Let $M
    \hookrightarrow P \to X$ be a smooth Hamiltonian
    fibration. For a given pair of data $\mathcal{D}
    _{1}, \mathcal{D} _{2}  $ for $P$, the 
    functors $$F _{P, \mathcal {D} _{1} }: Simp (X) \to A_{\infty}-Cat, $$
    $$F _{P, \mathcal {D} _{2}}: Simp (X) \to A_{\infty}-Cat $$ are concordant.
 \end{theorem}
\begin{proof}  
The pair $\mathcal{D} _{1}, \mathcal{D} _{2}$  are
concordant by Theorem \ref{thm:concordanceMain}.
   Let $\widetilde{\mathcal{D} } (P \times I)  $
   denote the corresponding data. Then clearly 
$$F _{P \times I, \widetilde{\mathcal{D} } (P
   \times I)}: Simp (X) \to A_{\infty}-Cat, $$
gives the required concordance.   

\end{proof}
\begin{remark} 
\label{sub:homotopy_groups}
Concordance
relation is an equivalence relation (in the special case above). Although we will not show this here. The concordance
class of the functor $F _{P, \mathcal{D}} $ is
   then the most fundamental invariant of the
   Hamiltonian fibration $P$ that is constructed in this paper, however calculating with it may be very
difficult.
\end{remark}

\section {Global Fukaya category} \label{section.GF}
Let $M \hookrightarrow P \to X$ be as previously. In
this section, we will associate to the
previously constructed functors
$F _{P, \mathcal{D}} $  a certain
geometric-categorical object
which we call the global Fukaya category. More
specifically this
will have the structure of a categorical fibration
over $X _{\bullet }$, which is our name for a
categorical fibration over a Kan complex. 
So one necessary ingredient for this story will
be the notion of an $\infty$-category, or a
quasi-category in the specific model here. As this
model is fixed in the paper we will no longer
mention this. An $\infty$-category is a simplicial set with an additional property, relaxing the notion of Kan complex.  

Whereas Kan complexes are fibrant objects in the Quillen model structure on the category
$sSet$ of simplicial sets,  
$\infty-categories$ are in turn the fibrant objects for a different non Quillen
equivalent model structure on $sSet$ called the Joyal model structure. 
For the reader's convenience we will review some of
this theory of simplicial sets in the Appendix \ref{appendix.quasi}.

We will see in Section \ref{section:extension} how
to enrich our construction so that our geometric
functors $F _{P, \mathcal{D} }$ extend to functors 
$$F _{P, \mathcal{D}}: \Delta (X) 
\to A_{\infty}-Cat ^{unit}, $$ in other words so
that degeneracies are included. This is purely
algebraic and we assume this for now. 
\begin{remark}
   \label{remark:NaiveColimit} A naive idea for an
   invariant of the Hamiltonian fibration $M
   \hookrightarrow P \to X$ is to try to form the
   colimit  directly:
$$Fuk (P, \mathcal{D}) = colim _{\Delta (X) } F _{P,
\mathcal{D} },$$ which one may hope is an $A
   _{\infty}$-category. However, this has great
technical difficulties. The colimit may not even
exist, as in general colimits of diagrams of $A
_{\infty}$ categories may not exist.  Our category
$A _{\infty}-Cat ^{unit}$ is a very special
sub-category of all $A _{\infty}$ categories, so
that such co-limits may exist (this is perhaps open). But this is not good
enough, as we need suitable invariance of $Fuk (P,
\mathcal{D} ) $, say up to quasi-isomorphism,    under change of
$\mathcal{D} $, which means that in our case we
need some kind of homotopy colimit, which means
that our $A _{\infty}-Cat ^{unit}$ needs to be some kind
of model category. This is again a technical
challenge particularly because $A _{\infty}-Cat
   ^{unit}$
is so special.  See however ~\cite{citeLefevre-HasegawaSurlesAinftycategories} where a
kind of model structure is constructed on a more
general  category of $A _{\infty}$  categories,
(but with no co-limits!). 
\end{remark}
We are going to compose $F$ with  the nerve
functor to land in the much more robust category of
$\infty$-categories, and then take the colimit. The use of the nerve functor has some perhaps unexpected benefits. We get a certain rich additional structure for our invariant
object, closely tied the geometry, (a categorical
fibration structure). This will be crucial for computations in Part II.   
\subsection{The $A _{\infty} $-nerve} \label{sec:OutlineAinftyNerve}
We have already briefly discussed the $A _{\infty}
$-nerve in the Introduction, and from now on it
will just be called the nerve $N$.  

 We want a certain natural functor $$N: A _{\infty}-Cat ^{unit} \to
\infty-\mathcal{C}at.$$ A full construction is in
Appendix \ref{appendix:nerve}, but here is an
outline.
Let ${C}$ be a strictly unital $A
_{\infty}$ category. The 2-skeleton of the nerve $N ({C})$, has objects of
${C}$ as 0-simplices, morphisms of ${C}$ as $1$-simplices and the 2-simplices consist of a triple of objects $X, Y, Z$, a
triple of morphisms $$f \in hom  _{  {C}} (X,Y), g \in hom _{ 
{C}} (Y, Z), h \in hom _{  {C}} (X, Z),$$
a morphism $e \in \hom  _{{C}} (X, Z) _{1}$,
(subscript $1$ corresponds to the degree) with $de= h - f \circ g$.  

\subsection {Definition of the global Fukaya
category}
\label{sec:GlobalFukaya}
\begin{definition}  We define:
\begin{equation*} Fuk _{\infty} (P, \mathcal {D})
   :=    colim _{\Delta (X)
   } N \circ F ^{unit} _{P, \mathcal{D} }  \in sSet.
\end{equation*}  
\end{definition} 
An explicit construction of the colimit is given in Lemma \ref{lemma.colimit}.  
In principle the above definition could be very impractical since general
objects in $sSet$ are difficult to deal with, while taking fibrant replacements
for the Joyal model category structure
could obfuscate all the original geometry contained in the Fukaya category.
Thankfully none of this is necessary as we have a couple of miracles coming from the underlying 
geometry to save us.  The upshot of these miracles is the
following theorem, to be proved in Section \ref{section.algebraic}.

\begin{theorem} \label{prop:quasicatfibration} As defined $Fuk _{\infty} (P,
\mathcal {D}) \in \infty-\mathcal {C}at$, i.e. is
   a $\infty$-category moreover
there is a natural categorical fibration $$N
   (Fuk (M, \omega )) \hookrightarrow Fuk
_{\infty} (P, \mathcal {D}) \to X _{\bullet},$$
   whose  concordance equivalence (Definition
   \ref{def:concordanceCartesian})  class
   is independent of the choice of $
\mathcal {D} $.
\end{theorem}

\subsection {Universal construction via diffeological
spaces} \label{sectionFukayaHochschild}
Let $E _{M} \to BHam (M, \omega) $ be the associated Hamiltonian
$M$-bundle $$E _{M} = E \times _{Ham (M,  \omega)
} M,$$ for $E$   
the universal principal $Ham (M,
\omega)$-bundle $E \to BHam (M,  \omega) $.
$B Ham (M,
\omega)$ or the classifying space of any smooth Lie group, based on Milnor's
construction \cite{citeMilnoruniversalbundles}, 
admits a well-defined notion of smooth maps into it from smooth manifolds. To be precise it has a natural diffeology, see Magnot-Watts~\cite{citeMagnoWattsDiffeology} and likewise the universal $M$-bundle $$E _{M}  \to BHam (M,
\omega)$$ has a natural diffeology, so that for a diffeological map $$f: B \to BHam (M,
\omega)$$ the pull-back bundle $f ^{*} E _{M}  $ is naturally diffeological. If
$B $ is in addition a smooth dimension $k$ manifold then $f ^{*} E _{M}  $ is a diffeological
space locally (diffeologically) diffeomorphic to
$U \times M$,  for $U \subset \mathbb{R} ^{k} $. 
But the
latter is clearly locally diffeomorphic to $\mathbb{R} ^{k+2n} $, for $2n$ the dimension
of $M$. Thus $f ^{*} E _{M}  $ is a diffeological space locally diffeomorphic to
$\mathbb{R} ^{k+2n} $ and hence is a smooth manifold. More formally, $f ^{*} E _{M}  $ is contained in the
full subcategory of the category of diffeological spaces corresponding to smooth
manifolds.

In the case $B = \Delta ^{k}
$  is the $k$-simplex, and given an open $U$ with $\Delta ^{k} \subset U \subset \mathbb{R} ^{k}
$, by a smooth map $\Sigma: \Delta ^{k}  \to BHam
(M, \omega)$ we mean a map with a diffeological extension
$\widetilde{\Sigma }: U  \to BHam (M, \omega)$, with $U$ given the 
diffeology induced from $\mathbb{R} ^{k} $ .  
Then we may conclude as above that $\widetilde{\Sigma }  ^{*} E _{M}  $ is naturally a smooth bundle.  

So to each smooth, in the sense
above, map $\Sigma:
\Delta^{k} \to BHam (M,  \omega)  $  we have a
naturally corresponding smooth bundle $\Sigma
^{*}  E _{M}$ over $\Delta^{k} $. The construction of the
previous section then works as before, associating
to $\Sigma $ an $A _{\infty}$ category $F _{E
_{M}, \mathcal{D}} (\Sigma
) $.   We may then define $B Ham (M,  \omega)
_{\bullet } $ as the simplicial set, with $B Ham (M,  \omega)
_{\bullet } (k) $   the set of
diffeological, collared maps $\Sigma: \Delta ^{k}
\to BHam (M, \omega)$. $B Ham (M,  \omega)
_{\bullet } $ is readily seen to be a Kan complex.

\begin{proposition}
   \label{prop:functorSimpBham} There is a natural 
functor 
   $$F _{E
_{M},\mathcal{D} }:
Simp(BHam (M,  \omega)) \to A _{\infty} -Cat
   , $$
and so an induced functor
   $$F ^{unit} _{E
_{M},\mathcal{D} }:
\Delta(BHam (M,  \omega)) \to A _{\infty} -Cat
   ^{unit}, $$
for $Simp(BHam (M,  \omega))$ the category of 
smooth (diffeological simplices) as above.
\end{proposition}
The proof is omitted since this is just a
summary of what we have already discussed. 
\subsection {Universal construction via smooth
simplicial sets}
\label{section:SmoothSimplicialSet}
A more abstract but technically more elementary
approach to the universal construction can
   be extracted from the author's
~\cite{citeSavelyevSmoothSimplicial}. There, an
abstract Kan complex 
$BG ^{\mathcal{U} } _{\bullet} $ \footnote {The
notation in ~\cite{citeSavelyevSmoothSimplicial}
omits the subscript: $\bullet$.}  is
constructed for any Frechet Lie group and for each choice of a particular
Grothendieck universe $\mathcal{U} $.  The  
Kan complex $BG ^{\mathcal{U} } _{\bullet} $ 
has a certain additional structure called a smooth
structure. Concretely, this smooth structrure implies that for every
$k$-simplex $\Sigma \in BG
^{\mathcal{U} } _{\bullet } (k) $, there is a
canonically associated smooth $G$-fibration $P
_{\Sigma } \to \Delta^{k} $. Using this, we
immediately 
obtain
a functor using our construction:
\begin{equation}
   \label{eq:Fabastract}
   F _{E _{M}, \mathcal{D} }:
Simp(BHam (M,  \omega) ^{\mathcal{U} } _{\bullet
   }) \to A _{\infty} -Cat,
\end{equation}
where  $Simp(BHam (M,  \omega) ^{\mathcal{U} } _{\bullet
   })$ denotes the simplex category of the
simplicial set $BHam (M,  \omega) ^{\mathcal{U} } _{\bullet
   }$, cf. Section
   \ref{section:simplexCategory}.    
   Now, as a particular case of 
   ~\cite[Theorem
   7.5]{citeSavelyevSmoothSimplicial},  $$|BHam (M,
\omega) ^{\mathcal{U} } _{\bullet }| \simeq BHam
(M,  \omega)  $$  for $|\cdot|$  the geometric
realization, and $\simeq$ homotopy equivalence.  
In particular, to prove  Theorems
   \ref{thmInfinityUniversal},
   \ref{thmGroupHomoIntro} we may also start with
   \eqref{eq:Fabastract}. See the proof below. One
   advantage of this simplicial approach is that
   the connection with homotopy groups becomes
elementary.

\subsubsection* {Proof of Theorems
\ref{thmInfinityUniversal},
\ref{thmGroupHomoIntro}}
\label{sec:proofTheoremsMain} 
Given a smooth Hamiltonian fibration $M \hookrightarrow P
\to X$, by Theorem \ref{prop:quasicatfibration}
we obtain a well-defined concordance class of
a categorical fibration $Fuk _{\infty} (P) \to X _{\bullet}$. By Theorem \ref{corollaryStraightening} this is
classified by a homotopy class of a map:
\begin{equation*}
	cl _{P}: X _{\bullet} \to (\mathbb{S}, N
	Fuk (M,  \omega)).   
\end{equation*}

Likewise, given the functor  $$F ^{unit} _{E
_{M}, \mathcal{D}}: \Delta (BHam (M,  \omega)) \to A _{\infty} -Cat ^{unit}, $$ by Theorem \ref{prop:quasicatfibration} we
obtain a  categorical fibration 
\begin{equation} \label{eq:universalCocartesian}
	N (Fuk (M, \omega )) \hookrightarrow Fuk
_{\infty} (E _{M}, \mathcal {D}) \to BHam (M,  \omega)  _{\bullet}.
\end{equation}
By Theorem \ref{corollaryStraightening},
there is then a uniquely determined (simplicial) homotopy class of the ``classifying'' simplicial map 
\begin{equation*} 
cl = cl (Fuk _{\infty} (E _{M})  ): B Ham (M,
	\omega) _{\bullet}    \to (\mathbb{S}, NFuk (M)),
\end{equation*}
of the categorical fibration \eqref{eq:universalCocartesian}.  Then we obtain a group homomorphism  of simplicial
homotopy groups $$cl _{*}: \pi _{i} (B Ham (M,
\omega) _{\bullet}, x _{0})   \to \pi
_{i} (\mathbb{S}, NFuk (M)).
$$
If we knew that $B Ham (M,  \omega) _{\bullet}$  is weakly equivalent to the usual continuous
singular set of $B Ham (M,  \omega) $,  then we
would obtain a group homomorphism
$$cl _{*}: \pi _{i} (B Ham (M,
\omega), x _{0})   \to \pi
_{i} (|\mathbb{S}|, NFuk (M)).
$$ This is probably true, but I don't know if a
ready reference exists.
Alternatively, we can use the map $$cl: B Ham (M,
\omega) ^{\mathcal{U}} _{\bullet}   \to
(\mathbb{S}, NFuk (M)),
$$
induced by \eqref{eq:Fabastract}.  Since $|BHam (M,
\omega) ^{\mathcal{U} } _{\bullet }| \simeq BHam
(M,  \omega)  $ we immediately  obtain the
homomorphism
$cl _{*}: \pi _{i} (B Ham (M,
   \omega), x _{0})   \to \pi
   _{i} (|\mathbb{S}|, NFuk (M)).
   $ 
And this fully proves Theorem \ref{thmGroupHomoIntro}.

Now if $M \hookrightarrow P \to X$  is a smooth Hamiltonian fibration then $P \simeq f _{P} ^{*}E _{M}$ for some
diffeological smooth map $f _{P}: X \to B Ham
(M,  \omega) $.   Then by Theorem \ref{thmNaturality} 
$$Fuk _{\infty} (P) = f _{P, \bullet} ^{*} Fuk _{\infty} (E
_{M}), $$  with $f _{P, \bullet }: X _{\bullet } \to B Ham
(M,  \omega) _{\bullet }$ denoting the induced
simplicial map.  In particular, $Fuk _{\infty} (P) $
is classified as a categorical fibration by the map
$cl \circ f _{P}$.    And so by Theorem
\ref{corollaryStraightening}  $cl _{P} \simeq
cl \circ f _{P}$. And so we have proved Theorem
\ref{thmInfinityUniversal}.  (We could also
have proceeded via smooth simplicial sets for this
part.) 
\qed
\begin{remark} \label{remark:unitalreplacement}
In the construction of $Fuk _{\infty} (P)$ we had to take a
unital replacement for the functor $F: \Delta/X _{\bullet}  \to A _{\infty}-Cat $. One may worry then that this algebraic step will obfuscate the ``geometry'' of simplices of $Fuk _{\infty} (P) $. This is not really the case. First the $A _{\infty} $ nerve $NC$ of a non-unital $A _{\infty} $ category $C$ still
exists as a semi-simplicial set,  that is as 
a co-functor $\Delta ^{inj} \to Set$, with $\Delta ^{inj} $ the subcategory of
$\Delta$ consisting of injective morphisms. For a unital replacement equivalence $C 
\to C ^{unit}  $ of $C$, constructed as in Section \ref{sectionReplacement},
we then have an induced morphism of semi-simplicial sets
$NC \to NC ^{unit}$, which by construction induces a bijection $NC ([n]) \to NC
^{unit} ([n]) $, for each $[n]$. So we may think without loss of geometric information, of simplices of $NC
^{unit} $ in terms of simplices of $NC$. (The former just
have an extra formal algebraic structure.)
\end{remark}
\section{Extending $F$ to degeneracies} \label{section:extension}
We have to construct our perturbation data $\mathcal{D}$ for all simplexes in such a way that there is a natural functor:
\begin{equation} \label{eq:extendedF}
   F: \Delta(X) \to A _{\infty}-Cat ^{unit},
\end{equation}
extending the geometric functor
$$F _{\mathcal{D}}: Simp(X) \to A _{\infty}-Cat ^{unit},
$$ for this data $\mathcal{D}$
as previously constructed. This perturbation data will be referred to as \emph{extended perturbation data}.

Suppose that we are given a commutative diagram: 
\begin{equation*}
\begin {tikzcd}
\Delta ^{0}  \ar [r, "j+1"]  \ar [dr, "
   x _{j}"]   & \Delta ^{n+1}
   \ar [d, "\widetilde{\Sigma} "]   \ar [r, "pr"]  & \Delta ^{n} \ar [ld, "\Sigma"] \\
                                   & X, 
\end {tikzcd}
\end{equation*}
where 
$$pr: \Delta ^{n+1} \to \Delta ^{n}, \quad j \in [n] $$
is induced by the unique surjection $[n+1] \to
[n]$ taking $j$ and $j+1$  to $j$.
Here $x _{j}  = \Sigma \circ j $, for 
${j} $ also denoting the  
map $pt \to \Delta ^{n} $   whose image is the
vertex $j$.   In particular, we have a
morphism $pr: \widetilde{\Sigma } \to \Sigma  $
in $\Delta (X) $.

Let the system of natural maps $\mathcal{U} (X) $
be fixed throughout in what follows.  And $\mathcal{D} = (\mathcal{U} (X),
\mathcal{F}) $. Let $\mathcal{F} _{\Sigma} $
correspond to $\Sigma$.  We first show how to construct certain induced perturbation
data $\mathcal{F} _{\widetilde{\Sigma}} = pr ^{*}
\mathcal{F} _{\Sigma} $. And using this we
construct an $A _{\infty} $ category $F
(\widetilde{\Sigma} ) $ as previously.  

We may as before define $$\obj F 
(\widetilde{\Sigma} ) = \bigcup _{0 \leq i \leq
n+1} \obj F (x _{i}), \quad x _{i} =
\widetilde{\Sigma} \circ i,$$ $i: pt \to
\Delta^{n+1} $ the inclusion map of $i$'th vertex. And in this case $pr$ clearly induces a map of sets of objects $$pr _{*}: \obj F  (\widetilde{\Sigma} ) \to \obj F (\Sigma).
$$ 
We need to specify our system $\mathcal{F}
_{\widetilde{\Sigma } }$  of connections and
almost complex structures corresponding to
$\widetilde{\Sigma}$. We say what to do with
connections, the case of almost complex structures
is analogous. Given vertices $i,j$  of $\Delta
^{n+1} $ and objects $L \in {F}  (x
_{i}), L' \in {F} (x _{j}  )$,  we set $$\mathcal{A} (L, L') =  \mathcal{A} (pr _{*} L, pr _{*} L'),
$$ where the latter connection is determined by $\mathcal{F} _{\Sigma}  $, and where the equality is with respect to the natural identification 
$$(\Sigma \circ pr \circ m
_{i,j}) ^{*}  P =  (\widetilde{\Sigma} \circ m
_{i,j}) ^{*} P.
$$

Likewise, given objects $L _{0}, \ldots, L _{s} \in F (\widetilde{\Sigma} ) $  we set 
$$\mathcal{F}   (L_0, \ldots , L_s, \widetilde{\Sigma} ,r  ) = \mathcal{F}   (pr_*L_0,
\ldots, pr_*L_s, \Sigma,r  ).$$ Here the equality
is again with respect to the natural
identification of the corresponding bundles,
(based on the Axiom \ref{axiom:naturalityX2}  of $\mathcal{U} (X) $).

All together this determines the (partial) data
$\widetilde{D} _{\Sigma} = (\mathcal{U} (X),
\mathcal{F} _{\widetilde{\Sigma } } )  $.

Using this $\mathcal{D} _{\widetilde{\Sigma} } $ we then define an 
 $A _{\infty} $ category denoted by $ F (\widetilde{\Sigma} ) $ as previously.
By construction, the natural map on objects:
 \begin{equation*} 
 pr _{*}:  \obj F   (\widetilde{\Sigma} ) \to \obj F (\Sigma)
\end{equation*}
extends to a strict $A _{\infty} $ functor 
$$F (pr): F (\widetilde{\Sigma}) \to F (\Sigma) $$ satisfying:
\begin{equation*}
F (pr) \circ F (\sigma) = id,
\end{equation*} where $\sigma: \Sigma \to
\widetilde{\Sigma}$ is induced by the map  $d
^{j+1}: [n] \to [n+1]   $, which is the unique
injection in $\Delta$  whose image misses the
vertex $j+1$. We are going to call the above the
\textbf{\emph{extension construction for
$\widetilde{\Sigma}, \mathcal{D}  _{\Sigma} $}}
with the corresponding extended data denoted by $\mathcal{D} 
_{\widetilde{\Sigma } }$, called \emph{extended
perturbation data for $\widetilde{\Sigma} $ }. 

We are now going to proceed by induction. Let $S
(N) $ be the statement: there exists an extended
perturbation data $\mathcal{D}$ for all simplices
up to degree $N$, so that we have an extension
of $F| _{Simp ^{N} (X) } $ to a functor $$F ^{N}:
\Delta ^{N}(X)  \to A _{\infty}-Cat ^{unit}, $$
where $\Delta^{N} (X)  $ and $Simp ^{N} (X)$ are
the subcategories of simplices of degree at most
$N$. The statement $S (0) $  is trivial since $Simp
^{0} (X)  = \Delta ^{N} (X) $, so that there is
nothing to prove.  As usual for us, we also denote
by $S (N) $  the corresponding partial
perturbation data.

We prove $S (N) \implies S (N+1) $, and 
moreover $S (N+1) $    can be assumed to extend $S
(N) $. Let ${\Sigma}'  $ be a general $(N+1)$-simplex
of $X _{\bullet} $. If $\Sigma ' $ is degenerate,
so that $\Sigma' = \Sigma \circ pr $ for some
degeneracy morphism $pr$, for some $\Sigma \in
Simp ^{N} (X) $, then define the data $\mathcal{D}
_{\Sigma'} $ and $F ^{N+1} (\Sigma') $ as in the
extension construction above for
$(\widetilde{\Sigma } =  \Sigma'), \mathcal{D} _{\Sigma} $.
Otherwise, if $\Sigma$ is a non-degenerate $(N+1)$-simplex then its faces are $N$-simplices for which we already have perturbation data $\mathcal{D}$,
which is then extended arbitrarily to perturbation
data $\mathcal{D} _{\Sigma} $ for $\Sigma$.
The extension is obtained as in the proof of
Lemma \ref{lemma:naturalF}. Using this data define
$F ^{N+1} (\Sigma) $ as previously. 

We have thus completed the induction step.   By recursion, we may then define a sequence of
systems $\{S _{N}\} _{N \geq 0}$,  so that $S (N+1)  $
extends $S (N) $, for each $N$. 
We then
define that total extended data as $\mathcal{D} = \bigcup_{N}
S (N)  $. And so  we obtain our  extension
$$F:
\Delta (X)  \to A _{\infty}-Cat ^{unit}, $$ by the
previous construction, using the extended data
$\mathcal{D} $. 

\section {Algebraic-topological considerations} \label{section.algebraic}
In this section by equivalence of
$\infty$-categories we always mean categorical equivalence. This and other
categorical preliminaries needed for this section are discussed in the Appendix
A.  
We will prove here Theorem
\ref{prop:quasicatfibration}. 

\subsection {Colimit of $F$} 
\begin{definition}\label{def:geometricFunctor}
  A functor $F: \Delta(X) \to A _{\infty}-Cat
   ^{unit}$ which is induced by a geometric
   functor $F _{\mathcal{D}}: Simp
   (X)  \to A _{\infty}-Cat ^{unit}$  as in Section
   \ref{section:extension} will be called
   \textbf{\emph{geometric}}. 
\end{definition}
\begin{remark}
   \label{remark:} It would be more ideal to extract
   suitable minimal algebraic axioms for our ``geometric functors''. However, this may take us too far afield.  
\end{remark}

Given a geometric functor $
F: \Delta(X) \to A _{\infty}-Cat ^{unit}
$, let 
\begin{equation}
   \label{eq:FukF}
   Fuk _{\infty} (F)  := colim
_{\Delta (X)} 
   N{F}. 
\end{equation}
The category of simplicial sets is well known to
be (co)-complete
so that the limit certainly exists as a simplicial
set. However, we shall show, in the following
proposition, that this limit has
additional structure of a categorical fibration. 
\begin{remark} \label{remark:}
It is possible that the proposition below can be
obtained as a consequence of more general
principles, using general theory of colimits of
$\infty$-categories.  However, I suspect that for a general
functor of the form $G: \Delta (X) \to
\infty-\mathcal{C}at $, we must first take a fibrant
replacement of the functor $G$ (for the Reedy-Joyal
model structure), if we want a similar
structure on $colim _{\Delta (X)  } G$.
(I am not however sure that this alone is sufficient.)  
\end{remark}
\begin{proposition} \label{propostion.simplicialmap}  There is a natural projection of simplicial sets 
$$p: Fuk _{\infty} (F)   \to X _{\bullet}  ,$$ and this
is a categorical fibration.
\end{proposition}
\begin{proof} 
Let us first give a more easily
conceptualized presentation of the colimit $Fuk _{\infty}
(F)$.  We should say that we are just
simplifying the standard, ``level wise'' 
construction, of colimits  of simplicial
sets, in our specific context, so that the
structure of a categorical fibration becomes
apparent.

Define a partial order $<$ on the set of pairs $(f, \Sigma)$,  $f \in NF(\Sigma) (k) $ a $k$-simplex, $k \geq 0$,  $ \Sigma
\in \Delta(X) $ as follows.  $$(f, \Sigma) < (f', \Sigma')$$ if
there is a morphism $$\sigma: \Sigma \to \Sigma'$$
in $\Delta  (X)$ induced by injective $[n] \to [m]$ with
$n \leq m$, i.e.
a \emph{face morphism}, s.t. $$NF (\sigma) (f) = f'.
$$ Clearly for every $(f, \Sigma)$ there is a unique least pair $$(f _{min}, \Sigma _{\min}) < (f, \Sigma).
$$  Note that if
$f$ is a $k$-simplex then $\Sigma _{min}$ is
not necessarily a $k$-simplex. However,  once we
impose the following equivalence relation, we
get something similar, see Lemma \ref{lemma:canonicalrespresentative} below.

  Let $\widetilde{C}$ be the set of minimal pairs. Define an equivalence relation
on $\widetilde{C} $ first by defining
$$(f, \Sigma) \sim (f', \Sigma')$$ if there exists a degeneracy morphism $d: \Sigma \to
\Sigma'$ induced by $[m] \to [n]$ with $m>n$,  such that $$NF (d) (f) = f'. 
$$  And then by imposing symmetry and
transitivity. Denote the equivalence class of $(f, \Sigma)$ by $[f, \Sigma]$.    
   
The following is not formally necessary, but it
   might be helpful for visualization.
\begin{lemma}
      \label{lemma:canonicalrespresentative} 
   Each class $[f, \Sigma] $ has a unique
   natural  representative $(f _{c}, \Sigma _{c})$  so
   that if $f$ is a $k$-simplex then $\Sigma
   _{c}$ is a $k$-simplex.    
   \end{lemma}
  \begin{proof} Let $(f, \Sigma ) $ as above be
     given,  with $\deg (f) =k$.  And
     let $(f _{min}, \Sigma _{min}) \in
     \widetilde{C} $  be as above. Then 
     $\deg(\Sigma _{min}) \leq (n = \deg (f
     _{min}))$.   Suppose that  $\deg(\Sigma
     _{min}) <
     n$. Then by the extension construction of Section \ref{section:extension} there is a degeneracy $$d _{c}:
     \Sigma _{c} \to \Sigma _{\min},$$  
   with $\deg (\Sigma _{c}) =k$ together with a $k$-simplex
  $f _{c} \in NF (\Sigma _{c})  $   so that $NF
     (d _{c}) (f _{c}) = f _{min}.  $  Moreover, $(d
     _{c}, f _{c})$  are uniquely determined, (by
     the extension construction). So
     that we set $\Sigma  _{c} = \Sigma
     _{min}\circ d _{c} $.  Also note that $\Sigma
     _{c}$ is just $p _{\Sigma} (f _{min}) $, with
     $p _{\Sigma}$ as in \eqref{eq:pSigma} ahead.
   \end{proof}

Continuing with the proof of the proposition, we then define $C =\widetilde{C}  /
\sim$. This is naturally a simplicial set, with $$C (k) = \{[f, \Sigma] \in C \, \vert \, f \in NF
   (\Sigma) (k)\}.
    $$ 
   For example $C (0)$ is naturally isomorphic
   to $$\sqcup
_{x \in X} Obj \, F  (x).$$

The following in particular will give a direct proof
that the colimit \eqref{eq:FukF} exists.
\begin{lemma} \label{lemma.colimit} $C = colim _{\Delta (X)} 
   N{F} $, 
   with equality meaning natural
   isomorphism.   
\end{lemma}
\begin{proof}   Note first that $C$ is a co-cone on the diagram $NF $. Indeed, for
each $\Sigma$ define $\phi _{\Sigma}: NF   (\Sigma) \to C$ by $$\phi _{\Sigma} (f) = [f _{min} , {\Sigma} _{\min}   ].$$ 

 It is easy to see that for a face morphism $i: \Sigma
\to \Sigma'$ we have that the composition $$NF  (\Sigma) \xrightarrow{NF
  (i)} NF  (\Sigma')
\xrightarrow{\phi _{\Sigma'}} C,$$  coincides with $\phi _{\Sigma} $.
Likewise for a degeneracy morphism 
$d: \Sigma
\to \Sigma'$ we have that the composition $$NF  (\Sigma) \xrightarrow{NF
  (d)} NF  (\Sigma')
\xrightarrow{\phi _{\Sigma'}} C,$$  coincides with $\phi _{\Sigma} $,
because of  the equivalence relation $\sim$.

The universal property is also easy to verify, for given another
co-cone $C'$ with maps $\rho _{\Sigma}: NF   (\Sigma) \to C'$, $\Sigma \in
   \Delta  (X)$ we can naturally define $U: C \to C'$ by $$U ([f, \Sigma]) = \rho  _{\Sigma}
(f).$$ Then $U$ is clearly well-defined, by $C'$
   being a co-cone. And moreover, $U$ is a map of
   co-cones (all the relevant diagrams
   commute). Since for a given $f \in NF
   (\Sigma)$, we have  $$U (\phi
   _{\Sigma} (f)) = \rho _{\Sigma _{\min} } (f _{\min} ) = \rho
   _{\Sigma} (f),$$
   where the last equality holds since $(C', \{\rho _{\Sigma}\})$ is a co-cone, and
 since by construction there is a morphism  $$i: NF (\Sigma _{\min} ) \to NF (\Sigma),$$
   with $F (i) f _{\min}  = f $. 
\end {proof}

Continuing with the proof of the proposition, recall that
a given $\Sigma: \Delta ^{n} \to X $ could equally be
thought of as an element of $X _{\bullet} (n)$ or as
a simplicial map $\Delta^{n} _{\bullet} \to X _{\bullet} $.
In what follows $\Sigma$ denotes all these objects simultaneously.
With this understanding, for each $n$-simplex $\Sigma \in \Delta (X) $, we have a natural simplicial map 
\begin{equation}
   \label{eq:pSigma}
   p _{\Sigma}:
   N {F} (\Sigma) 
   \to {\Sigma} ({\Delta} ^{n} _{\bullet}) \subset X
   _{\bullet},
\end{equation}
defined as follows. On the vertices of $N {F} (\Sigma)$, $p
_{\Sigma} $ is just the obvious projection. That is a vertex 
$v \in NF (\Sigma) (0) $ corresponds to an element of $NF (x
_{i}) $ for a uniquely determined $x _{i} \in \Sigma (\Delta ^{n} _{\bullet }) (0)  $, and $p _{\Sigma } (v)  = x _{i}$.

Given a $k$-simplex $f$ in $N F (\Sigma) (k) $, we get
a composable chain $(f _{1}, \ldots, f _{k})$, with $f _{i}$
the edge of $f$ between its $i-1$'st and $i$'th vertex, for $1 \leq i \leq k$. This  determines a list of vertices $v_0, \ldots, v_k \in NF (\Sigma) $ s.t. the source/target of $f_i$ is $v _{i-1}$ respectively $v _{i} $.
This in turn determines a list of vertices $\{p _{\Sigma} (v _{i}
)\} \subset {\Sigma} ({\Delta} ^{n} _{\bullet}) (0)  $,  and we set $p _{\Sigma} (f) $ to be the unique (possibly degenerate)
$k$-simplex in ${\Sigma} ( {\Delta} ^{n} _{\bullet}) (k) $ with
these vertices. We will omit the verification that $p _{\Sigma} $ is
simplicial. 
   
The simplicial projection  $$p: C  \to X _{\bullet}$$ is then: send $ [f, \Sigma]$ to $p _{\Sigma} (f)$, which is readily seen to be well-defined.

 It is immediate from the definition of
an inner fibration in Section \ref{sec:innerfibrations} that $p$ is an inner-fibration if
and only if the pre-image of every simplex $\Sigma: \Delta ^{n} _{\bullet} \to X
   _{\bullet} $ by $p$ is a $\infty$-category, where the ``pre-image'' $p ^{-1} (\Sigma) $ is the pre-image by $p$ of the simplicial subset $\Sigma  (\Delta ^{n} _{\bullet} ) $.
In our case this follows by construction, as $p ^{-1} (\Sigma ) $ is clearly identified with $NF  (\Sigma)$, which is an $\infty$-category by properties of $N$.

We now verify that in addition $p$ is a
categorical fibration.   By the Definition
\ref{def:categoricalFibration} we
need to show that for every equivalence $m: a \to
b$ in $\mathcal{X} = {X} _{\bullet} $  and
for every object $a' \in Fuk _{\infty}(P) $ with
$p(a') = a$, there exists an equivalence
$\widetilde{m}: a'  \to b'$ in $\mathcal{E} = Fuk
_{\infty}(P) 
$ with $p(\widetilde{m}) = f$.

 \begin{lemma} \label{lemma.weakequiv} The functor $N: A_{\infty}-Cat
 ^{unit} \to \infty- \mathcal{C}at$,  takes quasi-equivalences to weak
equivalences in the Joyal model structure, i.e.
categorical equivalences.  
\end{lemma}
\begin{proof} The proof of this is contained in the proof of Proposition
1.3.1.20, Lurie \cite{citeLurieHigherAlgebraa}. We can also prove this directly by first recalling
that quasi-equivalences of $A _{\infty}
$-categories $A,B$ are invertible (when working
over a field with characteristic $0$), up
to homotopy, and then  via the nerve construction translate this to a categorical
equivalence of $N (A), N (B) $. 
\end {proof} 
Recall that the morphisms of $A _{\infty}-Cat$ are
in particular quasi-equivalences.
Since the inclusions $F (x _{i}) \to F (m) $ are
quasi-equivalences by Lemma \ref{lemma:functorialF}, it follows
by the lemma above, and by the construction of
$L$ that the inclusions of $L_{i} $ into
$L_{m} $ are categorical equivalences of
   $\infty$-categories, and so $\widetilde{m}$ as above must exist.
\end {proof}
\begin{proof} [Proof of Theorem \ref{prop:quasicatfibration}] 
By the discussion above we have a categorical fibration
$$Fuk _{\infty} (P, \mathcal{D})   \to X _{\bullet}  .$$
The first part of the theorem follows by the following general fact: for an inner fibration of
simplicial sets $$p: P _{\bullet} \to X _{\bullet},$$ if $X _{\bullet}$ is a $\infty$-category 
then $P _{\bullet}$ is a $\infty$-category. Let us prove this elementary point.
Suppose we are given  $${\rho}: \Lambda ^{n} _{k} \to P _{\bullet}$$  for $0< k < n$.
As $X _{\bullet}$ is a $\infty$-category there a simplex $$ \widetilde{\rho}: \Delta
   ^{n} _{\bullet}  \to X _{\bullet}$$ extending $p \circ \rho$.  But then
  $\rho$ maps into the $\infty$-category 
 $p ^{-1} ( \widetilde{\rho})$, and consequently  there is an
 extension of $\rho$, c.f. Proposition \ref{prop.innerfib2}.

 The final part of the theorem follows by the following.
\begin{lemma} For the geometric functor $F _{P,
   \mathcal {D}} $ the concordance class
   of the categorical fibration $p: Fuk _{\infty} (P, \mathcal {D}) \to
X _{\bullet}$ is
independent of the choice of $ \mathcal {D} $. 
\end{lemma}
\begin{proof} By Theorem \ref{thm.concordance} 
given a pair $\mathcal{D} _{0}, \mathcal{D}
   _{1}$     of perturbation data for $P$, 
there is a geometric functor:
\begin{equation*}
   \widetilde{F}: \Delta (X \times I) \to A
   _{\infty}-Cat ^{unit}, 
\end{equation*} 
which gives a concordance of the functors
$${F} _{P, \mathcal{D} _{i}}: \Delta (X) \to A
   _{\infty}-Cat ^{unit}. 
$$ 
Then by the Proposition \ref{propostion.simplicialmap} there exists an categorical fibration:
\begin{equation*}
\mathcal{T} \to X _{\bullet}  \times I _{\bullet},
\end{equation*}
whose restriction over $X _{\bullet}  \times \partial I _{\bullet} $ coincides with $$Fuk
_{\infty} (P, \mathcal {D}_{0} ) \sqcup Fuk
_{\infty} (P, \mathcal {D}_{1} ).$$ 
\end {proof}
\end {proof}
\subsection {Naturality} \label{sectionNaturalityPullback} We
may expect if our constructions are really natural
   that the categorical fibration $$N Fuk (M,
   \omega) \hookrightarrow  Fuk _{\infty} (P,
   \mathcal{D}) \to X _{\bullet}$$ is functorial with respect to pull-back and this is indeed the case. 
Let $f: X \to Y$ be a smooth map, $P \to Y$ a smooth
Hamiltonian fibration and $\mathcal{D} =
   \mathcal{D} (P) $ extended
perturbation data. We may then 
define pull-back extended perturbation data $f ^{*} \mathcal{D} $ for $f ^{*}P
   \to X$, as follows.
First we have the ``pull-back''  natural system
   $\mathcal{U} (X) $, defined by $$u (m _{1},
   \ldots, m _{s}, \Sigma) := u (m _{1},
   \ldots, m _{s}, \widetilde{\Sigma } = f \circ \Sigma
   ),$$ where the maps $u$   on the right are part
   of $\mathcal{U} (Y) $. 
Next, let $$\widetilde{f}: f ^{*}P \to P$$
   be the natural bundle map, then given $\Sigma \in X _{\bullet} (d) $ we set
$$ \mathcal{F} (L_0, \ldots , L_s, {\Sigma},r) =
   \mathcal{F} (\widetilde{f} (L_0), \ldots,
   \widetilde{f} (L_s), \widetilde{\Sigma},r), $$ for $\widetilde{\Sigma}= f \circ \Sigma  $.
This determines our data $f ^{*}\mathcal{D} $.

   Let
   ${f} _{\bullet} : X _{\bullet} \to Y _{\bullet}   $ be the induced map of
   simplicial sets.
\begin{theorem} \label{thmNaturality} $$Fuk _{\infty} (f ^{*} P, f ^{*}
   \mathcal{D}) = {f} ^{*} _{\bullet}   Fuk
   _{\infty} (P, \mathcal{D}), $$ where $f ^{*} \mathcal{D}$ is
   as above, and where  ${f} ^{*} _{\bullet}   Fuk
   _{\infty} (P,  \mathcal{D}) $  denotes the
   standard pull-back of the
   simplicial fibration by $f _{\bullet }$.  
\end{theorem}
\begin{proof} The proof is immediate.
\end{proof}

\appendix
\section {$\infty$-categories and Joyal model structure} \label{appendix.quasi}
We don't need absolutely everything in this section, particularly we can avoid ever mentioning model categories, but the latter helps with the narrative. A very good concise reference for much of this material is Riehl  
\cite{citeRiehlAmodelstructureforquasi-categories},
which we will mostly follow. The material on
various fibrations $\infty$-categories is taken
from
Lurie \cite[Section 2.4]{citeLurieHighertopostheory}. First let us recall the notion of a Kan
complex, which may be thought of as formalizing
the property of a simplicial set to be like the
singular set of a topological space, defined in
Section \ref{section:simplexCategory}.

 Let
$\Delta ^{n} $ be the standard representable $n$-simplex: $\Delta
    (i) =
\Delta ([i], [n])$. Previously we denoted this by $\Delta ^{n} _{\bullet}   $,
   but as there are no topological simplices in this section we simplify the
   notation, which is also consistent with above references.
   Let $\Lambda ^{n} _{k} \subset \Delta ^{n}$ denote the
sub-simplicial set corresponding to the ``boundary'' of $\Delta ^{n}  $ with the
$k$'th face removed, $0 \leq k \leq n$. By $k'th$ face we mean the face opposite
to $k$'th vertex.  This is called  the 
 \textbf{\emph{$k$'th horn}} or just horn. 

A simplicial set $S _{\bullet}$ is said
to be a \emph{Kan complex} if for all $n,k$ given a diagram with solid arrows
\begin{equation*} 
\begin {tikzcd}
   \Lambda ^{n} _{k} \ar [r]  \ar [d] & S _{\bullet} \\
 \Delta ^{n} \ar [ur, dotted]  &, \\ 
\end{tikzcd}
\end{equation*}
there is a dotted arrow making the diagram commute.

An \emph{$\infty$-category} is a simplicial set $S
_{\bullet}$ for which the above extension property
is only required to hold for \emph{inner horns}  $\Lambda ^{n} _{k}$,
i.e. those horns with $0<k<n$.  A Kan
complex is a simplicial model of an $\infty$-groupoid, 
as the Kan condition can be interpreted as giving
the condition that morphisms are invertible up to
(coherent) homotopy. Likewise, the  defining
property of an $\infty$-category tells us that
while 1-morphisms/edges may not be invertible (up to
homotopy),  the higher
morphisms are. It would perhaps take us
too far afield too further motivate $\infty$-categories
here. 
However, the introductory sections of Lurie~
~\cite{citeLurieHighertopostheory}   should be
highly accessible.

A morphism between   $\infty$-categories is just a
simplicial map.
We will denote $\infty$-categories by
calligraphic letters e.g. $ \mathcal {B} $.  In
\cite[Chapter 3]{citeLurieHighertopostheory} an
$\infty$-category of $\infty$-categories is
constructed, with 1-morphisms simplicial maps, and we call this $\mathcal
{C}at _{\infty}$. On the other hand the
full-subcategory of  the category of simplicial
sets with objects $\infty$-categories will be
denoted by $\infty-Cat$.       
\begin{notation}
We denote the maximal Kan subcomplex of
   $\mathcal{C}at _{\infty} $ by
$\mathbb{S} $.
\end{notation}

\subsection {Categorical equivalences, morphisms and equivalences}
\label{section:prelimQuasi} We have a
natural functor $\tau: sSet \to Cat$, defined as
follows. $\tau (S _{\bullet})$ is the category with objects 0-simplices of $S
_{\bullet}$, 1-simplices as morphisms, degenerate 1-simplices as identities and
freely generated composition subject to the relation $g= f \circ h$ if there is
a 2-simplex $e$ with 0-face $h$, 2-face $f$ and
1-face $g$. 
(Remembering our diagrammatic order for
composition.) See the figure below.     
\begin{figure}[h]
  \includegraphics[width=1.5in]{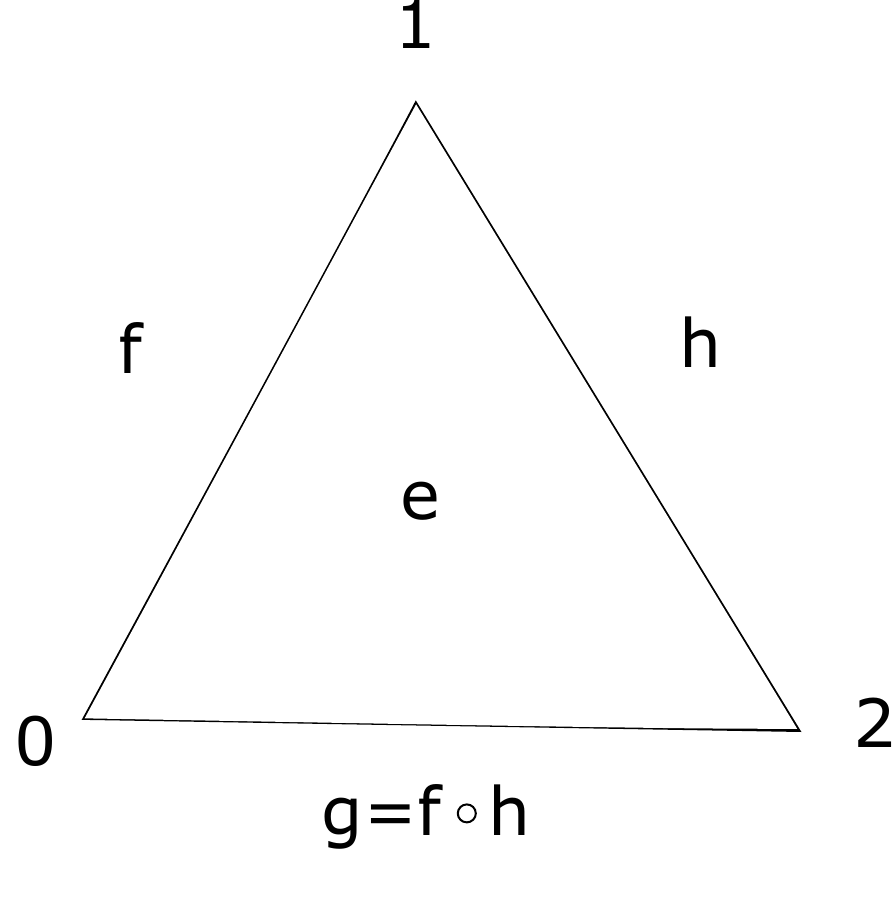}
 \caption {} \label{figure:simplex}
\end{figure} 
The category $\tau (S _{\bullet }) $  may be
understood as the \emph{fundamental category} of
$X _{\bullet }$    in analogy to the fundamental
groupoid.

We also have a functor $\tau _{0}: sSet \to Set$ by sending $A _{\bullet}$ to the set of isomorphism classes of objects in $\tau
(A _{\bullet})$.

If $S _{\bullet} = \mathcal {X}$ is a $\infty$-category an edge $e
\in \mathcal {X} $, i.e. a 1-simplex $e: \Delta
^{1} \to \mathcal{X} $,  is said to be an
\emph{equivalence} if $\tau (e) $ is an isomorphism
in $\tau \mathcal ({X}) $. We may use morphism
notation, so that $e: a \to
b$ signifies that the edge $e$ goes   from the
vertex $a$ to
$b$.

The \emph{maximal Kan subcomplex}  of a $\infty$-category $
\mathcal{X} $ is the maximal sub-simplicial set 
$K (\mathcal{X} ) \subset \mathcal{X} $ with all
edges equivalences.  (It can be constructed simply
by removing edges that are not equivalences, and
all simplices containing them.) 
The fact that $K (\mathcal{X} ) $
is forced to be a Kan
complex can be readily verified using the
$\infty$-category structure of $\mathcal{X} $.
(This is an instructive exercise.) 

\begin{definition}\label{def:connectedComponent}
Let $\mathcal{X} $ be a Kan complex and $a \in
   \mathcal{X} (0)  $. A \textbf{\emph{connected
   component}}  of $a$  is the set of vertices sharing an
   edge with $a$. We may sometimes
   denote such a connected component by
   $(\mathcal{X}, a) $.
\end{definition}

We define $sSet ^{\tau_0}$ to be the category with the same objects as $sSet$
but with the morphisms given by $sSet ^{\tau_0} (A _{\bullet}, B _{\bullet}) =
\tau ^{0} (B _{\bullet} ^{A _{\bullet}})$. 
A map of simplicial sets $$u: A _{\bullet} \to B
_{\bullet}$$ is
said to be a \emph{categorical equivalence} if the induced map in $sSet ^{\tau
_{0}}$ is an isomorphism. 

 We will say that a pair of $\infty$-categories
are \emph{categorically equivalent} if
there is categorical equivalence between them.   
As we
are following Riehl \cite{citeRiehlAmodelstructureforquasi-categories}, we refer the reader there for the following:
\begin{theorem} \label{thm:joyal}[Joyal, Lurie,
   Riehl] There is a
model structure on $sSet$, called the Joyal
model structure so that the fibrant objects are
$\infty$-categories and a weak 
equivalence between $\infty$-categories is a categorical equivalence.
\end{theorem} 
We do not formally need the above theorem, but we
hope it places things into some perspective,  by
making $\infty$-categories a less mysterious object.
\subsection {Inner fibrations}
\label{sec:innerfibrations}
A map $p: \mathcal {A}\to \mathcal {B}$ of
$\infty$-categories is said to be an \emph{inner fibration}
if it has the lifting property with respect to all inner horn inclusions. More
specifically, for $ 0< k < n$, whenever we are given a commutative diagram with 
solid arrows:
\begin{equation} 
\begin{tikzcd} \label{diagram:inner}
   \Lambda ^{n} _{k}  \ar [r]  \ar[d, hookrightarrow] &
\mathcal {A} \ar [d, "p"]  \\
 \Delta ^{n}  \ar [r] \ar [ur, dashrightarrow] & \mathcal {B}, \\
\end{tikzcd}
\end{equation}
there exists a dashed arrow as indicated, making the whole diagram commutative. 

For reference $p$ is said to be
an \emph{Kan fibration} if the above extension property holds for all horns. A
Kan fibration is an analogue in the simplicial world of Serre fibrations of
topological spaces. The following is immediate from definitions. 
\begin{proposition} \label{prop.innerfib2} A map
   $p: \mathcal{A} \to \mathcal{B}$ is an inner fibration, if and only if the pre-image of every simplex
of $\mathcal{B}$ is a $\infty$-category. 
\end{proposition}
\subsection {Categorical fibrations} These are
the analogues of Serre fibrations for the Joyal
model structure. 

\begin{definition}\label{def:categoricalFibration} 
We say that $p: \mathcal{E} \to \mathcal{X} $
is a \textbf{\emph{categorical fibration}} if:
  \begin{enumerate}
      \item  The map $p$ is an inner fibration.
\item For every equivalence $f: a \to b$ in
   $\mathcal{X} $  and every object $a' \in
        \mathcal{E} $ with $p(a') = a$, there
        exists an equivalence
        $\widetilde{f}: a'  \to b'$ in $\mathcal{E}
        $ with $p(\widetilde{f}) = f$.
   \end{enumerate} 
\end{definition} 
\begin{definition}
   \label{def:concordanceCartesian} 
We say that a pair of categorical fibrations
$p _{i}: \mathcal{P} _{i} \to \mathcal{X}  $,
   $i=0,1$  over
a Kan complex $\mathcal{X} $ are
\textbf{\emph{concordant}} if the following
holds. 
There is a categorical fibration $$\mathcal{Y}  \to \mathcal{X}
\times \Delta ^{1},$$ whose
pull-back by $i _{0}: \mathcal{X}   \to
\mathcal{X}  \times \Delta ^{1}    $
is identified with $\mathcal{P} _{0}$  and whose
   pull-back by $i _{1}: \mathcal{X}  \to
\mathcal{X}  \times \Delta ^{1} _{\bullet}   $
is identified with $\mathcal{P} _{1} $. Here the two maps $i_{0}, i_{1}$ correspond to the two vertex inclusions $\Delta ^{0} _{\bullet}  \to \Delta ^{1} _{\bullet}   $.
\end{definition}
\begin{notation} 
  \label{} 
We shall denote the set of concordance classes of
categorical fibrations over a Kan complex $\mathcal{X} $  by
$Fib _{\infty}(\mathcal{X} ) $.
\end{notation}

For convenience we recall the basic definition.
\begin{definition}\label{def:simplicialhomotopyofmaps}
   We say that a pair of maps of simplicial sets
   $f,g: A _{\bullet } \to B _{\bullet }$ are
   \textbf{\emph{homotopic}} if  there is map of
   simplicial sets
   \begin{equation*}
      F: A _{\bullet } \times \Delta ^{1} \to B
      _{\bullet},  
   \end{equation*}
  so that $F| _{A _{\bullet} \times
   \{0\}} =f$  and $F_{\bullet}| _{B _{\bullet}
   \times \{1\}} =g$. 
\end{definition}
In \cite[Section
3.3.2]{citeLurieHighertopostheory} Lurie
  constructs a universal  categorical fibration over
$\mathbb{S}$.  More specifically he constructs a
universal Cartesian fibration over $
\mathcal{C}at _{\infty}$. It's restriction over
$\mathbb{S}$  is a categorical fibration by ~\cite
[Proposition 3.3.1.8.]{citeLurieHighertopostheory}.  
  As a direct consequence we have the
following theorem. 
\footnote{The statement should be interpreted with care, since
there are set theoretic issues. The most natural
(and arguably standard)
interpretation is via Grothendieck universes, 
as for example done explicitly in
~\cite{citeSavelyevSmoothSimplicial}, in a very similar
context. But we ignore these subtleties here.}

\begin{theorem} [Lurie ~\cite{citeLurieHighertopostheory}]  \label{corollaryStraightening}
For a Kan complex $\mathcal{X}$, there is a natural 
isomorphism  $$Fib _{\infty} (\mathcal{X} ) \simeq
[\mathcal{X},  \mathbb{S}],$$          
with $[\mathcal{X},  \mathbb{S}]$  denoting the
``set'' of homotopy classes of maps $\mathcal{X}
\to \mathbb{S}  $.
\end{theorem} 
\subsection{ $A_\infty$-nerve}
\label{appendix:nerve}  This section
mostly follows Tanaka \cite[2.3]{citeLeeAfunctorfromLagrangiancobordismstotheFukayacategory}, except that for us
everything will be ungraded, and for simplicity with $ \mathbb{F}_2$-coefficients. 

For $ [n] \in \Delta$,  a \emph{length $s$ wedge decomposition} of
$[n]$ is a collection of monomorphisms in $\Delta$
\begin{equation*} j_i: [n_i] \to [n], \quad i=1,\ldots,s,
\end{equation*} 
satisfying  the following properties:
\begin{itemize}
      \item $\forall i:  n _{i} \geq 1.$
      \item  $1 \in \image (j _{1}) $, $n \in
         \image (j
      _{s})$.  
\item $\forall 2 \leq i \leq n: \max _{[n _{i}-1] }
   j _{i-1} = \min _{[n _{i}] } j_{i}$ 
\end{itemize}
%

We denote the set of all
decompositions of $[n]$ by $D  [n]$.   We may of
course equally understand a length $s$
decomposition as a finite set  $\{J _{1}, \ldots J 
_{s}  \}$ of subsets of $[n]$ decomposing $[n]$. 
However, it is a bit simpler to formulate the following in terms of the maps $j _{i} $.
\begin{definition}
\label{def.nerve} For $  {A}$ a small unital $A _{\infty}$ category its
nerve $N ({A})$ is a simplicial set with the set of vertices  the set
of objects of ${A}$. A $n$-simplex $f$ of $N (  {A})$ consists of
the following data:
\begin{itemize}
   \item  A map $[n] \to Objects ({A})$. We denote the corresponding objects $X _{0}, \ldots, X _{n}$. 
\item For each mono-morphism $j: [n _{j}] \to [n]$
   in $\Delta$, with $n _{j} \geq 1$,
an element
\begin{equation*} f _{j} \in hom _{  {A}} (X _{j (0)}, X _{j
(n _{j})}).
\end{equation*} 
We may completely characterize each such $j $ by its image set, and will sometimes write $j $ for the corresponding set and vice
versa, thus $f _{ [n]} $ corresponds to the identity $j: [n] \to [n]$.
\item For a given $j: [n _{j}] \to  [n] $,
   and $i \in [n _{j}] $ 
   denote by $j - j(i): [n _{j} -1] \to [n]$ the
unique morphism in $\Delta$  with image set $j-j(i) $. Then the collection of these $f _{j}$ is required to satisfy the
following equation:
 \begin{equation} \label{eq.simplex} \mu ^{1} (f _{ j}) = \sum _{0 < i < n
 _{j}} f _{ j- j (i)} + \sum _{s \geq 2} \sum
    _{decomp _{s} \in D [n _{j}]} \mu ^{s} (f _{
       {j \circ j_1}},
 \ldots, f _{{j \circ j _{s}}}),
\end{equation}
with $decomp _{s} \in D [n_j]$ denoting a length $s$ decomposition and ${j_i} $, $1 \leq i \leq s$,  its elements.
This also corresponds to the discussion in
Section \ref{sec:OutlineAinftyNerve} describing
      the case of $2$-simplices of $N (A) $.   
\end{itemize} 
\end{definition}
The simplicial maps are as follows. Given an injection $k: [m] \to [n]$  and an
$n$ simplex $f $,  define an $m$-simplex $f'$ by
$\{f' _{j} = f _{k \circ j}\}$, where $j: [l] \to
[m] $ is an injection.

On the other hand, given the unique surjection in $\Delta
$:  $s _{i} : [n+1] \to [n]$, $s _{i} (i+1) = s _{i} (i)$,
and given an $n $-simplex $f $, define an $(n+1)$-simplex $f' $ by setting 
\[f' _{j} = \left \{ \begin{array}{ll}  
         e _{X_i} & \mbox{if $j= \{i, i+1\}$ };\\
        f _{s _{i}  \circ j} & \mbox{if $s _{i}|  _{j}$  is
        injective}. \\ 
        0 & \mbox {otherwise},
          \end{array} \right.
        \]
for $j: [l] \to [n+1] $ an injection. It is straightforward but tedious to
verify that the latter is indeed a face and that simplicial relations are
satisfied. On the other hand, Faonte
~\cite{citeFaonteSimplicialNerve} given a
conceptual construction of the above nerve so that
the above is automatic.
 \begin{proposition}
 \cite[2.3.2]{citeLeeAfunctorfromLagrangiancobordismstotheFukayacategory},
 \cite{citeFaonteSimplicialNerve}
\label{proposition.quasicat} For $ {A}$ a unital $A _{\infty}$ category its nerve $ \mathcal {A}= N ( 
A)$ is a $\infty$-category.
\end{proposition}
This proposition, has been central for us.  For the reader's convenience we outline the proof here. 
\begin{proof} Suppose we have an inner horn $\rho_k: \Lambda ^{n} _{k} \to
 NA$. In particular,  by the construction of the
   simplices of $NA$, corresponding to the faces of
   the horn, there are determined  $f _{j} \in hom (A) $,
   for all $j: [n _{j}] \to [n] $ except $j=
   [n] -\{k\} $ and $j= [n]$. Set $f _{[n]} =0$, and set
\begin{equation*}  f _{ [n] - \{k\}} =  \sum _{0 < i <
n; i \neq k} f _{ [n]- \{i \}} + \sum _{s \geq 2} \sum _{decomp _{s} \in D [n]} \mu ^{s} (f
_{ {j_1}}, \ldots, f _{ {j _{s}}}),
\end{equation*}
Only thing left to check is that as defined the
data $\{f _{j}\}$ determines  a $n$-simplex.

This amounts to verifying a pair of identities:
\begin{align*}
   0 & =  \sum _{0 < i < n} f _{
   [n] -i} + \sum _{s \geq 2} \sum _{decomp _{s} \in D
[n]} \mu ^{s} (f _{ {j_1}},
 \ldots, f _{{j _{s}}}),  \\
\mu ^{1} (f _{n-\{k\}}) & = \sum _{0 < i < n-1; i
   \neq k} f _{
   [n] -k -i} + \sum _{s \geq 2} \sum _{decomp
_{s} \in D ([n-1])} \mu ^{s} (f _{{j _{k} \circ j_1}},
 \ldots, f _{{j _{k} \circ j _{s}}}),
\end{align*} 
where $j _{k}: [n-1] \to [n]  $   is the inclusion
with image $[n] - \{k\}$.  

The first identity is immediate by the definition of $f
   _{[n] - \{k\}  }$.
For the second identity, a direct calculation is
   long but straightforward, using the $A
   _{\infty} $ associativity
equations. For $n=2$ this is automatic  and for
   $n=3 $ this is can be checked in
a few lines. However, for general $n$ it is certainly better to
prove this using conceptual methods as is done in
Faonte ~\cite{citeFaonteSimplicialNerve}. 
\end{proof}
For $F: A \to B$ an $A _{\infty} $ functor we define $NF: NA \to NB $ 
via the assignment:
\begin{equation*} f _{j} \mapsto \sum _{decomp_s \in D [n_j]} F ^{s} (f
_{j_1}, \ldots,  f_{j_s}).
\end{equation*}
\begin{lemma} \cite{citeLeeAfunctorfromLagrangiancobordismstotheFukayacategory},
 \cite{citeFaonteSimplicialNerve}
   The assignment $A \mapsto NA$, and $F \mapsto NF$ as above, determines a
functor $$N: A _{\infty} -Cat ^{unit}   \to \infty
   -\mathcal{C}at.$$
\end{lemma}
The details on why this constitutes a functor $N$
are omitted. 
\bibliographystyle{siam}    
\bibliography{C:/Users/yasha/texmf/bibtex/bib/link} 
%
\end{document}